\documentclass{amsart}
\usepackage[numbers,square]{natbib}
\usepackage{amsfonts,amsthm,amsmath,bbm,amssymb}
\usepackage{mathtools,verbatim,mathrsfs}
\usepackage{enumitem}
\usepackage{tikz,tikz-cd}
\usetikzlibrary{matrix, calc, arrows}

\usepackage{tabularx}

\usepackage{color,soul}

\usepackage{hyperref}

\numberwithin{equation}{section}

\theoremstyle{plain}
\newtheorem{theorem}{Theorem}[section]
\newtheorem{lemma}[theorem]{Lemma}

\newtheorem{corollary}[theorem]{Corollary}
\newtheorem{question}[theorem]{Question}

\newtheorem*{theorem*}{Theorem}

\newtheorem*{theoremintroa}{Theorem \ref{thm:charofdensityifsfcplus}}
\newtheorem*{theoremintroc}{Theorem \ref{thm:finitarymultsyndgivesaddthickinfields}}
\newtheorem*{theoremintrob}{Corollary \ref{cor:omegaismultsyndetic}}
\newtheorem*{theoremintrod}{Corollary \ref{cor:dioappomega}}
\newtheorem*{theoremintroe}{Theorem \ref{thm:relativeszeminsubsemigroupsofn} (Special Case)}
\newtheorem*{theoremintrof}{Corollary \ref{cor:geoarithmeticinfpx}}

\theoremstyle{definition}
\newtheorem{definition}[theorem]{Definition}
\newtheorem{remark}[theorem]{Remark}
\newtheorem{example}[theorem]{Example}
\newtheorem{examples}[theorem]{Examples}

\newcommand{\defeq}{\vcentcolon=}

\newcommand{\sfc}{\text{(SFC)}}
\newcommand{\sfcplus}{\text{(SFC+)}}
\newcommand{\crich}{combinatorially rich}

\newcommand{\al}{\alpha}
\newcommand{\ga}{\gamma}
\newcommand{\be}{\beta}
\newcommand{\eps}{\epsilon}

\newcommand{\denlet}{\delta}
\newcommand{\irratnum}{\xi}

\newcommand{\R}{\mathbb{R}}
\newcommand{\T}{\mathbb{T}}
\newcommand{\Q}{\mathbb{Q}}
\newcommand{\Z}{\mathbb{Z}}
\newcommand{\N}{\mathbb{N}}
\newcommand{\NPsi}{\N_\Psi}
\newcommand{\F}{\mathbb{F}}
\newcommand{\one}{\mathbbm{1}}

\newcommand{\multif}{\mathfrak{F}}
\newcommand{\underf}{F}
\newcommand{\multih}{\mathfrak{H}}

\newcommand{\arbclass}{\mathcal{X}}

\newcommand{\syndetic}{\mathcal{S}}
\newcommand{\thick}{\mathcal{T}}
\newcommand{\pws}{\mathcal{PS}}

\newcommand{\central}{\mathcal{C}}
\newcommand{\density}{\mathcal{D}}

\newcommand{\crichclass}{\mathcal{CR}}

\newcommand{\ip}{\text{IP}}
\newcommand{\ipr}{\text{IP}_{r}}
\newcommand{\ipnaught}{\ip_{\text{0}}}

\newcommand{\ipclass}{\mathcal{IP}}
\newcommand{\iprclass}{\mathcal{IP}_{r}}
\newcommand{\ipnaughtclass}{\ipclass_{\text{0}}}

\newcommand{\ipset}{\text{IP}}

\newcommand{\finiteproducts}{\text{FP}}
\newcommand{\finitesums}{\text{FS}}

\newcommand{\folner}{F{\o}lner}
\newcommand{\dstar}{d^*}
\newcommand{\sfcdstar}{d_{\textsc{F{\o}}}^*}

\newcommand{\tone}{\cdot}

\newcommand\restr[2]{{ \left.\kern-\nulldelimiterspace #1 \right|_{#2}}}

\newcommand{\mean}{\lambda}

\newcommand{\subsets}[1]{\mathcal{P}({#1})}
\newcommand{\finitesubsets}[1]{\mathcal{P}_f({#1})}

\newcommand{\seminplus}{(\N,+)}
\newcommand{\semintimes}{(\N,\cdot)}
\newcommand{\seminring}{(\N,+,\cdot)}

\newcommand{\semizplus}{(\Z,+)}

\newcommand{\semizring}{(\Z,+,\cdot)}

\newcommand{\semisplus}{(S,+)}
\newcommand{\semistimes}{(S,\cdot)}
\newcommand{\semisring}{(S,+,\cdot)}

\newcommand{\semittimes}{(T,\cdot)}

\newcommand{\semirtimes}{(R,\cdot)}
\newcommand{\semirplus}{(R,+)}

\newcommand{\semigplus}{(G,+)}
\newcommand{\semihplus}{(H,+)}

\newcommand{\matdz}{M_d(\Z)}

\newcommand{\gldq}{GL_d(\Q)}

\newcommand{\matM}{\mathbf{M}}
\newcommand{\matm}{M}

\newcommand{\fstar}{\F^*}

\newcommand{\multevenn}{E_{\N}}

\newcommand{\multoddn}{O_{\N}}

\newcommand{\Endset}{\mathcal{E}}
\newcommand{\Endelt}{\psi}

\newcommand{\normlq}{N_{L/\Q}}
\newcommand{\intringl}{\mathcal{O}_L}

\author[V. Bergelson]{Vitaly Bergelson}
\thanks{The first author gratefully acknowledges the support of the NSF under grant DMS-1500575.}
\address{Department of Mathematics, The Ohio State University, 231 W. 18th Ave., Columbus, OH 43210}
\email{vitaly@math.osu.edu}
\author[D.\ Glasscock]{Daniel Glasscock}
\address{Mathematical Sciences Department \\ University of Massachusetts Lowell\\
Lowell, MA, USA}
\email{daniel\textunderscore glasscock@uml.edu}

\title[Interplay between additive and multiplicative largeness]{on the interplay between additive and multiplicative largeness and its combinatorial applications}

\begin{document}

\begin{abstract}
Many natural notions of additive and multiplicative largeness arise from results in Ramsey theory. In this paper, we explain the relationships between these notions for subsets of $\N$ and in more general ring-theoretic structures. We show that multiplicative largeness begets additive largeness in three ways and give a collection of examples demonstrating the optimality of these results.  We also give a variety of applications arising from the connection between additive and multiplicative largeness. For example, we show that given any $n, k \in \N$, any finite set with fewer than $n$ elements in a sufficiently large finite field can be translated so that each of its elements becomes a non-zero $k^{\text{th}}$ power. We also prove a theorem concerning Diophantine approximation along multiplicatively syndetic subsets of $\N$ and a theorem showing that subsets of positive upper Banach density in certain multiplicative sub-semigroups of $\N$ of zero density contain arbitrarily long arithmetic progressions.  Along the way, we develop a new characterization of upper Banach density in a wide class of amenable semigroups and make explicit the uniformity in recurrence theorems from measure theoretic and topological dynamics. This in turn leads to strengthened forms of classical theorems of Szemer\'edi and van der Waerden on arithmetic progressions.
\end{abstract}

\maketitle

\section{Introduction}\label{sec:intro} 

\subsection{Background} 

A classic result of van der Waerden \cite{vdw} states that at least one cell of any finite partition of $\N = \{1, 2, 3, \ldots\}$ contains arbitrarily long arithmetic progressions. A far-reaching generalization by Rado \cite{rado} characterizes those systems of homogeneous linear equations to which at least one cell of any finite partition of $\N$ contains a solution. Underlying both of these foundational results in Ramsey theory are notions of ``additive largeness,'' and for many of these notions of largeness there are two essential theorems: \emph{at least one cell of any partition of $\N$ is additively large}, and \emph{additively large sets contain the sought-after combinatorial configurations}.

The pertinent notions of largeness underlying van der Waerden's theorem are syndeticity, piecewise syndeticity, and additive upper Banach density. In Rado's theorem, the notions of IP structure and centrality play a fundamental role. We will be concerned with these notions in this work.\footnote{The reader is encouraged to consult the index on page \pageref{indextable} upon encountering unfamiliar symbols or terminology.} To facilitate the discussion in the introduction, we define them now for subsets $A \subseteq \N$.
\begin{itemize}
\item $A$ is \emph{syndetic} if there exists a finite set $F \subseteq \N$ for which
\[\bigcup_{f \in F} \left(A - f\right) = \N, \quad \text{where } A - f \defeq \big\{n \in \N \ \big| \ n + f \in A \big\}.\]
\item $A$ is \emph{piecewise syndetic} if there exists a syndetic set $S \subseteq \N$ and a sequence $(m_n)_{n \in \N} \subseteq \N$ for which
\[S \cap \bigcup_{n=1}^\infty \big\{m_n+1,\ldots,m_n+n\big\} \subseteq A.\]
\item $A$ has positive \emph{additive upper Banach density} if
\begin{align}\label{eqn:densityinseminplusintro}
\begin{aligned}\limsup_{N - M \to \infty} \frac{\big |A \cap \{N, \ldots, N+M-1\} \big|}{N-M} &\defeq \\
\limsup_{n \to \infty} \ \max_{m \in \N} \ \frac{\big |A \cap (m+\{1, \ldots, n\}) \big|}{n} &> 0.\end{aligned}
\end{align}
\item $A$ is \emph{AP-rich} if it contains arbitrarily long arithmetic progressions $\{x,\allowbreak x+y,\allowbreak x+2y,\allowbreak \ldots,\allowbreak x+ky\}$.
\item $A$ is an \emph{IP set} if it contains a set of the form
\begin{align*}
\text{FS}(n_i)_{i \in \N} = \left\{ \sum_{i \in I } n_i \ \middle| \ \text{finite, non-empty } I \subseteq \N \right\}, \quad (n_i)_{i \in \N} \subseteq \N.
\end{align*}
\item $A$ is \emph{central}\footnote{Central sets were introduced by Furstenberg in \cite{furstenberg-book} in the language of topological dynamics. They play an essential role in Ramsey theory, forming one of the most natural and rich classes of largeness. See Definition \ref{def:defofcentral} and the discussion surrounding it for more details.} if it is a member of an additively minimal idempotent ultrafilter in $(\beta \N,+)$.\footnote{Ultrafilters capture many important notions of largeness via algebra in $\beta \N$, the Stone-\v{C}ech compactification of $\N$, but we largely avoid their use in this paper; see \cite{bhnonmetrizable}, \cite{hindmanstrauss-book}, and \cite{hindmanstrassalgebra}. For readers familiar with $\beta \N$, we point out that $A \subseteq \N$ is syndetic if and only if $\emph{cl}_{\beta \N}(A)$ has non-empty intersection with every left ideal in $\beta \N$; is \emph{thick}, that is, contains arbitrarily long intervals, if and only if $\emph{cl}_{\beta \N}(A)$ contains a left ideal of $\beta \N$; is piecewise syndetic if and only if $\emph{cl}_{\beta \N}(A)$ has non-empty intersection with the smallest ideal in $\beta \N$; and is an IP set if and only if $\emph{cl}_{\beta \N}(A)$ contains an idempotent.}
\end{itemize}

Many fundamental results in Ramsey theory can be interpreted as elucidating the relationships between these notions of largeness. For example, that syndetic sets are AP-rich is a consequence of van der Waerden's theorem. Hindman's theorem \cite{hindmanoriginal} implies that if $A$ piecewise syndetic, then $A-n$ is an IP set for some $n \in \N$, and it follows by a remark of Furstenberg \cite[Page 163]{furstenberg-book} that $A-n$ is central for some $n \in \N$.\footnote{More precisely, if $A$ is piecewise syndetic, then there exists a finite set $F \subseteq \N$ for which $\cup_{f \in F} (A - f)$ is thick, that is, contains arbitrarily long intervals. Thick sets are both IP and central. Hindman's theorem implies that one cell of any finite partition of an IP set contains an IP set. It is also true that one cell of any finite partition of a central set is central; see Remark \ref{rmk:partitionregularityofipnaughtandcentral}.} Szemer\'edi's theorem \cite{szemeredi} gives that sets of positive upper Banach density are AP-rich, and Furstenberg and Katznelson's IP Szemer\'edi theorem \cite{furstenbergkatznelsonipszem} (Theorem \ref{thm:ipszemeredicomb} below) gives that in any set $A$ of positive upper Banach density, for every $\ell \in \N$, the set of step sizes of arithmetic progressions of length $\ell$ appearing in $A$ has non-empty intersection with every IP set.

The set of positive integers supports another associative and commutative operation, multiplication, which is just as fundamental to its structure. For each of the notions of additive largeness defined above, there is a natural analogous notion of multiplicative largeness for $A \subseteq \N$. Fix $p_1, p_2, \ldots$ an enumeration of the prime numbers.
\begin{itemize}
\item $A$ is \emph{multiplicatively syndetic} if there exists a finite set $F \subseteq \N$ for which
\[\bigcup_{f \in F} \left(A \big/ f\right)  = \N, \quad \text{where } A \big/ f \defeq \big\{n \in \N \ \big| \ nf \in A \big\}.\]
\item $A$ is \emph{multiplicatively piecewise syndetic} if there exists a multiplicatively syndetic set $S \subseteq \N$ and a sequence $(m_n)_{n \in \N} \subseteq \N$ for which
\[S \cap \bigcup_{n=1}^\infty \big\{m_n p_1^{e_1} \cdots p_{n}^{e_{n}} \ | \ e_1, \ldots, e_{n} \in \{1, \ldots, n\} \big\} \subseteq A.\]
\item $A$ has positive \emph{multiplicative upper Banach density}\footnote{We will show in Section 3 that the multiplicative upper Banach density does not depend on the specific choice of multiplicative \emph{\folner{} sequence} (see Definition \ref{def:sfcandfolner}) in the definition given here.} if
\begin{align}\label{eqn:densityinseminplusintromult}
\limsup_{n \to \infty} \ \max_{m \in \N} \ \frac{|A \cap \{m p_1^{e_1} \cdots p_n^{e_n} \ | \ e_1, \ldots, e_{n} \in \{1, \ldots, n\} \}|}{n^n} > 0.
\end{align}
\item $A$ is \emph{GP-rich} if it contains arbitrarily long geometric progressions $\{x,\allowbreak xy,\allowbreak xy^2,\allowbreak \ldots,\allowbreak xy^k\}$.
\item $A$ is a \emph{multiplicative IP set} if it contains a set of the form
\begin{align*}
\text{FP}(n_i)_{i \in \N} = \left\{ \prod_{i \in I } n_i \ \middle| \ \text{finite, non-empty } I \subseteq \N \right\}, \quad (n_i)_{i \in \N} \subseteq \N.
\end{align*}
\item $A$ is \emph{multiplicatively central} if it is a member of a multiplicatively minimal idempotent ultrafilter in $(\beta \N,\cdot)$.
\end{itemize}

Multiplicative analogues of the Ramsey theoretical results described above explain some of the relationships between these notions of multiplicative largeness. For example, a multiplicative analogue of Hindman's theorem \cite[Lemma 2.1]{bergelsonhindman-elementary} gives that if $A$ multiplicatively piecewise syndetic, then $A / n$ is a multiplicative IP set for some $n \in \N$, and an analogue of Szemer\'edi's theorem gives that sets of positive multiplicative upper Banach density are GP-rich (c.f. \cite[Theorem 3.8]{bmultlarge}).

With addition and multiplication occupying a central role in the study of the positive integers, it is incumbent on us to explain the relationships between these notions of additive and multiplicative largeness. It has been known for some time that additive and multiplicative largeness can be simultaneously guaranteed in one cell of any partition of $\N$; this was first explained by Hindman \cite{hindmanmultaddip} and later developed further in \cite{bergelsonhindman-elementary}. Such partition results leave open the question as to how these notions of additive and multiplicative largeness interact: does one beget the other, or are the notions in general position and hence guaranteed to co-exist with potentially no relation?

A natural first question exploring the direct relationship between additive and multiplicative largeness is: to what extent can a set of positive integers be additively large but multiplicatively small? The set $(4\N -2) \cup \text{FS}\big(2^{2^i} \big)_{i \in \N}$ is additively large (both syndetic and an IP set) but multiplicatively small (neither of positive multiplicative density nor a multiplicative IP set). In searching for a complementary example -- that is, a set which is multiplicatively large but additively small -- one encounters an asymmetry between the relationships amongst notions of additive and multiplicative largeness: \emph{multiplicatively large sets cannot be additively very small}. Several instances of this principle have appeared in the literature, and we cite three of them here.\footnote{Those readers unfamiliar with the notion of centrality should note in the following results that additively central sets are additively piecewise syndetic and additively $\ip$, while \emph{multiplicatively thick} sets (sets $A \subseteq \N$ for which for all finite $F \subseteq \N$, there exists $n \in \N$ so that $nF \subseteq A$) are multiplicatively central. These relationships are explained in Section \ref{sec:defs} and depicted in that section in Figure \ref{fig:containmentdiagram}.}

\vskip.2cm
\noindent\quad \textbf{Theorem A} \ (\cite[Lemma 5.11]{BergelsonSurveytwoten}) \ \emph{For all $A \subseteq \N$,}
\[\text{\emph{if $A$ is multiplicatively syndetic, then $A$ is additively central.}}\]
\quad \textbf{Theorem B} \ (\cite[Theorem 3.5]{berghind-onipsets}, \cite[Proposition 4.1]{BRcountablefields}) \ \emph{For all $A \subseteq \N$,}
\[\text{\emph{if $A$ is multiplicatively central, then $A$ is additively $\ipnaught$.}}\footnote{A set $A \subseteq \N$ is \emph{additively $\ipnaught$} if for all $r \in \N$, there exist $n_1, \ldots, n_r \in \N$ such that for all non-empty $I \subseteq \{1, \ldots r\}$, $\sum_{i \in I} n_i \in A$.}\]
\quad \textbf{Theorem C} \ (\cite[Theorems 3.2 \& 3.15]{bmultlarge}) \ \emph{For all $A \subseteq \N$,}
\begin{gather*}
\text{\emph{if $A$ has positive multiplicative upper Banach density, then $A$ is AP-rich;}}\\
\text{\emph{what is more, $A$ contains arbitrarily large ``geo-arithmetic'' patterns.}}\footnotemark
\end{gather*}
\footnotetext{For example, for any $n \in \N$, there exist $a, c, d \in \N$ for which $\big\{ c (a+id)^j \ \big| \ 1 \leq i, j \leq n \big\} \subseteq A$. We will elaborate on such configurations in Section \ref{sec:geoarithmeticprogressions}.}Thus, while an additively syndetic IP set need not have positive multiplicative density nor be a multiplicative IP set, it follows from Theorem A that a multiplicatively syndetic set necessarily has positive additive upper Banach density and is an additive IP set.

In spite of these results, the relationships between notions of additive and multiplicative largeness are not entirely understood. Attempting to better understand these relationships quickly leads to interesting open problems. For example, it is unknown whether or not additively syndetic subsets of $\N$ are necessarily GP-rich. While we only discuss this particular problem in passing (see Section \ref{sec:negexamplea}), we advance our understanding of these relationships in this paper and obtain new and interesting applications in the theory of amenability and invariant means on semigroups, Diophantine approximation, and combinatorial number theory.

\subsection{Results}

In this paper, we wish to explore in the proper generality the interplay between notions of additive and multiplicative largeness. The natural numbers $\N$ with addition and multiplication form a \emph{semiring}: a set $S$ supporting a commutative additive semigroup $\semisplus$ and a multiplicative semigroup $\semistimes$ in which distributivity holds: $s (s_1+s_2) = s s_1 + s s_2$, and similarly on the right. Semirings are minimally structured algebraic objects in which addition and multiplication co-exist, and it is in them that we study the relationships between additive and multiplicative largeness.

This work is comprised of two parts. First, we generalize Theorems A, B, and C to the context of semirings, in the process strengthening Theorem B for subsets of $\N$. Thus, we prove the following theorems for semirings $\semisring$ that contain a suitable multiplicative subsemigroup $(R,\cdot)$; the requisite definitions can be found in Sections 2, 4, and 7, and the meaning of ``suitable'' can be found in the precise statements of the theorems in the sections to which they belong.
\vskip.2cm
\noindent\quad \textbf{Theorem \ref{thm:maintheorema}} \ \emph{For all $A \subseteq R$,}
\[\text{\emph{if $A$ is multiplicatively syndetic, then $A$ is additively central.}}\]
\quad \textbf{Theorem \ref{thm:multpwsimpliesaddipnaught}} \ \emph{For all $A \subseteq R$,}
\[\text{\emph{if $A$ is multiplicatively piecewise syndetic, then $A$ is additively $\ipnaught$.}}\footnotemark\]
\quad \textbf{Theorems \ref{thm:multdensitygivesaddpatterns} \& \ref{thm:geopatternsinsetsofdensity}} \ \emph{For all $A \subseteq R$,}
\begin{gather*}\text{\emph{if $A$ has positive multiplicative upper Banach density, then $A$ is}} \\ \text{\emph{additively combinatorially rich; what is more, $A$ contains}} \\ \text{\emph{arbitrarily large ``geo-arithmetic'' patterns.}} \end{gather*}
\footnotetext{Theorem \ref{thm:multpwsimpliesaddipnaught} strengthens Theorem B when $S = \N$ because multiplicatively central sets are multiplicatively piecewise syndetic. This relationship is explained in Section \ref{sec:defs} and depicted in that section in Figure \ref{fig:containmentdiagram}.}In Section \ref{sec:counterexamples}, we provide extremal examples that serve to illustrate the optimality of these results.\\

Useful throughout the paper is a new characterization of upper Banach density, developed in Section \ref{sec:banach}, that is applicable in a wide class of amenable semigroups. The upper Banach densities in (\ref{eqn:densityinseminplusintro}) and (\ref{eqn:densityinseminplusintromult}) are defined with the help of \emph{\folner{} sequences} (see Definition \ref{def:sfcandfolner}); in a general left amenable semigroup $\semistimes$, we define the upper Banach density of $A \subseteq S$ to be
\[d^*(A) = \sup \big\{ \lambda(\one_A) \ \big| \ \lambda \text{ a left translation invariant mean on } (S,\cdot) \big\},\]
(see Definition \ref{def:defofamenable}) and we prove the following result.
\begin{theoremintroa}
Let $\semistimes$ be an amenable group or a commutative semigroup. For all $A \subseteq S$,
\[\dstar(A) = \sup \left\{ \denlet \geq 0 \ \middle| \ \forall \text{ finite } F \subseteq S, \ \exists \ s \in S, \ |F \cap A \cdot s^{-1}| \geq \denlet |F| \right\},\]
where $A \cdot s^{-1} = \{t \in S \ | \ t \cdot s \in A \}$.
\end{theoremintroa}

This characterization of upper Banach density not only serves as an important tool throughout the paper, it allows us to improve on recent results regarding densities defined along \folner{} sequences obtained in \cite[Section 2]{hindmanstrauss-semigroupdensity2}.\\

The main impetus for this work arose not from results in the abstract semiring framework but from the diverse array of related results and applications afforded to us by the ideas behind these theorems. We proceed now with a sampling of these related results and representative applications of the theorems above. The focus in each of the following results is on the interplay between addition and multiplication.\\

In Section \ref{sec:posexamplea}, we combine Theorem \ref{thm:maintheorema} and an $\ip^*$ version of Szemer\'edi's theorem to prove that cosets of finite index multiplicative subgroups in infinite division rings are additively thick (cf. \cite[Remark 5.23]{Bcombanddioph}).  The following finitary analogue of that result is also achieved using the ideas behind the proof of Theorem \ref{thm:maintheorema}.

\begin{theoremintroc}
For all $n, k \in \N$, there exists an $N=N(n,k) \in \N$ with the following property. For all finite fields $\F$ with $|\F| \geq N$ and all $F \subseteq \F$ with $|F| \leq n$, there exists a non-zero $x \in \F$ for which every element of $x^k + F$ is a non-zero $k^{\text{th}}$ power.
\end{theoremintroc}

As another application of Theorem \ref{thm:maintheorema}, we exhibit a family of additively large sets which are defined with the help of familiar multiplicative functions from number theory. Let $\nu_2$, $\Omega: \N \to \N \cup \{0\}$ denote the 2-adic valuation and the prime divisor counting function, defined by $\nu_2(2^{e_1}p_2^{e_2} \cdots p_k^{e_k}) = e_1$ and $\Omega(p_1^{e_1} \cdots p_k^{e_k}) = \sum_{i=1}^k e_i$, where the $p_i$'s are distinct primes. Denote by $\{ x\}$ the fractional part of $x \in \R$.

\begin{theoremintrob}
Let $p \in \R[x]$ be a non-constant polynomial with irrational leading coefficient, and let $I \subseteq [0,1)$ be an interval. The sets
\begin{align}\label{eqn:multdefndaddlargesets}\big\{n \in \N \ \big| \ \{p(\nu_2(n)) \} \in I \big\} \text{ and } \big\{n \in \N \ \big| \ \{p(\Omega(n)) \} \in I \big\}\end{align}
are multiplicatively syndetic in $\semintimes$, hence additively central in $\seminplus$.
\end{theoremintrob}

This result complements known results about the uniform distribution of sequences of the form $(\{g(n) \})_{n \in \N}$ and $(\{p(g(n)) \})_{n \in \N}$ where $p$ is a polynomial and $g: \N \to \N$ is an arithmetic function arising in number theory; see, for example, \cite{delange,koninckkatai}. Since additively central sets contain solutions to partition regular systems of homogeneous linear equations (see \cite[Chapter 8]{furstenberg-book} or \cite[Chapter 14]{hindmanstrauss-book}), another consequence of Corollary \ref{cor:omegaismultsyndetic} is that the sets in (\ref{eqn:multdefndaddlargesets}) contain solutions to such systems. Corollary \ref{cor:dioappomega} exhibits yet another application.\\

In Section \ref{sec:posexampleb}, we explore the main applications of a corollary of Theorem \ref{thm:multpwsimpliesaddipnaught}: \emph{additive $\ipnaught^*$ sets -- those sets which have non-empty intersection with all $\ipnaught$ sets -- have non-empty intersection with all multiplicatively piecewise syndetic sets.} Additive $\ipnaught^*$ sets arise naturally in recurrence theorems in measure theoretic and topological dynamics to describe the set of ``return times'' of sets of positive measure and open sets, and thus Theorem \ref{thm:multpwsimpliesaddipnaught} enhances theorems in which these classes appear.

We demonstrate in Section \ref{sec:posexampleb} how to derive $\ipnaught^*$ versions of two prominent examples of such recurrence theorems: Furstenberg and Katznelson's IP Szemer\'edi theorem \cite{furstenbergkatznelsonipszem} and the polynomial IP van der Waerden theorem \cite{bergelsonleibmanams}. Then, we find a number of applications making use of the enhanced statements of these results: Theorem \ref{thm:finitarymultsyndgivesaddthickinfields} and Corollary \ref{cor:geoarithmeticinfpx} are corollaries of the $\ip_0^*$ version of the IP Szemer\'edi theorem, and the following corollary in Diophantine approximation follows from the $\ip_0^*$ version of the polynomial IP van der Waerden theorem. Denote by $\| x\|$ the distance from $x$ to the nearest integer.

\begin{theoremintrod}
Let $p \in \R[x]$ be a non-constant polynomial with irrational leading coefficient, and let $I \subseteq [0,1)$ be an interval. For all $d \in \N$,
\[\min_{\substack{1 \leq n \leq N \\ \{ p(\Omega(n)) \} \in I}} \|f(n) \| \longrightarrow 0 \ \text{ as } \  N \to \infty\]
uniformly in polynomials $f \in \R[x]$ of degree less than or equal to $d$ with no constant term.\\
\end{theoremintrod}

In Section \ref{sec:posexamplec}, we prove that subsets of a semiring that have positive multiplicative density in an additively \crich{} multiplicative sub-semigroup are additively \crich{}. This conclusion is particularly novel in the case that the multiplicative sub-semigroup has zero additive upper Banach density since this rules out the case that the additive combinatorial richness of the subset arises from it having positive additive upper Banach density.  As an example of an application, we use the theorem of Green and Tao \cite{greentaooriginal} to prove in Theorem \ref{thm:relativeszeminsubsemigroupsofn} that any subset of positive multiplicative density in the multiplicative sub-semigroup of positive integers that appear as norms of algebraic integers in a number field contains arbitrarily long arithmetic progressions.  The following is a special case of this result that makes use of the ring of integers in $\Q(\sqrt[3]{2})$; note that the multiplicative sub-semigroup $R$ has zero additive upper Banach density \cite[Theorem T]{odonireps}.

\begin{theoremintroe}
Let
\[R = \big\{ |x_1^3 + 2x_2^3 + 4x_3^3 - 6x_1x_2x_3 | \ \big| \ x_1, x_2, x_3 \in \Z \big\} \setminus \{0\}.\]
The set $(R,\cdot)$ is a sub-semigroup of $\semintimes$, and any set $A \subseteq R$ satisfying $\dstar_{(R,\cdot)}(A) > 0$ contains arbitrarily long arithmetic progressions.
\end{theoremintroe}

We also show that subsets of positive multiplicative density in positive additive density sub-semigroups of a semigroup contain geo-arithmetic patterns. The following is an example application which, in the $k=1$ case and considering additive upper Banach density in place of multiplicative upper Banach density, is a consequence of Szemer\'edi's theorem in finite characteristic, \cite[Theorem 9.10]{furstenbergkatznelsonipszem}.

\begin{theoremintrof}
Let $p \in\N$ be prime, and suppose $A \subseteq \F_p[x] \setminus \{0\}$ has positive multiplicative upper Banach density in $(\F_p[x] \setminus \{0\},\cdot)$. For all $d,n \in \N$, there exist $y,z \in \F_p[x]$, $z \neq 0$, and a $d$-dimensional vector subspace $V$ of $\F_p[x]$ such that
\[\bigcup_{v \in V} z \{ y + v, (y + v)^2, \cdots, (y + v)^k \big\} \subseteq A.\]
\end{theoremintrof}

This paper is organized as follows. In Section \ref{sec:defs} we define several notions of largeness for subsets of semigroups and discuss the hierarchy between them. The alternate characterization of upper Banach density is proved in Section \ref{sec:banach}. In Section \ref{semiringsandexamples} we define semirings and describe how notions of largeness behave under semigroup homomorphisms. Section \ref{sec:preamble} is a preamble to Sections \ref{sec:posexamplea}, \ref{sec:posexampleb}, and \ref{sec:posexamplec}, where we prove Theorems \ref{thm:maintheorema}, \ref{thm:multpwsimpliesaddipnaught}, and \ref{thm:multdensitygivesaddpatterns}, respectively, and prove the related results and applications discussed above. Finally, in Section \ref{sec:counterexamples}, we present the extremal examples that serve to show the optimality of the main theorems. Section \ref{sec:index} is an index of the most frequently used symbols in the work.

\subsection{Acknowledgements} The authors would like to thank Jim Cogdell and Warren Sinnott for their assistance with the material in Section \ref{sec:apinsubsemigroups} and Tim Browning for calling our attention to the work of R. W. K. Odoni. Thanks also goes to Neil Hindman for improving Lemma \ref{lem:bigandbigimpliesbig} and Corollary \ref{cor:bigandbigimpliesbig} and to Joel Moreira and Florian Richter for numerous helpful conversations and a close reading of the paper.  Last but not least, the authors give thanks to the referees for numerous helpful corrections and suggestions.


\section{Notions of largeness in semigroups}\label{sec:defs}  

In this section we define several notions of largeness for subsets of semigroups. The best general references are \cite[Chapter 9]{furstenberg-book} and \cite[Section 1]{BHabundant}, though we largely avoid the machinery of ultrafilters in this paper. While the definitions and results in this section make reference to only a single semigroup, we will apply this material in later sections to subsets of both the additive semigroup and the multiplicative semigroup of a semiring (see Definition \ref{def:semiring}).

Denote by $\N$ the set of natural numbers $\{1, 2, \ldots\}$. For $S$ a set, denote by $\finitesubsets S$ and $\subsets S$ the collections of all finite subsets (including the empty set) and all subsets of $S$, respectively. 

\begin{definition}\label{def:dualclasses}
Let $\arbclass \subseteq \subsets S$ be a collection of subsets of a set $S$.
\begin{enumerate}[label=(\Roman*)]
\item The class $\arbclass$ is \emph{partition regular} if for all $A \in \arbclass$ and all partitions $A = \cup_{i=1}^k A_i$ of $A$ into finitely many pieces, at least one $A_i$ is in $\arbclass$.
\item The class $\arbclass$ is a \emph{filter} if it is non-empty, $\emptyset \not\in \arbclass$, $\arbclass$ is upward closed, and for all $A, B \in \arbclass$, $A \cap B \in \arbclass$.
\item The \emph{dual class} $\arbclass^* \subseteq \subsets S$ is the collection of subsets of $S$ having non-empty intersection with every member of $\arbclass$; in other words, $A \in \arbclass^*$ if and only if for all $B \in \arbclass$, $A \cap B \neq \emptyset$.
\end{enumerate}
\end{definition}

\begin{remark}\label{rmk:setalgebra}
If a collection $\arbclass$ of subsets of a set $S$ is upward closed, then $(\arbclass^*)^* = \arbclass$. When $\arbclass$ does not contain $\emptyset$ and is upward closed, it is partition regular if and only if its dual $\arbclass^*$ is a filter. These facts follow from simple exercises with the definitions and will be used repeatedly throughout the work.
\end{remark}

For $\semistimes$ a semigroup written multiplicatively, we denote multiplication by juxtaposition. For $x \in S$ and $A \subseteq S$, let $x A$ denote $\{x a \ | \ a \in A \}$ and $x^{-1} A$ denote $\{s \in S \ | \ x s \in A \}$; the sets $A x$ and $A x^{-1}$ are defined analogously.  Note that $x x^{-1} A \subseteq A$ and that a strict inclusion is possible.

\begin{definition}\label{def:syndthick}
Let $\semistimes$ be a semigroup and $A \subseteq S$.
\begin{enumerate}[label=(\Roman*)]
\item $A$ is (right) \emph{syndetic} if there exist $s_1, \ldots, s_k \in S$ such that $S = s_1^{-1}A \cup \cdots \cup s_k^{-1}A$.
\item $A$ is (left) \emph{thick} if for all $F \in \finitesubsets S$, there exists $x \in S$ for which $F x \subseteq A$.
\item $A$ is (right) \emph{piecewise syndetic} if there exist $s_1, \ldots, s_k \in S$ such that $s_1^{-1}A \cup \cdots \cup s_k^{-1}A$ is (left) thick.
\end{enumerate}
Denote by $\syndetic \semistimes$, $\thick \semistimes$, and $\pws \semistimes$ the collections of all syndetic, thick, and piecewise syndetic subsets of the semigroup $\semistimes$. When the semigroup is apparent, we refer to these classes simply as $\syndetic$, $\thick$, and $\pws$.
\end{definition}

The choice of \emph{left} and \emph{right} in Definition \ref{def:syndthick} makes thickness ``dual'' to syndeticity in the sense that $\syndetic^* = \thick$ and $\thick^* = \syndetic$.  The opposite classes \emph{left syndetic}, \emph{right thick}, and \emph{left piecewise syndetic} may be defined analogously, and the opposite versions of the results appearing in this work hold for them.

We describe now another notion of largeness, $\ip$ structure, which is of fundamental importance in Ramsey theory and ergodic theory; see \cite{hindmanoriginal}, \cite{fwdynamics}, \cite{furstenbergkatznelsonipszem}, \cite{fkbulletin}, and \cite{bergelsonleibmanannals}.

\begin{definition}\label{def:ipstructure}
Let $\semistimes$ be a semigroup and $A \subseteq S$.
\begin{enumerate}[label=(\Roman*)]
\item \label{item:ipitemone} Given $(s_i)_{i=1}^r \subseteq S$ and a non-empty $\al = \{\al_1 < \cdots < \al_k\} \subseteq \{1, \ldots, r\}$, let
\[s_\al = s_{\al_1} s_{\al_2} \cdots s_{\al_k}.\]
\item $A$ is an \emph{$\ipr$ set}, $r \in \N$, if there exists $(s_i)_{i=1}^r \subseteq S$ for which
\begin{align*}\finiteproducts {(s_i)_{i=1}^r} = \big\{ s_\al \ \big| \ \emptyset \neq \al \subseteq \{1, \ldots, r\} \big\} \subseteq A.\end{align*}
\item $A$ is an \emph{$\ipnaught$ set} if for all $r \in \N$, it is an $\ipr$ set.
\item $A$ is an \emph{$\ip$ set} if there exists $(s_i)_{i\in \N} \subseteq S$ for which
\begin{align*}\finiteproducts {(s_i)_{i\in\N}} = \big\{ s_\al \ \big| \ \emptyset \neq \al \in \finitesubsets \N \big\} \subseteq A.\end{align*}
\end{enumerate}
We denote the class of $\ipr$, $\ipnaught$, and $\ip$ subsets of $\semistimes$ by $\iprclass \semistimes$, $\ipnaughtclass \semistimes$, and $\ipclass \semistimes$, respectively.
\end{definition}

In this definition, ``FP'' is short for ``finite products''; when the semigroup is written additively, we write ``FS,'' which abbreviates ``finite sums.'' The semigroup $\semistimes$ is not assumed to be commutative, so the order in which the products are taken in \ref{item:ipitemone} is important. The increasing order was chosen here so that every (left) thick set is an $\ipset$ set (see Lemma \ref{lem:hierarchylemma} below); decreasing $\ipset$ sets (those defined with a decreasing order) can be found in any right thick set.

The notion of \emph{centrality} combines piecewise syndeticity and $\ip$ structure. Central sets originated in $\N$ in a dynamical context with Furstenberg \cite[Definition 8.3]{furstenberg-book}, and it was revealed in \cite[Section 6]{bhnonmetrizable} that the property of being central is equivalent to membership in a minimal idempotent ultrafilter. See \cite{hindmanstraussmaleki-centralchar} for a combinatorial characterization of centrality, and, for more recent developments and applications, see \cite[Section 5]{hindmanstrassalgebra} and \cite{polyrado}.

\begin{definition}\label{def:defofcentral}
Let $\semistimes$ be a semigroup. A subset of $S$ is \emph{central} if it is a member of a minimal idempotent ultrafilter on $S$. We denote by $\central \semistimes$ the class of central subsets of $\semistimes$.
\end{definition}

\begin{remark}\label{rmk:partitionregularityofipnaughtandcentral}
The class $\central$ is partition regular. This fact was first remarked by Furstenberg \cite[Page 163]{furstenberg-book}, and it follows immediately from the ultrafilter characterization of centrality. The classes $\pws$, $\ipnaughtclass$, and $\ipclass$ are also partition regular. The first and third of these facts follow from the ultrafilter characterization of those classes (see \cite{hindmanstrauss-book}, Theorems 4.40 and 5.12), and a proof of the second fact may be found in \cite[Proposition 2.3]{BRcountablefields}. It follows from the set algebra discussed in Remark \ref{rmk:setalgebra} that the dual classes $\central^*$, $\pws^*$, $\ipnaughtclass^*$, and $\ipclass^*$ are filters.
\end{remark}

Density provides another natural notion of size for subsets of amenable semigroups.

\begin{definition}\label{def:defofamenable}
Let $\semistimes$ be a semigroup.
\begin{enumerate}[label=(\Roman*)]
\item $S$ is \emph{left amenable} if the space of bounded, real-valued functions on $S$ with the supremum norm admits a \emph{left translation invariant mean}, that is, a positive linear functional $\lambda$ of norm 1 which is left translation invariant: for all bounded $f: S \to \mathbb{R}$ and all $s \in S$, $\lambda \big( x \mapsto f(s x) \big) = \lambda ( f )$. When regarding means as finitely additive measures, we abuse notation slightly by writing $\lambda(A)$ instead of $\lambda(\one_A)$, where $\one_A$ is the indicator function of $A$, for subsets $A \subseteq S$.
\item \label{item:upperbdensity} The \emph{upper Banach density} of $A \subseteq S$ is
\[d^*(A) = \sup \big\{ \lambda(A) \ \big| \ \lambda \text{ a left translation invariant mean on } \semistimes \big\}.\]
\end{enumerate}
We denote by $\density \semistimes$ the class of subsets of $\semistimes$ of positive upper Banach density. (With the convention that the supremum of the empty set is $-\infty$, the class $\density$ is non-empty if and only if $\semistimes$ is left amenable.)
\end{definition}

The most familiar appearance of this notion of density is the additive upper Banach density for subsets $A \subseteq \N$ given by (\ref{eqn:densityinseminplusintro}) in the introduction. In addition to showing that this density coincides with the one in Definition \ref{def:defofamenable} \ref{item:upperbdensity}, the results in Section \ref{sec:banach} allow us to handle upper Banach density in more general semigroups. Two useful properties follow immediately from the definition: for all $A, B \subseteq S$ and all $s \in S$,
\begin{align}\label{eqn:densityproperties}\dstar(s^{-1}A) = \dstar(A) \quad \text{ and } \quad \dstar(A \cup B) \leq \dstar(A) + \dstar(B).\end{align}
The second of these properties implies that the class $\density$ is partition regular and, hence, that the class $\density^*$ is a filter.

Next we introduce a class of sets in commutative semigroups called \emph{\crich{}} sets which contain an abundance of combinatorial patterns. Roughly speaking, a set is combinatorially rich if it satisfies the conclusions of Theorem \ref{thm:ipszemeredicomb}, the $\ip_0^*$ version of Furstenberg and Katznelson's IP Szemer\'edi theorem.\footnote{The class of combinatorially rich sets is related to the class of \emph{J-sets} (\cite[Definition 14.8.1]{hindmanstrauss-book}); those are, roughly speaking, sets which satisfy the conclusions of the IP van der Waerden theorem \cite[Theorem 3.2]{fwdynamics}.}

For $n, r \in \N$, denote by $S^{r \times n}$ the set of $r$-by-$n$ matrices with elements in $S$. For $\matM = (M_{i,j}) \in S^{r \times n}$ and non-empty $\al \subseteq \{1, \ldots r\}$, denote by $\matm_{\al,j}$ the sum $\sum_{i \in \al} M_{i,j}$.

\begin{definition}\label{def:combrich}
Let $\semisplus$ be a commutative semigroup. A subset $A \subseteq S$ is \emph{\crich{}} if for all $n \in \N$, there exists an $r \in \N$ such that for all $\matM \in S^{r\times n}$, there exists a non-empty $\al \subseteq \{1, \ldots, r\}$ and $s \in S$ such that for all $j \in \{1, \ldots, n\}$,
\[s+ \matm_{\al,j} \in A.\]
We denote by $\crichclass \semisplus$ the class of \crich{} subsets of $\semisplus$.
\end{definition}

Every set of positive upper Banach density is combinatorially rich; we will prove this in Theorem \ref{thm:densityimpliesar}. The converse does not hold: there are combinatorially rich subsets of $\N$ of zero density.\footnote{To construct a subset of $\seminplus$ that is combinatorially rich but of zero upper Banach density, show first that for all $n \in \N$, there exists $r \in \N$ such that for all $\matM \in \N^{r\times n}$, there exists a non-empty $\al \subseteq \{1, \ldots, r\}$ for which the set $\big\{\matm_{\al,j} \big\}_{j=1}^n$ is $n$-separated. Using this, one can construct a combinatorially rich set of zero density by taking a countable union of sufficiently separated finite sets.} In the following subsection, we will give an equivalent and perhaps more friendly reformulation of Definition \ref{def:combrich}, give further examples of combinatorially rich sets in Examples \ref{examples:combcubes}, and say much more about the class of combinatorially rich sets.

Now we will explain the hierarchy between the notions of largeness just defined.

\begin{lemma}\label{lem:hierarchylemma}
Let $\semistimes$ be a semigroup and $r \in \N$. Figure \ref{fig:containmentdiagram} illustrates containment amongst classes of largeness in $\semistimes$ ($\mathcal{X} \to \mathcal{Y}$ indicates that $\mathcal{X} \subseteq \mathcal{Y}$), with the position of the classes $\density$ and $\density^*$ regarded only in the case that $\semistimes$ is left amenable and the position of the classes $\crichclass$ and $\crichclass^*$ regarded only in the case that $\semistimes$ is commutative.
\end{lemma}

\begin{figure}[ht]
\centering
\begin{tikzpicture}[>=triangle 60]
  \matrix[matrix of math nodes,column sep={15pt,between origins},row sep={12pt,between origins}](m)
  {
    & & & & & & & & |[name=arichstar]|\crichclass^* & & & & & & \\
    \\
    |[name=iprstar]|\iprclass^* & & & & & & & & & & & & & & \\
    & & & & & & & & |[name=densitystar]|\density^* & & & & & & \\
    & & |[name=ipnaughtstar]|\ipnaughtclass^* \\
    \\
    & & & & |[name=ipstar]|\ipclass^* & & & & |[name=pwsstar]|\pws^* \\
    \\
    & & & & & & |[name=centralstar]|\central^* & & & & |[name=thick]| \thick \\
    \\
    & & & & |[name=syndetic]|\syndetic & & & & |[name=central]|\central \\
    \\
    & & & & & & |[name=pws]| \pws & & & & |[name=ip]|\ipclass\\
    \\
    & & & & & & & & & & & & |[name=ipnaught]|\ipnaughtclass \\
    & & & & & & |[name=density]|\density \\
    & & & & & & & & & & & & & & |[name=ipr]|\iprclass \\
    \\
    & & & & & & |[name=arich]|\crichclass \\
 };
 
   \draw[-angle 90] (arichstar) edge (densitystar)
   			(iprstar) edge (ipnaughtstar)
            (ipnaughtstar) edge (ipstar)
            (ipstar) edge (centralstar)
            (centralstar) edge (central)
            (central) edge (ip)
            (ip) edge (ipnaught)
            (ipnaught) edge (ipr)
            (densitystar) edge (pwsstar)
            (pwsstar) edge (thick)
            (pwsstar) edge (centralstar)
            (centralstar) edge (syndetic)
            (syndetic) edge (pws)
            (thick) edge (central)
            (central) edge (pws)
            (pws) edge (density)
            (density) edge (arich)
  ;
\end{tikzpicture}
\caption{Containment amongst classes of largeness in a semigroup.}
\label{fig:containmentdiagram}
\end{figure}

The containments indicated in Figure \ref{fig:containmentdiagram} which do not involve the classes $\density$ or $\crichclass$ are proved in \cite[Section 1]{BHabundant}. The following lemma shows that in left amenable semigroups, $\pws \subseteq \density$; by considering the dual classes, it follows that $\density^* \subseteq \pws^*$.

\begin{lemma}
Let $\semistimes$ be a left amenable semigroup. If $A \subseteq S$ is piecewise syndetic, then $\dstar(A) > 0$.
\end{lemma}

\begin{proof}
Let $A \subseteq S$ be piecewise syndetic. There exist $s_1, \ldots, s_k \in S$ for which the set $B \defeq s_1^{-1} A \cup \cdots \cup s_k^{-1} A$ is thick. By \cite[Prop. 1.21]{patersonbook}, there exists a left invariant mean $\lambda$ on $\semistimes$ for which $\lambda(B) = 1$. Since $\lambda$ is positive, linear, and invariant,
\[1 = \lambda(B) \leq \lambda \big( \one_{s_1^{-1} A} + \cdots + \one_{s_k^{-1} A} \big) = \lambda \big( s_1^{-1} A \big) + \cdots + \lambda \big( s_k^{-1} A \big) = k \lambda \big( A \big),\]
whereby $\dstar(A) \geq 1 \big/ k$.
\end{proof}

All that remains to be shown in Lemma \ref{lem:hierarchylemma} is that sets of positive upper Banach density in commutative semigroups are combinatorially rich. This is proven in Theorem \ref{thm:densityimpliesar} below with the help of an alternate characterization of upper Banach density developed in Section \ref{sec:banach}.

\subsection{Combinatorially rich sets and Hales-Jewett-type theorems}\label{sec:crsetsandhjtheorems}

The following lemma helps us give examples of combinatorial configurations within \crich{} sets, show that the class $\crichclass$ is partition regular, and connect these sets to sets which satisfy the conclusions of the $\ip_0^*$ version of Szemer\'edi's theorem, Theorem \ref{thm:ipszemeredicomb}.

\begin{lemma}\label{lem:equivtocombrich}
Let $\semisplus$ be a commutative semigroup and $A \subseteq S$. The following are equivalent:
\begin{enumerate}[label=(\Roman*)]
\item \label{item:crich} $A$ is \crich{}.
\item \label{item:multidcombrich} For all $n, m \in \N$, there exists $r=r(n,m) \in \N$ such that for all $\matM \in S^{r\times n}$, there exists non-empty, disjoint $\al_1, \ldots, \al_m \subseteq \{1, \ldots, r\}$ and $s \in S$ such that for all $j_1, \ldots, j_m \in \{1, \ldots, n\}$,
\[s+ \matm_{\al_1,j_1} + \matm_{\al_2,j_2} + \cdots + \matm_{\al_m,j_m} \in A.\]
\item \label{item:homdef} For all $n \in \N$, there exists $r \in \N$ such that for all commutative semigroups $T$ and all homomorphisms $\varphi_1, \ldots, \varphi_n$ from $T$ to $S$, the set
\begin{align}\label{eqn:setofgoodinputs} \left\{ t \in T \ \middle | \ \big(A - \varphi_1(t) \big) \cap \cdots \cap \big(A - \varphi_n(t) \big) \neq \emptyset \right\}\end{align}
is $\ip_r^*$ in $T$.\footnote{Note that $\big(A - \varphi_1(t) \big) \cap \cdots \cap \big(A - \varphi_n(t) \big) \neq \emptyset$ if and only if there exists $s \in S$ such that for all $i \in \{1,\ldots, n\}$, $s+\varphi_i(t) \in A$.  We will utilize this equivalence repeatedly usually without mentioning it explicitly.}
\end{enumerate}
\end{lemma}

\begin{proof}
To see that \ref{item:crich} implies \ref{item:multidcombrich}, suppose $A$ is \crich{}. Let $n, m \in \N$, put $n' = n^m$ and $r' = r(n')$ from Definition \ref{def:combrich} for the set $A$. Put $r = mr'$, and let $\matM \in S^{r \times n}$. Make $\matM' \in S^{r' \times n'}$ as follows: the columns of $\matM'$ are indexed by vectors $(j_1,\ldots,j_m) \in \{1, \ldots, n\}^m$, where for all $i \in \{1, \ldots r'\}$,
\[\matm'_{i,(j_1,\ldots,j_m)} = \matm_{i,j_1}+ \matm_{i+r',j_2} + \cdots+ \matm_{i+(m-1)r',j_m}.\]
Since $\matM' \in S^{r' \times n'}$ and $A$ is \crich{}, there exists $\al \subseteq \{1, \ldots, r'\}$ and $s \in S$ such that for all $(j_1,\ldots,j_m) \in \{1, \ldots, n\}^m$, $s + \matm'_{\al,(j_1,\ldots,j_m)} \in A$. Putting $\al_i = \al + (i-1)r'$, we see that $\al_1, \ldots, \al_m \subseteq \{1, \ldots, r\}$ are disjoint and for all $j_1, \ldots, j_m \in \{1, \ldots, n\}$,
\[s+ \matm_{\al_1,j_1} + \matm_{\al_2,j_2} + \cdots + \matm_{\al_m,j_m} \in A.\]

That \ref{item:multidcombrich} implies \ref{item:crich} is immediate. To see that \ref{item:crich} implies \ref{item:homdef}, let $n \in \N$ and take $r=r(n)$ from Definition \ref{def:combrich} for the set $A$. Let $T$ be a commutative semigroup and $\varphi_1, \ldots, \varphi_n$ be homomorphisms from $T$ to $S$. Let $t_1, \ldots, t_r \in T$, and put $\matM = \big( \varphi_j(t_i) \big) \in S^{r \times n}$. Since $A$ is \crich{}, there exists $\al \subseteq \{1, \ldots, r\}$ and $s \in S$ such that for all $j \in \{1, \ldots, n\}$, $s+\varphi_j(t_\al) \in A$. Since $t_1, \ldots, t_r$ were arbitrary, this shows that the set in (\ref{eqn:setofgoodinputs}) is $\ip_r^*$ in $T$.

To see that \ref{item:homdef} implies \ref{item:crich}, let $n \in \N$ and take $r=r(n)$ from \ref{item:homdef} for the set $A$. Let $\matM \in S^{r \times n}$. Put $T=S^n$ with coordinate-wise addition, and consider the coordinate projections $\pi_1, \ldots, \pi_n:S^n \to S$. For each $i \in \{1, \ldots, r\}$, let $t_i$ be the $i^{\text{th}}$ row of $\matM$. Since the set in (\ref{eqn:setofgoodinputs}) is $\ip_r^*$ in $S^n$, there exists an $\al \subseteq \{1,\ldots, r\}$ and $s \in S$ such that for all $j \in \{1, \ldots, n\}$, $s+\pi_j(t_\al) = s + \matm_{\al,j} \in A$. This shows that $A$ is \crich{}.
\end{proof}

We proceed by giving some examples of combinatorial configurations found in combinatorially rich sets in some common semigroups.

\begin{examples}\label{examples:combcubes} Let $\semisplus$ be a commutative semigroup, $A \subseteq S$ be \crich{}, $n, m \in \N$, and $r \in \N$ be sufficiently large. Specifying various matrices $\matM \in S^{r \times n}$ leads, by Definition \ref{def:combrich}, to combinatorial configurations in $S$.
\begin{enumerate}[label=(\Roman*)]
\item \label{item:additiveexampleofcrsets} Suppose $\semisplus = \seminplus$. Choosing $\matm_{i,j} = j$, there exists a non-empty $\al \subseteq \{1, \ldots, r\}$ and $s \in \N$ such that
\[s+ |\al|, s+2|\al|, \ldots, s+ n|\al| \in A.\]
Thus, \crich{} sets in $\seminplus$ are AP-rich: they contain arbitrarily long arithmetic progressions. Choosing $\matm_{i,j} = j(n+1)^i$ and using Lemma \ref{lem:equivtocombrich} \ref{item:multidcombrich}, we see that $A$ contains arbitrarily long, arbitrarily high-dimensional \emph{generalized arithmetic progressions} with no coinciding sums: there exist $d_1, \ldots, d_m \in \N$ so that for all $j_1, \ldots, j_m \in \{1, \ldots, n\}$,
\begin{align}\label{eqn:genap}s + j_1 d_1 + j_2 d_2 + \cdots + j_m d_m \in A,\end{align}
and there are exactly $n^m$ elements in this configuration. Given $s_1, \ldots, s_r \in \N$ as the generators of an additive $\ip_r$ set $B \subseteq \N$ and choosing $\matm_{i,j} = j s_i$, we see that $B$ contains the step sizes of a generalized arithmetic progression of length $n$ and dimension $m$ contained in $A$. Therefore, the set of (positive) differences of elements in a \crich{} set in $\seminplus$ and, more generally, the set of step sizes of generalized arithmetic progressions contained in a \crich{} set in $\seminplus$, is additively $\ipnaught^*$.

\item Suppose $\semisplus = \semintimes$. Choosing $M_{i,j} = p_i^j$, where $p_i$ is the $i^{\text{th}}$ prime number, it follows from Lemma \ref{lem:equivtocombrich} \ref{item:multidcombrich} that there exists $s \in \N$ and pairwise coprime $d_1, \ldots, d_m \in \N$ such that
\[\big\{ s d_1^{j_1} \cdots d_m^{j_m} \ \big| \ j_1, \ldots, j_m \in \{ 1, \ldots, n \} \big\} \subseteq A.\]
Thus, \crich{} sets in $\semintimes$ are more than GP-rich: they contain arbitrarily long, arbitrarily high-dimensional \emph{geometric cubes}. As an example, it is straightforward to check that the set $a \N + b$ is \crich{} in $\semintimes$ if and only if $b$ is a multiple of $a$.

\item \label{item:thirdexampleofcrstructure} Suppose $\semisplus = (\Z^d,+)$ and $n = N^d$. Using Lemma \ref{lem:equivtocombrich} \ref{item:homdef}, set $T = \seminplus$ and, for $x \in \{1, \ldots, N\}^d$, consider the homomorphism $\varphi_x: \N \to \Z^d$ defined by $\varphi_x(k) = kx$. There exist $t \in \N$ and $s \in \Z^d$ such that
\begin{align}\label{eqn:dilatedcube}s+ t\{1, \ldots, N\}^d \subseteq A.\end{align}
In fact, Lemma \ref{lem:equivtocombrich} \ref{item:homdef} gives that the set $t$'s for which there exists such an $s \in \Z^d$ is an $\ip_r^*$ subset of $\seminplus$. Thus, \crich{} sets in $(\Z^d,+)$ contain (many) dilations of any finite subset of $\Z^d$. Using $T = (\N^m,+)$ and similar homomorphisms, configurations such as the one in (\ref{eqn:dilatedcube}) can be made to be \emph{higher dimensional} as in (\ref{eqn:genap}). Setting $T = (\Z^d,+)$ and considering the $d$-by-$d$ integer matrices $W_1, \ldots, W_n$ as homomorphisms from $\Z^d$ to itself, the set of vectors $z \in \Z^d$ for which there exists $s \in \Z^d$ such that $s+W_1z, \ldots, s+W_n z \in A$ is $\ipr^*$ in $(\Z^d,+)$.

\item Suppose $\semisplus = (\F_p[x],+)$ and that $n \geq p$. Choosing $M_{i,j} = (j-1)x^i$, it follows from Lemma \ref{lem:equivtocombrich} \ref{item:multidcombrich} that there exists $s \in \F_p[x]$ and an $m$-dimensional vector subspace $V$ of $\F_p[x]$ such that
\[s + V \subseteq A.\]
Thus, \crich{} sets in $\F_p[x]$ contain \emph{affine vector subspaces} of arbitrarily high dimension.
\end{enumerate}
\end{examples}

We will show now that the class $\crichclass$ is partition regular by making use of Theorem \ref{thm:HJ} below, the Hales-Jewett theorem. Let $n, r \in \N$, and denote by $[n]$ the set $\{1,2, \ldots, n\}$. A \emph{variable word in $[n]^r$} is a word of length $r$ with letters from the alphabet $[n] \cup \{\star\}$ which contains the letter $\star$ at least once. A variable word $v$ in $[n]^r$ is considered to be a function $[n] \to [n]^r$ which sends the letter $x \in [n]$ to the word gotten by replacing each occurrence of the letter $\star$ in $v$ with the letter $x$. A \emph{combinatorial line in $[n]^r$} is the image of a variable word in $[n]^r$.

\begin{theorem}[{\cite{halesjewett}; see also \cite[Chapter 2, Theorem 3]{grsbook}}]\label{thm:HJ}
For all $n,k \in \N$, there exists $r = r(n,k) \in \N$ such that for all partitions $[n]^r = A_1 \cup \cdots \cup A_k$, some $A_i$ contains a combinatorial line.
\end{theorem}

\begin{lemma}\label{crichispr}
Let $\semisplus$ be a commutative semigroup. The class $\crichclass$ is partition regular, and the class $\crichclass^*$ is a filter.
\end{lemma}

\begin{proof}
The second statement follows from the first by Remark \ref{rmk:setalgebra}, so we focus on the first. Suppose $A_1 \cup A_2 \subseteq S$ is \crich{}; we will show that either $A_1$ or $A_2$ must be combinatorially rich. Let $n \in \N$, put $m = r(n,2)$ from Theorem \ref{thm:HJ}. Put $r = r(n,m)$ from Lemma \ref{lem:equivtocombrich} \ref{item:multidcombrich} by the combinatorial richness of the set $A_1 \cup A_2$. Let $\matM \in S^{r \times n}$. There exist disjoint, non-empty $\al_1, \ldots, \al_{m} \subseteq \{1, \ldots, r\}$ and $s \in S$ such that for all $w \in [n]^m$,
\[\pi(w) \defeq s+ \matm_{\al_1,w_1} + \matm_{\al_2,w_2} + \cdots + \matm_{\al_m,w_m} \in A_1 \cup A_2.\]
By our choice of $m$ and Theorem \ref{thm:HJ}, there exists $\ell \in \{1,2\}$ such that
\[\big\{ w \in [n]^m \ \big| \ \pi(w) \in A_\ell \big\}\]
contains the image of a combinatorial word $v$. Let $I_1 \subseteq \{1, \ldots, r\}$ be the set of indices of $v$ at which $\star$ appears, and let $I_2 = \{1,\ldots, r\} \setminus I_1$. It follows that for all $j \in \{1, \ldots, n\}$, $s + \sum_{i \in I_2} M_{\al_i,v_i} + \sum_{i \in I_1} M_{\al_i,j}  = s' + M_{\al,j} \in A_\ell$, where $s' = s + \sum_{i \in I_2} M_{\al_i,v_i}$ and $\al = \cup_{i\in I_1} \al_i$.

We have shown so far that for all $n \in \N$, there exists an $r_0=r_0(n) \in \N$ such that for all $\matM \in S^{r_0\times n}$, there exists a non-empty $\al \subseteq \{1, \ldots, r_0\}$, $s \in S$, and $\ell \in \{1,2\}$ such that for all $j \in \{1, \ldots, n\}$, $s+ \matm_{\al,j} \in A_\ell$. Suppose for a contradiction that neither $A_1$ nor $A_2$ is \crich{}. For each $\ell \in \{1,2\}$, there exists $n_\ell \in \N$ such that for $r = r_0(n_1+n_2)$, there exists $\matM^{(\ell)} \in S^{r\times n_\ell}$ such that for all $\al \subseteq \{1, \ldots, r\}$ and $s \in S$, there exists $j \in \{1, \ldots, n_\ell\}$ such that $s+ \matm^{(\ell)}_{\al,j} \not\in A_\ell$. Consider the matrix $\matM \in S^{r \times (n_1+n_2)}$ whose $i^{\text{th}}$ row is the concatenation of the $i^{\text{th}}$ rows of $\matM^{(1)}$ and $\matM^{(2)}$. By our choice of $r$, there exists a non-empty $\al \subseteq \{1, \ldots, r_0\}$, $s \in S$, and $\ell \in \{1,2\}$ such that for all $j \in \{1, \ldots, n_1+n_2\}$, $s+ \matm_{\al,j} \in A_\ell$. This contradicts the fact that there exists $j \in \{1, \ldots, n_\ell\}$ such that $s+ \matm^{(\ell)}_{\al,j} \not\in A_\ell$.
\end{proof}

The following theorem of Furstenberg and Katznelson is a density version of the Hales-Jewett theorem, Theorem \ref{thm:HJ}. We will make use of this theorem several times throughout this work to find combinatorial patterns in sets of positive density in amenable semigroups; we will show in Theorem \ref{thm:densityimpliesar}, for example, that sets of positive density are combinatorially rich.

\begin{theorem}[{\cite[Theorem 2.5]{FKdensityHJ}}]\label{thm:densityhj}
For all $n \in \N$ and $\eps > 0$, there exists $r = r(n,\eps) \in \N$ such that for all $A \subseteq [n]^r$ with $|A| \geq \eps n^r$, there exists a combinatorial line in $A$.
\end{theorem}


\section{An alternate characterization of upper Banach density}\label{sec:banach} 

In this section we give a useful alternate characterization of upper Banach density for a large class of semigroups. We first provide some motivation for the characterization; see also \cite[Corollary 9.2]{Griesmer}, \cite[Theorem G]{johnsonrichter}, and \cite[Lemma 9.6]{furstenbergkatznelsonipszem}.

Let $\semistimes$ be a cancellative semigroup. By definition, a subset $A \subseteq S$ is thick if for all $F \in \finitesubsets S$, there exists an $s \in S$ for which
\[Fs \subseteq A, \quad \text{which is the same as} \quad |Fs \cap A| \geq 1 |F|.\]
The second expression says that there exists a right translate of $F$ with the property that 100\% of its elements lie within $A$. This begs the question: given $A \subseteq S$ and $F \in \finitesubsets S$, how much of the set $F$, as a percentage of $|F|$, can we translate into $A$? That is, how large is the quantity
\[\sup \big\{ \denlet \geq 0 \ \big| \ \forall F \in \finitesubsets S, \ \exists s \in S, \ |Fs \cap A| \geq \denlet |F| \big\}?\]
This quantity is equal to 1 if and only if $A$ is thick in $S$. We will show that, with a slight modification, this quantity is equal to the upper Banach density of $A$ in a wide class of semigroups.

We begin by describing a multiset generalization of this idea. The utility of this more general formulation is manifest in the proof of Theorem \ref{thm:densityimpliesar} and will be useful several times later on.

\begin{definition}\label{def:multiset}
Let $S$ be a set. A \emph{multiset $\multif$ in $S$} is a 2-tuple consisting of a \emph{support} $\underf \subseteq S$ and a function $\multif: \underf \to \N$.  We say that $\multif$ is finite if $|\multif| \defeq \sum_{f \in F} \multif(f)$ is finite (equivalently, if $\underf$ is a finite set). If $\multif_1$ and $\multif_2$ are multisets in $S$, then $\multif_1 \cap \multif_2: F_1 \cap F_2 \to \N$ is defined by pointwise multiplication. There is a 1--1 correspondence between sets and constant $1$ multisets (regarding them as functions), and the terminology here agrees with the familiar terminology for sets.
\end{definition}

\begin{theorem}\label{thm:characterizationofdensity}
Let $\semistimes$ be a left amenable semigroup. For all $A \subseteq S$,
\begin{align}\label{eqn:densityequivalence}\dstar(A) \leq \sup \left\{ \denlet \geq 0 \ \middle| \ \begin{aligned} & \text{$\forall$ finite multisets $\multif$ in $S$}, \\ & \exists s \in S, \ |\multif \cap A s^{-1}| \geq \denlet |\multif| \end{aligned} \right\}.\end{align}
Moreover, for all left invariant means $\lambda$, all $\beta \in [0,\lambda(A)]$, and all finite multisets $\multif$ in $S$, the set
\[S_\beta(A,\multif) = \big\{ s \in S \ \big| \ |\multif \cap A s^{-1}| \geq \beta |\multif| \big\}\]
satisfies $\lambda \big( S_\beta(A,\multif) \big) \geq (\lambda(A) - \be) \big/ (1-\be)$; in particular, $\dstar \big(S_\beta(A,\multif) \big) \geq (\dstar(A) - \be) \big/ (1-\be)$.
\end{theorem}

\begin{proof}
Suppose $\dstar(A) > 0$, let $\eps > 0$, let $\lambda$ be a left invariant mean for which $\lambda(A) \geq \dstar(A) - \eps$. Let $\multif$ be a finite multiset in $S$ with support $F$. Consider the function $g: S \to [0,1]$ given by
\[g(s) = \frac{|\multif \cap A s^{-1}|}{|\multif|} = \frac {1}{|\multif|} \sum_{f \in \underf} \multif(f) \one_{A s^{-1}}(f) = \frac {1}{|\multif|} \sum_{f \in \underf} \multif(f) \one_{f^{-1}A}(s).\]
Since $\lambda$ is left invariant,
\begin{align}\label{eqn:gislarge}\lambda(g) = \frac 1{|\multif|} \sum_{f \in F} \multif(f) \lambda(f^{-1}A) = \lambda(A) \geq \dstar(A) - \eps.\end{align}
It follows by the positivity of $\lambda$ that there exists an $s \in S$ for which $g(s) \geq \dstar(A) - 2\eps$. Thus the right hand side of (\ref{eqn:densityequivalence}) is not less than $\dstar(A) - 2\eps$. Since $\eps > 0$ was arbitrary, the inequality in (\ref{eqn:densityequivalence}) holds.

Let $\lambda$ be an arbitrary left invariant mean on $S$, $\beta \in [0,\lambda(A)]$, and $\multif$ be a finite multiset in $S$. If $\beta = \lambda(A)$, there is nothing to show, so suppose $\beta < \lambda(A) - \eps$. Abbreviating $S_\beta(A,\multif)$ by $S_\beta$, it follows from (\ref{eqn:gislarge}) that
\begin{align*}
\lambda(A) = \lambda(g) &\leq \lambda \left( g \one_{S_\beta} \right) + \lambda \left( g \one_{S \setminus S_\beta} \right) \\
& \leq \lambda \left( S_\beta \right) + \beta \left( 1- \lambda(S_\beta) \right).
\end{align*}
Rearranging the previous inequality, we find that $\lambda \left( S_\beta \right) \geq (\lambda(A) - \be) \big/ (1-\be)$.
\end{proof}

\begin{theorem}\label{thm:densityimpliesar}
Let $\semisplus$ be a commutative semigroup and $A \subseteq S$. If $\dstar(A) > 0$, then $A$ is \crich{}.
\end{theorem}

\begin{proof}
Let $n \in \N$ and $0 < \eps < \dstar(A)$. Put $r = r(n,\eps)$ from Theorem \ref{thm:densityhj}. Let $\matM \in S^{r \times n}$, and define $\pi: [n]^r \to S$ by $\pi(w) = M_{1,w_1} + \cdots + M_{r,w_r}$. Let $F \subseteq S$ be the image of $\pi$, and define the multiset $\multif: F \to \N$ by
\[\multif(f) = \big| \{ w \in [n]^r \ | \ \pi(w) = f \}\big|.\]
By Theorem \ref{thm:characterizationofdensity}, there exists $s' \in S$ such that
\[\eps |\multif| \leq |\multif \cap A - s'| = \big| \{ w \in [n]^r \ | \ \pi(w) \in A - s' \}\big| = \big| \{ w \in [n]^r \ | \ s'+ \pi(w) \in A \}\big|.\]
Since $|\multif| = n^r$, Theorem \ref{thm:densityhj} gives that $\{ w \in [n]^r \ | \ s'+ \pi(w) \in A \}$ contains the image of a combinatorial word $v$ in $[n]^r$. Let $\al \subseteq \{1, \ldots, r\}$ be the set of indices of $v$ at which $\star$ appears, and let $\beta = \{1,\ldots, r\} \setminus \al$. Putting $s = s' + \sum_{i \in \beta} M_{i,v_i}$, it follows that for all $j \in \{1, \ldots, n\}$, $s + M_{\al,j} \in A$, meaning that $A$ is \crich{}.
\end{proof}

This completes the proof of Lemma \ref{lem:hierarchylemma}. We turn now to explaining when equality occurs in (\ref{eqn:densityequivalence}). It will be helpful to be able to construct left invariant means; this is simplest in semigroups satisfying the strong \folner{} condition, defined next.

\begin{definition} \label{def:sfcandfolner}
Let $\semistimes$ be a semigroup.
\begin{enumerate}[label=(\Roman*)]
\item $S$ satisfies the \emph{strong \folner{} condition} if
\begin{align}\tag{SFC}\label{eqn:sfc} \forall F \in \finitesubsets S, \ \forall \epsilon > 0, \ \exists H \in \finitesubsets S, \ \forall f \in F, \ \big| H \setminus fH \big| < \eps |H|.  \end{align}
\item Define the function $\sfcdstar: \subsets S \to [0,1]$ by
\[\sfcdstar(A) = \sup \Bigg\{ \denlet \geq 0 \ \Bigg| \ \begin{aligned} & \forall F \in \finitesubsets S, \ \forall \eps > 0, \ \exists H \in \finitesubsets S, \\ & |A \cap H| \geq \denlet |H| \text{ and } \forall f \in F, \ |H \setminus fH| < \eps |H| \end{aligned} \Bigg\}.\]
\item \label{item:folnersequence} A \emph{left \folner{} sequence} is a sequence $(F_n)_{n\in \N}$ of finite, non-empty subsets of $S$ with the property that for all $s \in S$, $|sF_n \bigtriangleup F_n| \big/ |F_n| \to 0$ as $n \to \infty$.
\end{enumerate}
\end{definition}

Any semigroup satisfying \sfc{} is left amenable. A countable semigroup satisfies \sfc{} if and only if it admits a left \folner{} sequence, and in this case,
\begin{align}\label{eqn:densityaslimitoffolner}\sfcdstar(A) = \sup \left\{ \limsup_{n \to \infty} \frac{|A \cap F_n|}{|F_n|} \ \middle| \ (F_n)_{n\in \N} \text{ is a \folner{} sequence} \right\}.\end{align}
The previous sentence holds true for uncountable semigroups with ``\folner{} sequence'' replaced by ``\folner{} net.'' We shall only need the fact that in semigroups $\semistimes$ satisfying \sfc{}, for all $A \subseteq S$, $\sfcdstar(A) \leq \dstar(A)$. This follows from the fact that any weak-$\ast$ cluster point $\lambda$ of a net of means arising from a \folner{} net in $S$ which realizes $\sfcdstar(A)$ will be left translation invariant and satisfy $\lambda(A) \geq \sfcdstar(A)$.\footnote{In fact, any such cluster point $\lambda$ is a left-invariant mean, and in a left-cancellative semigroup (see the definition in the following paragraph), every left-invariant mean is defined by a \folner{} net in the way described in the previous sentence. For a proof of these facts, see \cite[Section 1]{hindmanstrauss-semigroupdensity1}.} Since we will not make use of \folner{} nets in this paper, we refer the reader to \cite[Section 1]{hindmanstrauss-semigroupdensity1} for definitions, proofs, and related discussion.

A semigroup $\semistimes$ is \emph{left cancellative} if for all $x,y,z \in S$, the equality $x y = x z$ implies that $y=z$. \emph{Right cancellativity} is defined analogously. It will be convenient to name a class of semigroups satisfying more than \sfc{}; thus, we introduce the shorthand:
\begin{align}\tag{SFC+} \sfc{} \text{ \sc{and} } \big(\text{left cancellative } \text{\sc{or}} \text{ right cancellative } \text{\sc{or}} \text{ commutative}\big).  \end{align}
Since all commutative semigroups satisfy \sfc{}, all commutative semigroups satisfy \sfcplus{}. Also, all amenable groups satisfy \sfcplus{}.

\begin{theorem}\label{thm:charofdensityifsfcplus}
Let $\semistimes$ be a semigroup satisfying \sfcplus{}. For all $A \subseteq S$, $\sfcdstar(A) = \dstar(A)$, and both are equal to
\begin{align}\label{eqn:setdensity}\sup \left\{ \denlet \geq 0 \ \middle| \ \forall F \in \finitesubsets S, \ \exists s \in S, \ |F \cap A s^{-1}| \geq \denlet |F| \right\}.\end{align}
In particular, the inequality in (\ref{eqn:densityequivalence}) is an equality.
\end{theorem}

\begin{proof}
Denote temporarily the quantity in (\ref{eqn:setdensity}) by $\dstar_t(A)$ (``t'' for \emph{translate}). By Theorem \ref{thm:characterizationofdensity} and previous remarks,
\[\sfcdstar(A) \leq \dstar(A) \leq \text{right hand side of (\ref{eqn:densityequivalence})} \leq \dstar_t(A).\]
Therefore, it suffices to show that $\dstar_t(A) \leq \sfcdstar(A)$.

We will show first that $S$ satisfies a condition related to \sfc{}, namely
\begin{align}\label{eqn:strongerthansfc} \begin{gathered} \forall F \in \finitesubsets S, \ \forall \epsilon > 0, \ \exists H \in \finitesubsets S, \\ \forall s \in S, \ |Hs| = |H| \text{ and } \forall f \in F, \ \big| H \setminus fH \big| < \eps |H|. \end{gathered}  \end{align}
This follows immediately from \sfc{} if $S$ is right cancellative. If $S$ is commutative or left cancellative, this follows from the proof of \cite[Theorem 4]{awstrongfolner} and the discussion following it using the fact that left cancellative, left amenable semigroups satisfy \sfc{}.

Let $A \subseteq S$ and $\denlet < \dstar_t(A)$. We must show $\sfcdstar(A) \geq \denlet$. Let $F \in \finitesubsets S$ and $\eps > 0$. There exists a set $H \in \finitesubsets S$ satisfying the second line in (\ref{eqn:strongerthansfc}). By our choice of $\denlet$, there exists an $s \in S$ for which $|H \cap A s^{-1}| \geq \denlet |H|$. Since $|Hs| = |H|$, it follows that $|Hs \cap A| \geq \denlet |Hs|$ and that for all $f \in F$,
\[|Hs \setminus fHs| \leq |(H \setminus fH) s| = |H \setminus fH| < \eps |H| = \eps |Hs|.\]
The set $Hs$ satisfies the conditions required of $H$ in the definition of $\sfcdstar(A)$. Since $F$ and $\eps$ were arbitrary, this shows that $\sfcdstar(A) \geq \denlet$.
\end{proof}

In addition to extending the results in \cite[Section 2]{hindmanstrauss-semigroupdensity2}, this theorem allows us to handle upper Banach density in suitable semigroups without knowledge of the form \folner{} sequences take.

\begin{corollary}\label{cor:densitycharforcancelsemigroups}
Let $\semistimes$ be a right cancellative semigroup satisfying \sfc{}. For all $A \subseteq S$,
\begin{align}\label{eqn:dstarwithfs}\dstar(A) = \sup \left\{ \denlet \geq 0 \ \middle| \ \forall F \in \finitesubsets S, \ \exists s \in S, \ |Fs \cap A| \geq \denlet |F| \right\}.\end{align}
Suppose, in addition, that $\semistimes$ is countable, and fix a \folner{} sequence $(F_n)_{n\in\N}$. For all $A \subseteq S$,
\begin{align}\label{eqn:dstarwithtranslatedfolner}\dstar(A) = \limsup_{n \to \infty} \max_{s \in S} \frac{|A \cap F_n s|}{|F_n|}.\end{align}
\end{corollary}

\begin{proof}
Let $A \subseteq S$ and $F \in \finitesubsets S$. Since $\semistimes$ is right cancellative, for all $s \in S$, 
\[\big|Fs \cap A \big| = \big|( Fs \cap A ) s^{-1}\big| = \big|F \cap A s^{-1}\big|.\]
The equality in (\ref{eqn:dstarwithfs}) follows immediately from this and Theorem \ref{thm:charofdensityifsfcplus}.

It follows now from (\ref{eqn:dstarwithfs}) that the right hand side of (\ref{eqn:dstarwithtranslatedfolner}) is greater than $\dstar(A)$. If the limit supremum is obtained along the sequence $(n_k)_{k \in \N} \subseteq \N$ and $(s_k)_{k \in \N} \subseteq S$ is the corresponding sequence of elements of $S$, then it is simple to check that the sequence $(F_{n_k}s_k)_{k \in \N}$ is a \folner{} sequence. By (\ref{eqn:densityaslimitoffolner}) and Theorem \ref{thm:charofdensityifsfcplus}, the right hand side of (\ref{eqn:dstarwithtranslatedfolner}) is less than $\dstar(A)$, as was to be shown.
\end{proof}

This shows that the heuristic for upper Banach density discussed in the beginning of this section holds as written for right cancellative semigroups satisfying \sfc{}. This corollary also proves that the expressions for additive and multiplicative upper Banach density for subsets of $\N$ in (\ref{eqn:densityinseminplusintro}) and (\ref{eqn:densityinseminplusintromult}) from the introduction coincide with the upper Banach density defined in Definition \ref{def:defofamenable}. More generally, for countable, right cancellative semigroup satisfying \sfc{}, it shows that right translates of the finite sets which make up any single \folner{} sequence suffice to capture the upper Banach density as given in (\ref{eqn:densityaslimitoffolner}) as a limit over all \folner{} sequences.


\section{Semirings: context for the interaction of additive and multiplicative largeness}\label{semiringsandexamples}  

One goal of this paper is to explain the degree to which notions of additive and multiplicative largeness interact. Semirings are basic algebraic objects in which a study of this interaction makes sense.

\begin{definition}\label{def:semiring}
A \emph{semiring} $(S,+,\cdot)$ is a set $S$ together with two binary operations $+$ (addition) and $\cdot$ (multiplication) for which
\begin{enumerate}[label=(\Roman*)]
\item $\semisplus$ is a commutative semigroup;
\item $\semistimes$ is a semigroup;
\item \label{item:distributivity} \emph{left and right distributivity} hold: for all $s_1, s_2 \in S$ and $s \in S$,
\[s (s_1 + s_2) = s s_1 + s s_2 \quad \text{and} \quad (s_1 + s_2) s = s_1 s + s_2 s.\]
\end{enumerate}
The semiring $\semisring$ is \emph{commutative} if $\semistimes$ is a commutative semigroup.
\end{definition}

Addition on the left or right by a fixed element of $S$ will be called \emph{translation} and multiplication on the left or right will be called \emph{dilation}. Another way to phrase \ref{item:distributivity} above is that left and right dilation are homomorphisms of the additive semigroup $\semisplus$.

\begin{examples}\label{ex:beginningexamplessectionfour} Every (not-necessarily unital or commutative) ring is a semiring; thus, $\Z$, $\Q$, $\R$, and $\Z \big/ n\Z$ with the usual operations are semirings. Every two sided ideal of a ring is a semiring, and the cone of positive elements of a partially ordered ring is a semiring; the quintessential semiring $\N$ arises from the ring $\Z$ in this way. Familiar algebraic constructions may be used to generate semirings from a given semiring $\semisring$.
\begin{enumerate}[label=(\Roman*)]
\item $M_d(S)$, the set of $d$-by-$d$ matrices with elements from $S$ under matrix addition and multiplication, is a semiring. This class includes algebraic extensions of $\Z$ and $\Q$, for example the rational quaternions $\mathbb{H}(\Q)$ and (cones in) rings of integers of number fields such as $\N[\sqrt{2}]$.
\item When $\semisring$ is commutative, $S[x_1, \ldots, x_n]$, the set of polynomials in the variables $x_1, \ldots, x_n$ with coefficients in $S$ with polynomial addition and multiplication, is a semiring.
\item Given any set $X$ and a semiring $\semisring$, the set of functions $\{f: X \to S\}$ under pointwise addition and multiplication is a semiring.
\item For $\semistimes$ any semigroup, the endomorphism ring $\{\varphi: \semistimes \to \semistimes\}$ of semigroup homomorphisms under pointwise multiplication and composition is a semiring.
\end{enumerate}
The semirings $(\finitesubsets \N,\allowbreak \cup,\allowbreak \cap)$ and $(\Z,\min,\max)$, the former of which has significance in Ramsey theory, are further examples. Curiously, addition and multiplication are interchangeable in both of these semirings.
\end{examples}

The first results regarding the interaction of additive and multiplicative largeness concern translations and dilations of large sets. The following examples show that translation does not always preserve multiplicative largeness.

\begin{example} In $\N$, the even numbers $2\N$ have multiplicative density $\dstar_{\semintimes}(2\N) = 1$ because, by (\ref{eqn:densityproperties}), \[\dstar_{\semintimes}(2\N) = \dstar_{\semintimes}(2^{-1} \cdot 2\N) = \dstar_{\semintimes}(\N) = 1.\]
The odd numbers $2\N - 1$, however, have zero multiplicative density, since the set $2^{-1} \cdot (2\N - 1) = \emptyset$ has zero multiplicative density. For a more extreme example, note that the set $A = 4 \N + \{0,1\}$ is multiplicatively $\ip_3^*$ and in $\crichclass^*$ while $A+2$ is neither multiplicatively $\ip_3$ nor \crich{} (take $\matM$ in Definition \ref{def:combrich} to be a multiple of $4$, for example).
\end{example}

While translation does not preserve multiplicative largeness, dilation does preserve additive largeness. This is a consequence of the more general fact that the images of a ``large'' set under suitable semigroup homomorphisms are ``large.'' We prove two lemmas along these lines that will be useful later on. Recall the definition of \sfcplus{} just prior to the statement of Theorem \ref{thm:charofdensityifsfcplus} in Section \ref{sec:banach}.

\begin{lemma}\label{lem:surjectivehom}
Let $\semistimes$ and $\semittimes$ be semigroups, $\varphi: \semistimes \to \semittimes$ be a surjective homomorphism, $A \subseteq S$, $B \subseteq T$, and $r \in \N$. For all $\arbclass \in \{ \syndetic,\allowbreak \thick,\allowbreak \pws,\allowbreak \pws^*,\allowbreak \central,\allowbreak \central^*,\allowbreak \iprclass,\allowbreak \iprclass^*,\allowbreak \ipnaughtclass,\allowbreak \ipnaughtclass^*,\allowbreak \ipclass,\allowbreak \ipclass^*\}$,
\begin{align}
\label{eqn:rightarrow} A \in \arbclass \semistimes \quad &\Longrightarrow \quad  \varphi(A) \in \arbclass \semittimes, \\
\label{eqn:leftarrow} \varphi^{-1}(B) \in \arbclass \semistimes \quad &\Longleftarrow \quad B \in \arbclass \semittimes.
\end{align}
If $\semistimes$ and $\semittimes$ are commutative, then (\ref{eqn:rightarrow}) and (\ref{eqn:leftarrow}) hold for $\arbclass \in \{ \crichclass,\allowbreak \crichclass^*\}$. If $\semistimes$ and $\semittimes$ satisfy \sfcplus{}, then (\ref{eqn:rightarrow}) and (\ref{eqn:leftarrow}) hold for $\arbclass \in \{ \density,\allowbreak \density^*\}$. 
\end{lemma}

\begin{proof}
Note that if the lemma holds for $\arbclass$, then it holds for the dual class $\arbclass^*$; therefore, it suffices to prove the lemma for the eight classes $\syndetic$, $\pws$, $\central$, $\iprclass$, $\ipnaughtclass$, $\ipclass$, $\crichclass$, and $\density$. It is routine to check that the lemma holds for each $\arbclass \in \{ \syndetic, \pws, \central\}$ from the facts in \cite[Exercise 1.7.3]{hindmanstrauss-book} using the ultrafilter characterization of those classes. The lemma is easily verified from the definitions for $\arbclass \in \{ \ipclass, \iprclass, \ipnaughtclass\}$. The statement in (\ref{eqn:rightarrow}) follows for $\arbclass \in \{\crichclass, \density \}$ from Lemma \ref{lem:bigandbigimpliesbig} below.

Suppose $\semistimes$ and $\semittimes$ are commutative; we will write the operation ``$\cdot$'' as ``$+$'' in this paragraph. To verify (\ref{eqn:leftarrow}) for $\arbclass = \crichclass$, suppose $B$ is \crich{} in $T$. Let $n \in \N$, and choose $r = r(n)$ guaranteed by the definition of combinatorial richness for $B$. Let $\matM \in S^{r\times n}$. Considering the matrix $\varphi(\matM) \in T^{r \times n}$, since $B$ is \crich{}, there exists $\al \subseteq \{1,\ldots,r\}$ and $t \in T$ such that for all $j \in \{1,\ldots, n\}$, $t+\varphi(M_{\al,j}) \in B$. Since $\varphi$ is surjective, there exists $s \in S$ for which $t=\varphi(s)$. Therefore, for all $j \in \{1, \ldots, n\}$, $s+M_{\al,j} \in \varphi^{-1}(B)$, showing that $\varphi^{-1}(B)$ is \crich{}.

Suppose $\semistimes$ and $\semittimes$ satisfy \sfcplus{}. We will show (\ref{eqn:leftarrow}) for $\arbclass = \density$ with the help of Theorem \ref{thm:charofdensityifsfcplus}. Let $B \subseteq T$ with $\dstar(B) \geq \denlet$. Let $\multif$ be a finite multiset in $S$ with support $F$. Let $\multih$ be the multiset on $H = \varphi(F)$ defined by
\[\multih(h) = \sum_{\substack{f \in F \\ \varphi(f)=h}} \multif(f),\]
and note that $|\multih| = |\multif|$. Since $\dstar(B) \geq \denlet$, there exists a $t \in T$ such that $|\multih \cap Bt^{-1}| \geq \denlet |\multih|$. Since $\varphi$ is surjective, there exists an $s \in S$ for which $\varphi(s) = t$. Now
\begin{align*}
\denlet |\multif| = \denlet |\multih| \leq |\multih \cap By^{-1}| &= \sum_{h \in H} \multih(h) \one_{B}(ht) \\
&= \sum_{h \in H} \sum_{\substack{f \in F \\ \varphi(f)=h}} \multif(f) \one_{B}(ht) \\
&= \sum_{h \in H} \sum_{\substack{f \in F \\ \varphi(f)=h}} \big( \multif(f) \one_{B}(\varphi(f)\varphi(s)) \big) \\
&= \sum_{f \in F} \multif(f) \one_{B}(\varphi(fs)) = |\multif \cap \varphi^{-1}(B)s^{-1}|.
\end{align*}
Since $\multif$ was arbitrary, this shows that $\dstar(\varphi^{-1}B) \geq \denlet$.
\end{proof}

This lemma provides us with a tool to generate examples of large sets from other large sets with surjective homomorphisms.

\begin{examples}\label{examples:homomorphisms} \leavevmode
\begin{enumerate}[label=(\Roman*)]
\item For $p \in \N$ a prime, the $p$-adic valuation $\nu_p: \semintimes \to (\N \cup \{0\},+)$ defined by $\nu_p(p^{e}p_1^{e_1} \cdots p_k^{e_k}) = e$, where the $p_i$'s are distinct primes different from $p$, is a surjective semigroup homomorphism. Let $\irratnum \in \R \setminus \Q$ and $I \subseteq [0,1)$ an interval, and note that the set $A \defeq \{n \in \N \ | \ \{n \irratnum\} \in I \}$, where $\{x\}$ denotes the fractional part of $x$, is syndetic in $\seminplus$. By Lemma \ref{lem:surjectivehom}, the set
\[\nu_p^{-1}(A) = \big\{ n \in \N \ \big| \ \{\nu_p(n) \irratnum \} \in I \big\}\]
is syndetic in $\semintimes$. We will give a polynomial generalization of this class of multiplicatively syndetic sets in Section \ref{sec:newclassofmultsynd}.
\item \label{item:multeven} Define $\Omega: \semintimes \to (\N \cup \{0\},+)$ by $\Omega(p_1^{e_1} \cdots p_k^{e_k}) = \sum_{i=1}^k e_i$, where $p_1, \ldots, p_k$ are distinct primes. Since $\Omega$ is a surjective semigroup homomorphism, Lemma \ref{lem:surjectivehom} gives that the \emph{multiplicatively even and odd numbers}
\[\multevenn = \Omega^{-1} \big( 2\N \big) \qquad \multoddn = \Omega^{-1} \big( 2\N - 1 \big)\]
are both syndetic subsets of $\semintimes$. In fact, $\multevenn$ is multiplicatively $\ip_2^*$ since $2\N$ is additively $\ip_2^*$. The homomorphism $\Omega$ extends naturally  to a homomorphism $(\Q\setminus \{0\},\cdot) \to \semizplus$, allowing us to define the multiplicatively even and odd rational numbers. We will say more about these sets in Section \ref{sec:syndeticinfields}.
\item The natural logarithm $\log: ((0,\infty),\cdot) \to (\R,+)$ is a surjective semigroup homomorphism. Since the set of Liouville numbers is thick\footnote{Since the set of Liouville numbers $L \subseteq \R$ are a dense $G_\delta$ set, even more is true: for any countable set $F \subseteq \R$, the set $\cap_{f \in F} (-f + L)$ is a dense $G_\delta$ set, in particular non-empty, implying there exists an $x \in \R$ for which $F + x \subseteq L$.} in $(\R,+)$, the set of positive real numbers whose logarithm is a Liouville number is multiplicatively thick in $(0,\infty)$. In Section \ref{sec:newclassofmultsynd} we will use the natural logarithm to generate a class of multiplicatively syndetic sets.
\item The determinant $\det: (M_d(\N) \cap \gldq,\cdot) \to (\N,\cdot)$ is a surjective semigroup homomorphism. Let $A$ be the set of those $n \in \N$ which contain in their canonical prime factorization a term of the form $p^{100}$; it is shown in \cite[Pg. 1223]{bbhsaddimpliesmult} that $A$ is $\pws^*$ in $\seminplus$. Thus, the set of those matrices $M \in M_d(\N) \cap \gldq$ with $\Omega(\det(M)) \in A$ is in $\pws^*(M_d(\N) \cap \gldq,\cdot)$.
\item For any $a \in \F_p$, evaluation $f(x) \mapsto f(a)$ is a surjective homomorphism from $(\F_p[x],+)$ to $(\F_p,+)$ and from $(\F_p[x] \setminus \{0\},\cdot)$ to $(\F_p\setminus\{0\},\cdot)$. Thus, the set of those non-zero polynomials which take $1$ at $a$ is both additively syndetic in $(\F_p[x],+)$ and multiplicatively syndetic in $(\F_p[x]\setminus\{0\},\cdot)$.
\end{enumerate}
\end{examples}

Because dilation is rarely surjective, Lemma \ref{lem:surjectivehom} does not yield information about the additive largeness of dilated sets. The following lemma addresses this and will be useful several times throughout the paper. It says that if the image of a semigroup homomorphism is ``large,'' then the image of any ``large'' set is ``large.''

\begin{lemma}\label{lem:bigandbigimpliesbig}
Let $\semistimes$, $\semittimes$ be semigroups, $\varphi: \semistimes \to \semittimes$ be a homomorphism, $A \subseteq S$, and $r \in \N$.
\begin{enumerate}[label=(\Roman*)]
\item \label{item:central} If $A$ is central in $S$ and $\varphi(S)$ is piecewise syndetic in $T$, then $\varphi(A)$ is central in $T$.
\item[(I\,$'$)] \label{item:pshindman} If $A$ is piecewise syndetic in $S$ and $\varphi(S)$ is piecewise syndetic in $T$, then $\varphi(A)$ is piecewise syndetic in $T$.
\item \label{item:ipr} If $A$ is $\ip_r \big/ \ipnaught \big/ \ip$ in $S$, then $\varphi(A)$ is $\ip_r \big/ \ipnaught \big/ \ip$ in $T$, respectively.
\item \label{item:iprstar} If $A$ is $\ipnaught^*$ in $S$ and $\varphi(S)$ is $\ipnaught^*$ in $T$, then $\varphi(A)$ is $\ipnaught^*$ in $T$.
\item \label{item:comborich} Suppose $\semistimes$ and $\semittimes$ are commutative. If $A$ is \crich{} in $S$ and $\varphi(S)$ is $\ipnaught^*$ in $T$, then $\varphi(A)$ is \crich{} in $T$.
\item \label{item:density} Suppose $\semistimes$ and $\semittimes$ satisfy \sfcplus{}. If $\dstar_{\semistimes}(A) > 0$ and $\dstar_{\semittimes}(\varphi(S)) > 0$, then $\dstar_{\semittimes}(\varphi(A)) > 0$.
\end{enumerate}
\end{lemma}

\begin{proof}
\ref{item:central} \ (We employ here, and only here, the machinery of ultrafilters and algebra in the Stone-\v Cech compactification $\beta S$ of the semigroup $\semistimes$. The reader is referred to \cite[Chapter 4]{hindmanstrauss-book} for the requisite definitions.) The homomorphism $\varphi$ from the (discrete) semigroup $\semistimes$ to $\semittimes$ induces a homomorphism $\Phi$ from the (compact right topological) semigroup $(\beta S,\cdot)$ to $(\beta T,\cdot)$ defined for $p \in \beta S$ by
\[\Phi(p) = \big\{B \subseteq T \ \big| \ \varphi^{-1}B \in p \big\}.\]
For any $C \subseteq S$, $\Phi(\overline{C}) \subseteq \overline{\varphi(C)}$, and it is an exercise to show that $\Phi(\beta S) = \overline{\varphi(S)}$.

Denote by $K(\beta S)$ and $K(\beta T)$ the smallest (minimal) two-sided ideal of $(\beta S,\cdot)$ and $(\beta T,\cdot)$, respectively. Since $\varphi(S)$ is piecewise syndetic in $T$, $\overline{\varphi(S)} \cap K(\beta T) \neq \emptyset$ (\cite[Theorem 4.40]{hindmanstrauss-book}). Since $\overline{\varphi(S)} = \Phi(\beta S)$, this shows that $\Phi^{-1}(K(\beta T)) \neq \emptyset$. Since the preimage under $\Phi$ of any two-sided ideal is a two-sided ideal, $\Phi^{-1}(K(\beta T))$ is a non-empty two-sided ideal, so it contains $K(\beta S)$. This means $\Phi(K(\beta S)) \subseteq K(\beta T)$.

Denote by $E(\beta S)$ and $E(\beta T)$ the set of idempotents of $(\beta S,\cdot)$ and $(\beta T,\cdot)$, respectively. Since $\Phi$ is a homomorphism, $\Phi(E(\beta S)) \subseteq E(\beta T)$. By the definition of $A$ being central in $S$, $\overline{A} \cap K(\beta S) \cap E(\beta S) \neq \emptyset$. Applying $\Phi$, we see that
\[\Phi(\overline{A}) \cap \Phi(K(\beta S)) \cap \Phi(E(\beta S)) \neq \emptyset, \ \text{ whereby } \ \overline{\varphi(A)} \cap K(\beta T) \cap E(\beta T) \neq \emptyset.\]
This implies that $\varphi(A)$ is central in $T$.

(I$'$) \ It was remarked by Neil Hindman that (I$'$) follows as a consequence of (I) and \cite[Theorem 4.43]{hindmanstrauss-book}.\footnote{The version of \cite[Theorem 4.43]{hindmanstrauss-book} we use states: \emph{For any semigroup $\semistimes$, a subset $A \subseteq S$ is piecewise syndetic if and only if there exists $s \in S$ such that $s^{-1} A$ is central.} We have dropped the assumption in the statement that $\semistimes$ is infinite at the cost of a weaker conclusion; the proof in \cite{hindmanstrauss-book} must be modified slightly by taking the idempotent $e$ to be minimal.} By that theorem, there exists $s \in S$ such that $s^{-1}A$ is central in $S$. By (I), the set $\varphi(s^{-1}A)$ is central in $T$. Since $\varphi(s^{-1}A) \subseteq \varphi(s)^{-1} \varphi(A)$, the set $\varphi(s)^{-1} \varphi(A)$ is central in $T$, and it follows by \cite[Theorem 4.43]{hindmanstrauss-book} again that $\varphi(A)$ is piecewise syndetic in $T$.

\ref{item:ipr} \ If $\finitesums (s_i)_{i=1}^r \subseteq A$, then $\finitesums (\varphi(s_i))_{i=1}^r = \varphi( \finitesums (s_i)_{i=1}^r ) \subseteq \varphi(A)$. The other statements follow similarly.

\ref{item:iprstar} \ Let $r, t \in \N$ so that $A$ is an $\ip_r^*$ set in $S$ and $\varphi(S)$ is an $\ip_t^*$ set in $T$. We will show that $\varphi(A)$ is an $\ip_{rt}^*$ set in $T$ by showing that for every $(t_i)_{i=1}^{rt} \subseteq T$, there exists a non-empty $\gamma \subseteq \{1, \ldots, rt\}$ for which $t_\gamma \in \varphi(A)$.

Let $(t_i)_{i=1}^{rt} \subseteq T$. Since $\varphi(S)$ is an $\ip_t^*$ set in $T$, for each $i \in \{1, \ldots, r\}$, there exists $\alpha_i \subseteq \{(i-1)t + 1, \ldots, it\}$ for which $t_{\al_i} \in \varphi(S)$. For each $i$, let $s_i \in S$ be such that $\varphi(s_i) = t_{\al_i}$. Since $A$ is $\ip_r^*$ in $S$, there exists a non-empty $\beta \subseteq \{1, \ldots, r\}$ such that $s_\beta \in A$. Applying $\varphi$ and using that the $\alpha_i$'s are disjoint, we see that there exists a $\gamma \subseteq \{1, \ldots, rt\}$ for which $t_{\gamma} \in \varphi(A)$.

\ref{item:comborich} \ Suppose $\semistimes$ and $\semittimes$ are commutative; we will write the operation ``$\cdot$'' as ``$+$'' in this paragraph and the next. Let $n \in \N$. Since $\varphi(S)$ is $\ipnaught^*$ in $T$, $\varphi(S)^n$ is $\ipnaught^*$ in $T^n$, where the operation on $T^n$ is the operation from $T$ applied coordinate-wise.\footnote{This fact follows from a more general one: the Cartesian product of $\ipnaught^*$ sets is an $\ipnaught^*$ set. To see why, write the product of $\ipnaught^*$ sets as an intersection of preimages of $\ipnaught^*$ sets under the coordinate projection maps. It is simple to check that each preimage is an $\ipnaught^*$ set. Since the intersection of $\ipnaught^*$ sets is an $\ipnaught^*$ set (Remark \ref{rmk:partitionregularityofipnaughtandcentral}), the Cartesian product is an $\ipnaught^*$ set.} Let $R \in \N$ be such that $\varphi(S)^n$ is $\ip_R^*$ in $T^n$, let $r = r(n)$ be from the definition of combinatorial richness for $A$, and put $\ell = Rr$. Let $\matM \in T^{\ell \times n}$. Since $\varphi(S)^n$ is $\ip_R^*$ in $T^n$, for $i \in \{1, \ldots, r\}$, there exists $\al_i \subseteq \{(i-1)R+1,\ldots, iR\}$ such that $\matM' = (M_{\al_i,j}) \in \varphi(S)^{r \times n}$.

Let $\mathbf{N} \in S^{r \times n}$ be such that $\varphi(\mathbf{N}) = \matM'$. Since $A$ is \crich{}, there exists a non-empty $\beta \subseteq \{1, \ldots, r\}$ and $s \in S$ such that for all $j \in \{1, \ldots, n\}$, $s+ N_{\beta,j} \in A$. Applying $\varphi$, this means that there exists a non-empty $\al \subseteq \{1, \ldots, \ell\}$ such that for all $j \in \{1, \ldots, n\}$, $\varphi(s)+ M_{\al,j} \in \varphi(A)$. This shows that $\varphi(A)$ is \crich{}.

\ref{item:density} \ Suppose $\dstar_{\semistimes}(A) \geq \denlet$ and $\dstar_{\semittimes}(\varphi(S)) \geq \beta$. We will show that $\dstar_{\semittimes}(\varphi(A)) \allowbreak \geq \denlet \beta$ with the help of Theorem \ref{thm:charofdensityifsfcplus}. Let $\multif$ be a finite multiset in $T$ with support $F$. There exists $z \in T$ such that $|\multif \cap \varphi(S) z^{-1}| \geq \be |\multif|$. Define the multiset $\multif'$ with support $F' = Fz \cap \varphi(S)$ by
\[\multif'(f') = \sum_{\substack{f \in F \cap \varphi(S)z^{-1} \\ fz = f'}} \multif(f).\]
Note that
\[|\multif'| = \sum_{f' \in F'} \sum_{\substack{f \in F \cap \varphi(S)z^{-1} \\ fz = f'}} \multif(f) = \sum_{f \in F \cap \varphi(S) z^{-1}} \multif(f) = |\multif \cap \varphi(S)z^{-1}| \geq \beta |\multif|.\]
Since $F' \subseteq \varphi(S)$, there exists an $H \subseteq S$ such that $|H| = |F'|$ and $\varphi(H) = F'$. Define the multiset $\multih$ with support $H$ by $\multih(h) = \multif'(\varphi(h))$, and note that $|\multih| = |\multif'|$. Since $\multih$ is a multiset in $S$, there exists an $s \in S$ such that $|\multih \cap A s^{-1}| \geq \denlet |\multih|$. Let $t = \varphi(s)$. We see that
\begin{align*}
\denlet \be |\multif| \leq |\multih \cap As^{-1}| &= \sum_{h \in H} \multih(h) \one_A(hs)\\
&\leq \sum_{f' \in F'} \multif'(f') \one_{\varphi(A)}(f' t)\\
&= \sum_{f' \in F'} \sum_{\substack{f \in F \cap \varphi(S)z^{-1} \\ fz = f'}} \multif(f) \one_{\varphi(A)t^{-1}}(fz)\\
&= \sum_{f \in F \cap \varphi(S) z^{-1}} \multif(f) \one_{\varphi(A)t^{-1}z^{-1}}(f) \leq |\multif \cap \varphi(A)(zt)^{-1}|.
\end{align*}
Since $\multif$ was arbitrary, this shows that $\dstar_{\semittimes}(\varphi(A)) \geq \denlet \beta$.
\end{proof}

\begin{corollary}\label{cor:bigandbigimpliesbig}
Let $\semisring$ be a semiring, $A \subseteq S$, $s \in S$, and $r \in \N$.
\begin{enumerate}[label=(\Roman*)]
\item \label{item:centralpluspsimpliescentral} If $A$ is central in $\semisplus$ and $sS$ is piecewise syndetic in $\semisplus$, then $sA$ is central in $\semisplus$.
\item[(I\,$'$)] If $A$ is piecewise syndetic in $\semisplus$ and $sS$ is piecewise syndetic in $\semisplus$, then $sA$ is piecewise syndetic in $\semisplus$.
\item If $A$ is $\ip_r \big/ \ipnaught \big/ \ip$ in $\semisplus$, then $sA$ is $\ip_r \big/ \ipnaught \big/ \ip$ in $\semisplus$, respectively.
\item If $A$ and $sS$ are $\ipnaught^*$ in $\semisplus$, then $sA$ is $\ipnaught^*$ in $\semisplus$.
\item Suppose $\semistimes$ and $\semittimes$ are commutative. If $A$ is \crich{} in $\semisplus$ and $sS$ is $\ipnaught^*$ in $\semisplus$, then $sA$ is \crich{} in $\semisplus$.
\item If $\dstar_{\semisplus}(A) > 0$ and $\dstar_{\semisplus}(sS) > 0$, then $\dstar_{\semisplus}(sA) > 0$.
\end{enumerate}
All of these results hold with $sA$ and $sS$ replaced by $As$ and $Ss$, respectively.
\end{corollary}

\begin{proof}
Each statement follows immediately from Lemma \ref{lem:bigandbigimpliesbig} by taking $\varphi$ to be dilation on the left (or right) by $s$. 
\end{proof}

Thus, in a semiring $\semisring$, an additively large subset $A \subseteq S$ remains additively large under left dilation by an element $s \in S$ when the principal right ideal $sS$ is sufficiently additively large. The semirings $\N$, $\Z$, $\matdz$, and $\F_p[x]$ all have the property that (non-zero) principal multiplicative ideals are additively $\ipnaught^*$.  An example of a semiring which does not have this property is $\N[x]$, for $x\N[x]$ is not even additively \crich{}.

\section{Preamble to Sections \ref{sec:posexamplea} through \ref{sec:counterexamples}}\label{sec:preamble}

In Sections \ref{sec:posexamplea} through \ref{sec:counterexamples}, we explain the extent to which multiplicative largeness begets additive largeness and vice versa. More precisely, given a semiring $\semisring$ and two classes of largeness $\arbclass$ and $\mathcal{Y}$ from Section \ref{sec:defs} (for example, the class of syndetic sets or sets of positive upper Banach density), we will show that
\begin{align}\label{eqn:desiredcontainment}\arbclass\semistimes \subseteq \mathcal{Y}\semisplus\end{align}
or provide an example to the contrary. We will explain how reverse inclusions -- classes of additive largeness contained in classes of multiplicative largeness -- follow from (\ref{eqn:desiredcontainment}) by considering complements and dual classes.

Sections \ref{sec:posexamplea} through \ref{sec:posexamplec} each highlight one such inclusion and include a number of related applications, while counterexamples to the other possible inclusions are given in Section \ref{sec:counterexamples}. A summary of these results is contained in Figure \ref{fig:implicationdiagram}. The containments depicted in Figure \ref{fig:containmentdiagram} are indicated by solid lines, the positive results are indicated by solid arrows, and the negative results -- those implications which are, in general, false -- are indicated by dashed arrows.

For example, in Section \ref{sec:posexamplea}, we will show that under some mild conditions on the semiring $\semisring$ and a multiplicative subsemigroup $\semirtimes$ of $\semistimes$,
\begin{align}\label{eqn:syndeticimpliescentral}\syndetic \semirtimes \subseteq \central \semisplus,\end{align}
which means that if a set is multiplicatively syndetic in $R$, then it is additively central in $S$. This is a generalization of \cite[Lemma 5.11]{BergelsonSurveytwoten}, which gives the containment in (\ref{eqn:syndeticimpliescentral}) in the case that $R = S = \N$.

Note that (\ref{eqn:syndeticimpliescentral}) differs from (\ref{eqn:desiredcontainment}) by the appearance of $\semirtimes$, a multiplicative subsemigroup of $\semistimes$. The following examples serve to illustrate the need for the appearance of $\semirtimes$ when proving such containments in semirings more general than $\seminring$.
\begin{enumerate}[label=(\Roman*)]
\item \label{item:problemone} In $\semizring$, a set is multiplicatively syndetic if and only if it contains zero. The set $\{0\}$, however, is not additively central (it is not even additively \crich{}).
\item \label{item:problemtwo} In $(\N[x],+,\cdot)$, the set $x\N[x]$ is multiplicatively syndetic but is not additively central (it is not even additively \crich{}).
\end{enumerate}
This shows that $\syndetic \semistimes \not\subseteq \central \semisplus$, even for the ring $\semizring$. The problem in \ref{item:problemone} stems from the existence of multiplicative zeroes or zero divisors, while the problem in \ref{item:problemtwo} stems from the fact that there are additively small principal multiplicative right ideals.

Invoking a multiplicative subsemigroup $\semirtimes$ of $\semistimes$ that avoids these problems allows us to salvage statements of the form in (\ref{eqn:desiredcontainment}). Here, for example, is Theorem \ref{thm:maintheorema}, one of the main theorems from Section \ref{sec:posexamplea}:
\begin{quote}
\emph{Let $\semisring$ be a semiring. Suppose that $\semirtimes$ is a subsemigroup of $\semistimes$ which is additively large in the following ways: \underline{\smash{$R$ is central}} \underline{\smash{in $\semisplus$}}, and \underline{\smash{for all $r \in R$, the set $rS$ is piecewise syndetic in $(S,$}} \underline{\smash{$+)$}}. If $A \subseteq R$ is multiplicatively \underline{\smash{syndetic}} in $R$, then $A$ is additively \underline{\smash{central}} in $S$.}
\end{quote}
The first two underlined statements are assumptions regarding the additive largeness of $R$ in $S$. Since the theorem asserts that multiplicatively large subsets of $R$ are additively large in $S$, it is necessary to assume that $R$ is itself additively large in $S$. The main theorems in the following sections take this form, where the underlined portions change depending on the setting.

\begin{figure}[ht]
\centering
\begin{tikzpicture}[>=triangle 60]
  \matrix[matrix of math nodes,column sep={20pt,between origins},row sep={15pt,between origins}](m)
  {
    & & & & & & & & |[name=arichstar]|\crichclass_\Endset^* & & & & & & \\
    \\
    & & & & & & & & & & & & & & \\
    & & & & & & & & |[name=densitystar]|\density^* & & & & & & \\
    & & |[name=ipnaughtstar]|\ipnaughtclass^* \\
    \\
    & & & & |[name=ipstar]|\ipclass^* & & & & |[name=pwsstar]|\pws^* \\
    \\
    & & & & & & |[name=centralstar]|\central^* & & & & |[name=thick]| \thick \\
    \\
    & & & & |[name=syndetic]|\syndetic & & & & |[name=central]|\central \\
    \\
    & & & & & & |[name=pws]| \pws & & & & |[name=ip]|\ipclass\\
    \\
    & & & & & & & & & & & & |[name=ipnaught]|\ipnaughtclass \\
    & & & & & & |[name=density]|\density \\
    & & & & & & & & & & & & & & \\
    \\
    & & & & & & |[name=arich]|\crichclass_\Endset \\
 };

   \draw[line width=0.05mm,-] (arichstar) edge (densitystar)
            (ipnaughtstar) edge (ipstar)
            (ipstar) edge (centralstar)
            (centralstar) edge (central)
            (central) edge (ip)
            (ip) edge (ipnaught)
            (densitystar) edge (pwsstar)
            (pwsstar) edge (thick)
            (pwsstar) edge (centralstar)
            (centralstar) edge (syndetic)
            (syndetic) edge (pws)
            (thick) edge (central)
            (central) edge (pws)
            (pws) edge (density)
            (density) edge (arich)
  ;
  
     \draw[-angle 90,dashed] (arichstar) edge[bend left=20] node[right] {\small Sec. \ref{sec:negexamplea}} (thick)
     				  (arichstar) edge[bend right=5] (syndetic)
				      (ipnaughtstar) edge (syndetic)
				      (ipnaughtstar) edge[bend left=2] (thick)
				      (thick) edge[bend left=10] (density)
				      (thick) edge node[right] {\small Sec. \ref{sec:negexampleb}} (ip)
				      (ip) edge[bend left=30] node[xshift=0.6cm, yshift=0.2cm] {\small Sec. \ref{sec:negexamplec}} (ipnaught)
				      (ip) edge (arich)
				      (density) edge node[xshift=1.0cm, yshift=-0.2cm] {\small Sec. \ref{sec:negexampled}} (ipnaught)
				      (arich) edge[in=310,out=230,distance=10mm,loop,->,>=angle 90] node[right, near end] {\small Sec. \ref{sec:negexamplee}} (arich)

  ;
  
       \draw[line width=0.35mm,->] (syndetic) edge node[above] {\small Sec. \ref{sec:posexamplea}} (central)
       								(pws) edge node[above] {\small \ref{sec:posexampleb}} (ipnaught)
									(density) edge[bend right=30] node[left] {\small Sec. \ref{sec:posexamplec}} (arich)

  ;
\end{tikzpicture}
\caption{Containment amongst additive and multiplicative classes of largeness in a semiring $(S,+,\cdot)$. A solid line $\mathcal{X}$ --- $\mathcal{Y}$ with $\mathcal{X}$ positioned above $\mathcal{Y}$ indicates that $\mathcal{X}(S,+) \subseteq \mathcal{Y}(S,+)$ and $\mathcal{X}(S,\cdot) \subseteq \mathcal{Y}(S,\cdot)$.  A solid arrow $\mathcal{X} \to \mathcal{Y}$ indicates that $\mathcal{X}(R,\cdot) \subseteq \mathcal{Y}(S,+)$ for suitable subsemigroups $(R,\cdot)$ of $(S,\cdot)$. A dashed arrow $\mathcal{X} \dashrightarrow \mathcal{Y}$ indicates that $\mathcal{X}(\N,\cdot) \not\subseteq \mathcal{Y}(\N,+)$.  See Definition \ref{def:combrichuptoe} for the definition of the class $\crichclass_\Endset$.}
\label{fig:implicationdiagram}
\end{figure}

\section{Multiplicative syndeticity implies additive centrality}\label{sec:posexamplea}

It was shown in \cite[Lemma 5.11]{BergelsonSurveytwoten} that multiplicatively syndetic subsets of $\N$ are additively central. In this section, we extend this result to semirings and prove related results and applications.

The following theorem says that in suitable semirings, \emph{multiplicative syndeticity implies additive centrality}.

\begin{theorem}\label{thm:maintheorema}
Let $\semisring$ be a semiring. Suppose that $\semirtimes$ is a subsemigroup of $\semistimes$ which is additively large in the following ways:
\begin{itemize}
\item $R$ is central in $\semisplus$, and
\item for all $r \in R$, the set $rS$ is piecewise syndetic in $\semisplus$.
\end{itemize}
If $A \subseteq R$ is syndetic in $\semirtimes$, then $A$ is central in $\semisplus$.
\end{theorem}

\begin{proof}
Let $A \subseteq R$ be syndetic in $\semirtimes$. There exist $r_1, \ldots, r_k \in R$ for which $R = \cup_{i=1}^k r_i^{-1} A$, where each $r_i^{-1}A$ is computed in $R$. When computed in $S$, we have that $R \subseteq \cup_{i=1}^k r_i^{-1} A$. Since $R$ is central in $\semisplus$ and the class $\central \semisplus$ is partition regular (Remark \ref{rmk:partitionregularityofipnaughtandcentral}), there exists an $i \in \{1, \ldots, k\}$ such that $r_i^{-1} A$ is central in $\semisplus$. By Corollary \ref{cor:bigandbigimpliesbig} \ref{item:central}, since $r_i S$ is piecewise syndetic in $\semisplus$, the set $r_i r_i^{-1} A$ is central in $\semisplus$. Since $r_i r_i^{-1} A \subseteq A$, this implies that $A$ is central in $\semisplus$.
\end{proof}

The following is an example application of this theorem; recall the definition of multiplicatively even and odd integers in Examples \ref{examples:homomorphisms} \ref{item:multeven}. We present another application of Theorem \ref{thm:maintheorema} to division rings in the following subsection.

\begin{corollary}
The set 
\[A = \big\{ M \in \matdz \ \big| \ \det(M) \text{ a multiplicatively odd integer} \big\}\]
is central in $(\matdz, +)$.
\end{corollary}

\begin{proof}
Let $R = \gldq \cap \matdz$, and note that $A \subseteq R$. The corollary follows from Theorem \ref{thm:maintheorema} if we verify that 1) $A$ is syndetic in $\semirtimes$, 2) $R$ is central in $(\matdz,+)$, and 3) for all $W \in R$, the set $W \matdz$ is piecewise syndetic in $(\matdz,+)$. To see 1), let $W \in R$ be such that $\det(W)$ is multiplicatively odd, and note that $R = A \cup W^{-1}A$. To show 2), it suffices to show that $R$ is thick in $(\matdz,+)$: for $F \subseteq \matdz$ finite, there exists $\lambda \in \N$ such that every matrix in the set $\lambda \text{Id} + F$ is invertible. Finally, note that for all $W \in R$, the set $W \matdz$ is a finite-index subgroup of $(\matdz,+)$. This implies that $W \matdz$ is syndetic, hence piecewise syndetic, in $(\matdz,+)$, proving 3).
\end{proof}

The conclusion of the previous corollary holds with ``odd'' replaced by ``even'' with the same proof. Specifying to the case $d=1$, we conclude that the sets of multiplicatively even and odd integers are each additively central in $\semizplus$.

The following corollary is a dual form of Theorem \ref{thm:maintheorema} which says that \emph{sets which have non-empty intersection with every additively central set are multiplicatively thick}.

\begin{corollary}
Let $\semisring$ be a semiring. Suppose that $\semirtimes$ is a subsemigroup of $\semistimes$ which is additively large in the following ways:
\begin{itemize}
\item $R$ is central in $\semisplus$, and
\item for all $r \in R$, the set $rS$ is piecewise syndetic in $\semisplus$.
\end{itemize}
If $A \subseteq S$ is $\central^*$ in $\semisplus$, then $A \cap R$ is thick in $\semirtimes$.
\end{corollary}

\subsection{Multiplicatively syndetic subgroups of fields}\label{sec:syndeticinfields} For subgroups of the multiplicative group of an infinite division ring, multiplicative syndeticity implies considerably more than additive centrality. Let $(\F,+,\tone)$ be a division ring (a unital ring in which every non-zero element has a multiplicative inverse), and let $\fstar = \F \setminus \{0\}$.\footnote{The notation $\fstar$ is not be confused with the star notation for the dual class from Definition \ref{def:dualclasses}.} Note that a subgroup of the group $(\fstar,\cdot)$ is syndetic if and only if it is of finite index. With the help of Theorem \ref{thm:maintheorema} and Corollary \ref{cor:ipszemfordensityinsemigroups}, we prove in the following theorem that syndetic multiplicative subgroups of $(\fstar,\cdot)$ and their cosets are thick in $(\F,+)$; cf. \cite[Remark 5.23]{Bcombanddioph}.

\begin{theorem}\label{thm:multsyndgivesaddthickinfields}
Let $(\F,+,\tone)$ be an infinite division ring and $\Gamma$ be a finite index subgroup of $(\fstar,\tone)$. For all $g \in \fstar$, the sets $g \Gamma$ and $\Gamma g$ are thick in $(\F,+)$; in particular, $\Gamma$ is thick in $(\F,+)$.
\end{theorem}

\begin{proof}
Since $\F$ is infinite, the set $\{0\}$ is not syndetic in $(\F,+)$. It follows that $\fstar$ is thick, hence central, in $(\F,+)$. Since $\Gamma$ is syndetic in $(\fstar, \cdot)$, Theorem \ref{thm:maintheorema} gives that $\Gamma$ is central in $(\F,+)$. It follows that $\Gamma$ is an $\ip$ set in $(\F,+)$ and that $\dstar_{(\F,+)}(\Gamma) > 0$.

Let $g \in \fstar$. We will show that $\Gamma g$ is thick in $(\F,+)$. Let $F \in \finitesubsets \F$. Since $\Gamma$ is an $\ip$ set in $(\F,+)$, Corollary \ref{cor:bigandbigimpliesbig} \ref{item:ipr} gives that $g^{-1} \Gamma$ is an $\ip$ set in $(\F,+)$. Since $\dstar_{(\F,+)}(\Gamma) > 0$, Corollary \ref{cor:ipszemfordensityinsemigroups} gives that there exists $g^{-1} \gamma \in g^{-1} \Gamma$ and $z \in \F$ for which $z + F g^{-1}\gamma \subseteq \Gamma$. Multiplying on the right by $\gamma^{-1} g$, we see that $z\gamma^{-1} g + F \subseteq \Gamma \gamma^{-1} g = \Gamma g$. Since $F \in \finitesubsets \F$ was arbitrary, this shows that $\Gamma g$ is additively thick.

That the set $g \Gamma$ is thick in $(\F,+)$ follows in the same way from Corollary \ref{cor:ipszemfordensityinsemigroups} by using that $\Gamma g^{-1}$ is an $\ip$ set in $(\F,+)$.
\end{proof}

A subset of a left amenable semigroup is thick if and only if it has upper Banach density equal to 1 (see \cite[Prop. 1.21]{patersonbook}). Thus, the proof of Theorem \ref{thm:multsyndgivesaddthickinfields} shows how to use the IP Szemer\'edi theorem to improve the statement ``$\Gamma$ is an $\ip$ set in $(\F,+)$ and $\dstar_{(\F,+)}(\Gamma) > 0$'' to the statement ``$\dstar_{(\F,+)}(\Gamma) = 1$.''  In fact, we can conclude from Theorem \ref{thm:multsyndgivesaddthickinfields} that the left and right cosets of $\Gamma$ all have additive upper Banach density equal to $1$.

While Theorem \ref{thm:multsyndgivesaddthickinfields} is presented here as an example application of Theorem \ref{thm:maintheorema} and the IP Szemer\'edi theorem, it can be proved with less sophisticated tools. We will demonstrate how to derive a related result in a more elementary way in the proof of Corollary \ref{cor:multevenrationalsthick} below.

The multiplicatively even positive integers (recall Examples \ref{examples:homomorphisms} \ref{item:multeven}) are multiplicatively syndetic in $\N$, but it is an open problem to determine whether or not they are additively thick.\footnote{This is a specific instance of the problem of determining whether any finite word with letters from the set $\{-1,1\}$ is witnessed infinitely often as consecutive values of the Liouville function $\lambda: \N \to \{-1,1\}$, the completely multiplicative function taking the value $-1$ at every prime; see \cite{patterninliouville} for more details.} Theorem \ref{thm:multsyndgivesaddthickinfields} yields the following positive answer to the analogous problem in $\Q$.

\begin{corollary}\label{cor:multevenrationalsthick}
Let $\varphi: \Q \setminus \{0\} \to \{-1,1\}$ be non-constant and completely multiplicative. The sets $\varphi^{-1}(\{1\})$ and $\varphi^{-1}(\{-1\})$ are thick in $(\Q,+)$. In particular, the multiplicatively even and odd rational numbers (recall Examples \ref{examples:homomorphisms} \ref{item:multeven}) are additively thick.
\end{corollary}

\begin{proof}
The corollary follows immediately from Theorem \ref{thm:multsyndgivesaddthickinfields} by verifying that the set $\varphi^{-1}(\{1\})$ is a subgroup of $(\Q \setminus \{0\}, \cdot)$ of finite index. We provide now an alternate proof of this corollary using Hindman's theorem \cite{hindmanoriginal} and the IP van der Waerden theorem \cite[Section 3]{fwdynamics}.

Let $A=\varphi^{-1}(\{1\})$ and $F \in \finitesubsets \Q$. Let $m \in \N$ so that $\varphi(m) = 1$ and $mF \subseteq \Z$. By Hindman's theorem, either $A \cap \N$ or $\N \setminus A$ is an additive $\ip$ set in $\N$; multiplying by an $n \in \N$ for which $\varphi(n) = -1$, we see that both sets are additive $\ip$ sets. By the IP van der Waerden theorem, either
\[\big\{ d \in \N \ \big| \ \exists x \in \N, \ x + dmF \subseteq A \cap \N \big\}\]
is an $\ip^*$ set, or the same set with $\N \setminus A$ in place of $A \cap \N$ is an $\ip^*$ set. In the first case, let $d$ be from the set such that $\varphi(d) = 1$; in the second, choose $d$ so  $\varphi(d) = -1$. Either way, on dividing, $x \big/ (dm) + F \subseteq A$.
\end{proof}

Taking advantage of the massivity of the set of returns in the $\ip_0^*$ formulation of Szemer\'edi's theorem, Theorem \ref{thm:ipszemeredicomb}, the proof of Theorem \ref{thm:multsyndgivesaddthickinfields} can be made to work in the finite field setting. The following theorem is a generalization of \cite[Theorem 1.1]{Bcombanddioph}, which is a generalization of Schur's result \cite{schuroriginal} that Fermat's Last Theorem fails in finite fields.

\begin{theorem}\label{thm:finitarymultsyndgivesaddthickinfields}
For all $n, k \in \N$, there exists an $N=N(n,k) \in \N$ with the following property. For all finite fields $\F$ with $|\F| \geq N$ and all $F \subseteq \F$ with $|F| \leq n$, there exists a non-zero $x \in \F$ for which every element of $x^k + F$ is a non-zero $k^{\text{th}}$ power.
\end{theorem}

\begin{proof}
Using the partition regularity of the class $\ipnaughtclass$ (see Remark \ref{rmk:partitionregularityofipnaughtandcentral}), let $R \in \N$ be large enough that one cell of any partition of an $\ip_R$ set into $k$ parts is guaranteed to be an $\ip_r$ set, where $r = r(n+1,1/(2k))$ is from Corollary \ref{cor:ipszemfordensityinsemigroups}. Let $N > 2^R$.

Let $\F$ be a finite field with $|\F| \geq N$. To see that $\fstar$ is an $\ip_R$ set, let $x_1 \in \fstar$. Let $x_2 \in \fstar \setminus (-\{x_1\})$, and note $\finitesums(x_1,x_2) \subseteq \fstar$. Choose $x_3 \in \fstar \setminus (-\finitesums(x_1,x_2))$, and note $\finitesums(x_1,x_2,x_3) \subseteq \fstar$. Repeat this process to build and $\ip_R$ set in $\fstar$.

Let $\Gamma$ be the subgroup of $(\fstar,\cdot)$ consisting of non-zero $k^{\text{th}}$ powers. Since $\Gamma$ is of index at most $k$, by our choice of $R$, there exists a $g \in \F$ such that $g\Gamma$ is an $\ip_r$ set in $(\F,+)$. By Corollary \ref{cor:bigandbigimpliesbig} \ref{item:ipr}, $\Gamma$ is an $\ip_r$ set in $(\F,+)$.

Let $F \subseteq \F$ with $|F| \leq n$. Using Corollary \ref{cor:ipszemfordensityinsemigroups} in the same way as it was applied in the proof of Theorem \ref{thm:multsyndgivesaddthickinfields}, there exists a $z \in \F$ for which $z + F \subseteq \Gamma$. Consider the set $\{z\} \cup (z+ F)$. Using Corollary \ref{cor:ipszemfordensityinsemigroups} in the same way, there exists a $y \in \F$ for which $\{y+z\} \cup (y+z+F) \subseteq \Gamma$. Since $y+z \in \Gamma$, there exists a non-zero $x \in \F$ for which $y+z = x^k$. It follows that $x^k + F \subseteq \Gamma$. 
\end{proof}

It follows in the same way as in the proof of Theorem \ref{thm:multsyndgivesaddthickinfields} that any coset of the subgroup of non-zero $k^{\text{th}}$ powers satisfies the conclusion in Theorem \ref{thm:finitarymultsyndgivesaddthickinfields}.

\subsection{Lower density and a class of syndetic sets in \texorpdfstring{$\semintimes$}{(N,.)}}\label{sec:newclassofmultsynd} According to Theorem \ref{thm:maintheorema}, multiplicatively syndetic subsets of suitable semirings are additively central. In particular, they have positive additive upper Banach density. In fact, more is true: multiplicatively syndetic sets have positive lower density along \folner{} sequences in $\semisplus$ (see the expression in (\ref{eqn:defoflowerdensity})) which are well behaved with respect to the multiplication.

\begin{lemma}\label{lem:multsyndimplieslowerdensity}
Let $\semisring$ be a countable semiring with $\semistimes$ left cancellative. Suppose $(F_n)_{n \in \N}$ is a \folner{} sequence for $\semisplus$ with the property that for all $s_1, \ldots, s_k \in S$, $\liminf_{n \to \infty} | \cap_{i=1}^k s_i^{-1}F_n | \big/ |F_n| > 0$. If $A \subseteq \semistimes$ is syndetic, then
\begin{align}\label{eqn:defoflowerdensity}\underline{d}_{(F_n)}(A) \defeq \liminf_{n \to \infty} \frac{|A \cap F_n|}{|F_n|} > 0.\end{align}
\end{lemma}

\begin{proof}
Suppose $\underline{d}_{(F_n)}(A) = 0$. To show that $A$ is not syndetic in $\semistimes$, we need only to show that for any $s_1, \ldots, s_k \in S$, $\cup_{i=1}^k s_i^{-1} A \neq S$.

Let $s_1, \ldots, s_k \in S$, $G_n = \cap_{i=1}^k s_i^{-1}F_n$, and $0 < \eps < \liminf_{n \to \infty} | G_n | \big/ |F_n|$. Choose $n \in \N$ such that $|A \cap F_n | < \eps |F_n| \big / k$ and $|G_n| > \eps |F_n|$. By left cancellativity, $|s_i^{-1} A \cap G_n| \leq |A \cap F_n| < \eps |F_n| \big/ k$. Therefore, $|\cup_{i=1}^k s_i^{-1} A \cap G_n| < \eps |F_n|$. Since $|G_n| > \eps |F_n|$, this shows that $\cup_{i=1}^k s_i^{-1} A \neq S$.
\end{proof}

It is quick to check that the initial \folner{} sequence $(F_n = \{1, \ldots, n\})_{n \in \N}$ in $(\N,+)$ satisfies the property in the statement of Lemma \ref{lem:multsyndimplieslowerdensity}.  In this case, the lemma says that a multiplicatively syndetic subset of $\N$ has positive additive lower asymptotic density.

If $(F_n)_{n\in\N}$ is a \folner{} sequence for $\semistimes$, then for all $s_1, \ldots, s_k \in S$, $ | \cap_{i=1}^k s_i^{-1}F_n | \big/ |F_n| \to 1$ as $n \to \infty$. Thus a consequence of Lemma \ref{lem:multsyndimplieslowerdensity} is that any multiplicatively syndetic set has positive lower asymptotic density along sequences $(F_n)_{n\in\N}$ which are \folner{} for both $\semisplus$ and $\semistimes$. While such \folner{} sequences do not exist in $\N$,\footnote{Were $(F_N)_{N \in \N} \subseteq \finitesubsets{\N}$ a \folner{} sequence for both $\seminplus$ and $\semintimes$, the density function $\overline{d}_{(F_N)}(A) \defeq \limsup_{N \to \infty} |F_N|^{-1} |A \cap F_N|$ would, by multiplicative invariance, satisfy $\overline{d}_{(F_N)}(2 \N) = 1$.  This would imply that the set $2\N$ has additive upper Banach density equal to 1, a contradiction to, for example, Corollary \ref{cor:densitycharforcancelsemigroups}.} they do exist in countable fields \cite[Proposition 2.4 and Remark 6.2]{bergelson_moreira_2017}.

It follows from Theorem \ref{thm:charofdensityifsfcplus} that a subset of a countable, commutative semigroup $\semisplus$ is syndetic if and only if for every \folner{} sequence $(F_n)_{n\in\N}$, $\underline{d}_{(F_n)}(A) > 0$. The conclusion of Lemma \ref{lem:multsyndimplieslowerdensity} can thus be interpreted as being a weaker form of additive syndeticity for the set $A$: while $A$ may not be additively syndetic, it does have positive asymptotic lower density along certain \folner{} sequences.

The following is an immediate corollary of Lemma \ref{lem:multsyndimplieslowerdensity} for subsets of the positive integers. The contrapositive statement is related to \cite[Theorem 6.3]{bmmpactions}, where it is shown that if $\overline{d}_{\seminplus}(A) = 1$, then $\overline{d}_{\seminplus}(A/n) = 1$.

\begin{corollary}\label{cor:syndeticgiveslowerdensity}
Let $A \subseteq \N$. If $A$ is syndetic in $\semintimes$, then
\[ \liminf_{n \to \infty} \frac{|A \cap \{1, \ldots, n\}|}{n} > 0.\]
\end{corollary}

We conclude this section by describing a new class of multiplicatively syndetic subsets of $\N$. These sets provide applications for Theorem \ref{thm:maintheorema} and Corollary \ref{cor:syndeticgiveslowerdensity} and will appear again in applications in Section \ref{sec:posexampleb}. We begin by demonstrating uniformity in certain exponential sums.

\begin{lemma}\label{lem:uniformconvergenceofsums}
Let $c \in \R \setminus \Q$ and $d \in \N$. For all $\eps > 0$, there exists an $N_0 \in \N$ such that for all $N > N_0$, all $M \in \Z$, and all $f \in \R[x]$ of degree $d$ with leading coefficient $c$,
\[\left| \frac {1}{N} \sum_{n=M+1}^{M+N} e^{2 \pi i f(n)} \right| < \eps.\]
\end{lemma}

\begin{proof}
Since $c$ is irrational, Dirichlet's approximation theorem gives that there exist sequences $(N_i)_{i=1}^\infty \subseteq \N$ and $(a_i)_{i=1}^\infty \subseteq \Z$ with $N_i \to \infty$, $(N_i,a_i) =1$, and $| c - a_i / N_i | \leq N_i^{-2}$. We now apply Weyl's Inequality \cite[Lemma 3.1]{davenport} with ``$\eps$'' as $2^{-d}$: there exists a constant $C = C(d)$ such that for all $M \in \Z$, for all $i \in \N$, and for all $f \in \R[x]$ with degree $d$ with leading coefficient $c$,
\[\left| \sum_{n=M+1}^{M+N_i} e^{2 \pi i f(n)} \right| \leq C N_i^{1-1/2^d}.\]
This estimate is uniform in $M$ because the leading coefficient of a polynomial remains invariant under translating the argument.

Let $\eps > 0$.  If $i\in \N$ is so large that $C N_i^{-1/2^d} < \eps / 2$, then for all $N > 2 N_i \eps^{-1}$, all $M \in \Z$, and all $f \in \R[x]$ with degree $d$ with leading coefficient $c$,
\begin{align*}\left| \sum_{n=M+1}^{M+N} e^{2 \pi i f(n)} \right| &\leq \sum_{h=0}^{\lfloor N / N_i \rfloor -1} \left| \sum_{n=M+h N_i + 1}^{M+(h+1)N_i} e^{2 \pi i f(n)} \right| + \left| \sum_{n=M+\lfloor N / N_i \rfloor N_i + 1}^{M+N} e^{2 \pi i f(n)} \right|\\
&\leq \left\lfloor \frac{N}{N_i} \right\rfloor C N_i^{1-1/2^d} + N_i < N \eps.
\end{align*}
This shows the uniformity in the desired estimate.
\end{proof}

In the results that follow, denote by $\{x\}$ the fractional part of $x \in \R$.

\begin{lemma}\label{lem:uniformepsilondensnessinbeta}
Let $p \in \R[x]$ be a polynomial of degree $d \geq 1$ with leading coefficient $c$, and suppose $\irratnum \in \R$ is such that $c \irratnum^d \in \R \setminus \Q$. The sequence $\big(p(n\irratnum + \be) \big)_{n \in \N}$ is \emph{$\eps$-dense modulo 1 uniformly in $\beta$}: for all $\eps > 0$, there exists an $N \in \N$ such that for all $\be \in \R$, the set $\big\{p(n\irratnum + \be) \big\}_{n=1}^N$ is $\eps$-dense modulo $1$.
\end{lemma}

\begin{proof}
Let $\eps > 0$. We will prove that there exists an $N \in \N$ such that for all $\beta \in \R$,
\[D_N \defeq \sup_{0 \leq a < b \leq 1} \left| \frac{|\{ 1 \leq n \leq N \ | \ \{p(n\irratnum + \be)\} \in [a,b) \}|}{N} - (b-a)\right| < \eps.\]
By an inequality of LeVeque \cite[Ch. 2, Theorem 2.4]{kuipersbook},
\[D_N \leq \left( \frac{6}{\pi^2} \sum_{h=1}^\infty \frac 1{h^2} \left| \frac 1N \sum_{n=1}^N e^{2\pi i h p(\irratnum n + \be)} \right|^2 \right)^{1/3}.\]
By truncating the sum on $h$ at some sufficiently large $H \in \N$, it suffices to furnish $N \in \N$ such that for all $1 \leq h \leq H$ and all $\be \in \R$,
\begin{align}\label{eqn:targetexpsumest}\left| \frac 1N \sum_{n=1}^N e^{2 \pi i h p(\irratnum n + \be)} \right| < \sqrt{\frac{\eps^3}2}.\end{align}
Let $h \in \{1, \ldots, H\}$. Since the leading coefficient of the polynomial $h p(\irratnum n + \be)$ is irrational, Lemma \ref{lem:uniformconvergenceofsums} gives the existence of $N_h$ satisfying (\ref{eqn:targetexpsumest}) uniformly in $\beta$. Take $N = \max_{1 \leq h \leq H} N_h$.
\end{proof}

\begin{theorem}\label{thm:logpolyissyndetic}
Let $p \in \R[x]$ be a polynomial of degree $d \geq 1$ with leading coefficient $c$. Let $\varphi: \semintimes \to (\R,+)$ be a semigroup homomorphism, and suppose that there exists an $a \in \N$ for which $c \varphi(a)^d \in \R \setminus \Q$. For any interval $I \subseteq [0,1)$, the set
\begin{align}\label{eqn:multsyndsetone}A_I = \big\{n \in \N \ \big| \ \{p(\varphi(n)) \} \in I \big\}\end{align}
is multiplicatively syndetic in $\semintimes$.
\end{theorem}

\begin{proof}
Put $\eps = |I| \big/ 2$ and take $N$ from Lemma \ref{lem:uniformepsilondensnessinbeta} so that the sequence $\big( p(n \varphi(a) + \be) \big)_{n=1}^N$ is $\eps$-dense modulo 1 uniformly in $\beta$. Let $m \in \N$, and put $\beta = \varphi(m)$. There exists an $n \in \{1, \ldots, N\}$ for which $\{p\big(\varphi(m a^n) \big)\} = \{p(n \varphi(a) + \be)\} \in I$. Therefore, $m a^n \in A_I$. This shows that the set $F = \{a^1, \ldots, a^N\}$ has the property that for all $m \in \N$, $(Fm) \cap A_I \neq \emptyset$, proving that $A_I$ is multiplicatively syndetic.
\end{proof}

\begin{corollary}\label{cor:logismultsyndetic}
Let $p \in \R[x]$ be a non-constant polynomial, and let $I \subseteq [0,1)$ be an interval. The set
\begin{align}\label{eqn:multsyndsettwo}A_I = \big\{n \in \N \ \big| \ \{p(\log n) \} \in I \big\}\end{align}
is multiplicatively syndetic in $\semintimes$. If the leading coefficient of $p$ is rational, then for every sub-semigroup $\semirtimes$ of $\semintimes$, the set $A_I \cap R$ is multiplicatively syndetic in $\semirtimes$.
\end{corollary}

\begin{proof}
Denote by $c$ and $d$ the leading coefficient and degree of $p$, respectively. The Gelfond-Schneider theorem \cite[Section 3.2]{gelfondbook} gives that one of $c (\log 2)^d$ and $c (\log 3)^d$ is irrational (actually, transcendental). Now, we need only to apply Theorem \ref{thm:logpolyissyndetic} with $\varphi$ as $\log$ and $a \in \{2,3\}$ chosen so that $c (\log a)^d$ is irrational. The second statement follows by taking any $r \in R$, noting that $c\log(r)^d$ is irrational, and applying the same proof as in Theorem \ref{thm:logpolyissyndetic}.
\end{proof}

Note that by the monotonicity of the logarithm, the set in (\ref{eqn:multsyndsettwo}) is additively thick, hence additively central, and it is not hard to verify the conclusion of Corollary \ref{cor:syndeticgiveslowerdensity} by hand for this set, either. The relationship between additive thickness and the final conclusion in Corollary \ref{cor:logismultsyndetic} is discussed in Section \ref{sec:negexamplea}. These conclusions are not so clear for other homomorphisms $\varphi$ in Theorem \ref{thm:logpolyissyndetic}; in the following corollary, we consider $\nu_p$ and $\Omega$, the $p$-adic valuation and prime divisor counting function (with multiplicity) defined in Examples \ref{examples:homomorphisms}.

\begin{corollary}\label{cor:omegaismultsyndetic}
Let $p \in \R[x]$ be a non-constant polynomial with irrational leading coefficient, $I \subseteq [0,1)$ be an interval, and $q$ be prime. The sets
\begin{align}\label{eqn:multsyndsetthree}\big\{n \in \N \ \big| \ \{p(\nu_q(n)) \} \in I \big\} \text{ and } \big\{n \in \N \ \big| \ \{p(\Omega(n)) \} \in I \big\}\end{align}
are multiplicatively syndetic in $\semintimes$.
\end{corollary}

It follows from Theorem \ref{thm:multsyndgivesaddthickinfields} and Corollary \ref{cor:syndeticgiveslowerdensity} that the sets in (\ref{eqn:multsyndsetthree}) are additively central, hence piecewise syndetic and IP, and have positive asymptotic lower density in $\seminplus$. We will use the fact that these sets are additively $\ipnaught$ in an application to Diophantine approximation in Section \ref{sec:ipzerostarindynamics}.

We end this section by making note of another approach to the results in Corollary \ref{cor:omegaismultsyndetic} using the fact that the set of return times (a set of the form in (\ref{eqn:defofreturntimes})) of any point to a neighborhood of itself in a nilsystem\footnote{A nilsystem is a topological dynamical system of the form $(X,T)$ where $X$ is a compact homogeneous space of a nilpotent Lie group $G$ and $T$ is a translation of $X$ by an element of $G$.} is $\ipnaught^*$; see \cite[Theorem 0.2]{BLiprstarcharacterization}.\footnote{To avoid invoking the full strength of \cite[Theorem 0.2]{BLiprstarcharacterization} to prove Theorem \ref{thm:alternativemethod}, one may employ a strategy similar to the one used in the proof of \cite[Theorem 7.7]{BergelsonSurveytwoten}.} The following theorem generalizes \cite[Theorem 3.15]{bergelsonminimalidempotents}.

\begin{theorem}\label{thm:alternativemethod} 
Let $p_1, \ldots, p_k \in \R[x]$ be non-constant polynomials that are linearly independent in the following sense: for all $h_1, \ldots, h_k \in \Z$, not all zero, at least one of the non-constant coefficients of $\sum_{i=1}^k h_i p_i$ is irrational. Let $I_1, \ldots, I_k \subseteq [0,1)$ be sets that are open when $[0,1)$ is identified with the 1-torus. The set
\[A \defeq \big\{ n \in \Z \ \big| \ \text{for all } i \in \{1,\ldots,k\}, \ \{p_i(n)\} \in I_i \big\}\]
is, up to a translate, $\ip_0^*$ in $(\Z,+)$.  More precisely, the set $A$ is non-empty, and for $m \in \Z$, the set $A-m$ is $\ip_0^*$ if and only if $m \in A$.
\end{theorem}

\begin{proof}
Let $d_i$ be the degree of $p_i$ and let $\al_i \in \R \setminus \{0\}$ be its leading coefficient. (We do not assume that the $\al_i$'s are irrational.) Let $(X_i = \T^{d_i},T_i)$ be the skew-product
\[T_i(x_1,\ldots,x_d) = (x_1+d_i! \, \al_i, x_2 + x_1,\ldots,x_d+x_{d-1})\]
described in \cite[Chapter 1, Page 23]{furstenberg-book}.  There exists a point $x_{(i)} \in \T^{d_i}$ with the property that for all $n \in \Z$, the last coordinate of $T_i^n x_{(i)}$ is $\{p_i(n)\}$.

Let
\[(X,T) = (X_1 \times \cdots \times X_k,T_1 \times \cdots \times T_k),\]
and let $x = (x_{(1)},\ldots,x_{(k)}) \in X$. Note that $x$ is an element of the $(\sum_{i=1}^k d_i)$-dimensional torus, $X$. Let $\pi: X \to \T^k$ be the projection onto the last coordinates of each of the $X_i = \T^{d_i}$ factors of $X$, so that for all $n \in \Z$,
\[\pi (T^n x) = \big(p_1(n),\ldots,p_k(n) \big).\]
By our assumptions and Weyl's equidistribution theorem, the sequence $\big(\pi (T^n x)\big)_{n \in \Z }$ is dense in $\T^k$.

Let $X'$ be the orbit closure of $x$ under $T$. Since $(X,T)$ is distal, $(X',T)$ is minimal \cite[Theorem 3.2]{furstenbergdistalstructuretheory}, and since $(X,T)$ is a nilsystem, so too is $(X',T)$ \cite[Theorem 2.21]{leibman05}. Minimality means that for all $x \in X'$ and all non-empty, open sets $U \subseteq X'$, the set
\begin{align}\label{eqn:defofreturntimes}R(x,U) \defeq \{n \in \Z \ | \ T^n x \in U\}\end{align}
is non-empty. By \cite[Theorem 0.2]{BLiprstarcharacterization}, the fact that $(X',T)$ is a nilsystem implies that the set $R(x,U)$ is $\ip_0^*$ if and only if $x \in U$.\footnote{More precisely, \cite[Theorem 0.2]{BLiprstarcharacterization} gives that if $x \in U$, then the set $R(x,U)$ is $\ip_0^*$. Suppose $x \not\in U$, and let $V$ be an open neighborhood of $x$ disjoint from $U$. The same theorem gives that the set $R(x,V)$ is $\ip_0^*$. Since $U \cap V = \emptyset$, $R(x,U) \cap R(x,V) = \emptyset$, meaning that $R(x,U)$ is not an $\ip_0$ set. In particular, it is not an $\ip_0^*$ set.}

Let
\[U = X' \cap \big( (\mathbb{T}^{d_1-1} \times I_1) \times (\mathbb{T}^{d_2-1} \times I_2) \times \cdots \times (\mathbb{T}^{d_k-1} \times I_k) \big),\]
and note that $A = R(x,U)$. The set $U$ is open, and since $\pi( X') = \T^k$, it is non-empty. Since $(X',T)$ is minimal, the set $A$ is non-empty. Since $A-m = R(T^mx,U)$, this shows that the set $A-m$ is $\ip_0^*$ if and only if $T^mx \in U$, that is, if and only if $m \in A$.
\end{proof}

Since both $\nu_2$ and $\Omega$ are surjective homomorphisms, a slight modification to Lemma \ref{lem:surjectivehom} combined with Theorem \ref{thm:alternativemethod} shows that the sets in Corollary \ref{cor:omegaismultsyndetic}, namely $\nu_2^{-1}(X)$ and $\Omega^{-1}(X)$, can be dilated to become multiplicatively $\ipnaught^*$; in particular, both are multiplicatively syndetic. The full strength of Theorem \ref{thm:logpolyissyndetic} is still not easily recovered in this way since we do not require in that theorem that $\varphi$ be surjective.

\section{Multiplicative piecewise syndeticity implies additive \texorpdfstring{$\ipnaught$}{IP0}}\label{sec:posexampleb} 

Multiplicatively central sets in $\N$ were shown in \cite[Theorem 3.5]{berghind-onipsets} to be additively $\ipnaught$, and this result was generalized to the ring setting in \cite[Proposition 4.1]{BRcountablefields}. In this section, we strengthen both theorems, generalize them to semirings, and give some related results. In particular, we describe some rich sources of additively $\ipnaught^*$ sets from measure theoretic and topological dynamics.

The following theorem says that \emph{multiplicative piecewise syndeticity implies additive $\ipnaught$}.

\begin{theorem}\label{thm:multpwsimpliesaddipnaught}
Let $\semisring$ be a semiring. Suppose that $\semirtimes$ is a subsemigroup of $\semistimes$ which is additively large in the following way: $R$ is an $\ipnaught$ set in $\semisplus$. If $A \subseteq R$ is piecewise syndetic in $\semirtimes$, then $A$ is an $\ipnaught$ set in $\semisplus$.
\end{theorem}

\begin{proof}
Let $A \subseteq R$ be piecewise syndetic in $\semirtimes$. There exist $r_1, \ldots, r_k \in R$ such that $R' = \cup_{i=1}^k r_i^{-1}A$ (computed in $R$) is thick in $\semirtimes$. Since $R$ is an $\ipnaught$ set in $\semisplus$ and $R'$ is thick in $\semirtimes$, it is a consequence of Corollary \ref{cor:bigandbigimpliesbig} \ref{item:ipr} that $R'$ is an $\ipnaught$ set in $\semisplus$.

The class $\ipnaughtclass \semisplus$ is partition regular (Remark \ref{rmk:partitionregularityofipnaughtandcentral}). Since $R' \subseteq \cup_{i=1}^k r_i^{-1}A$ (where the right hand side is now computed in $S$) is an $\ipnaught$ set in $\semisplus$, there exists an $i$ for which $r_i^{-1}A$ is an $\ipnaught$ set in $\semisplus$. By Corollary \ref{cor:bigandbigimpliesbig} \ref{item:ipr}, this means $r_i r_i^{-1}A$ is an $\ipnaught$ set in $\semisplus$. The theorem follows since $r_i r_i^{-1}A \subseteq A$.
\end{proof}

Our proof of this theorem is similar to the proof of \cite[Theorem 1.3]{bbhsaddimpliesmult}, where it is shown that multiplicatively piecewise syndetic subsets of $\N$ contain arbitrarily long arithmetic progressions. We will show much more along the lines of Theorem \ref{thm:multpwsimpliesaddipnaught} in Section \ref{sec:posexamplec}, especially in Theorem \ref{thm:multdensitygivesaddpatterns}.

The following theorem is related to the previous one and says that \emph{multiplicative syndeticity implies additive $\ip$}. The proof of this theorem is nearly identical to the proof of Theorem \ref{thm:maintheorema}, so we omit it.

\begin{theorem}
Let $\semisring$ be a semiring. Suppose that $\semirtimes$ is a subsemigroup of $\semistimes$ which is additively large in the following way: $R$ is an $\ip$ set in $\semisplus$. If $A \subseteq R$ is syndetic in $\semirtimes$, then $A$ is an $\ip$ set in $\semisplus$.
\end{theorem}

The following is an immediate corollary to Theorem \ref{thm:multpwsimpliesaddipnaught} gotten by considering dual classes.

\begin{corollary}\label{cor:ipstarimpliesmultpsstar}
Let $\semisring$ be a semiring. Suppose that $\semirtimes$ is a subsemigroup of $\semistimes$ which is additively large in the following way: $R$ is an $\ipnaught$ set in $\semisplus$. If $A \subseteq S$ is an $\ip_0^*$ set in $\semisplus$, then $A \cap R$ is $\pws^*$ in $\semirtimes$.
\end{corollary}

It is interesting to consider what this corollary says in the case that there are multiple multiplicative structures atop a single additive semigroup which are all ``compatible'' with the addition. In this case, any additive $\ipnaught^*$ set is multiplicatively large with respect to \emph{all} of these multiplications.

\begin{example}
Let $A \subseteq \Z^2$ be additively $\ipnaught^*$. (Such sets arise in Theorems \ref{thm:ipszemeredicomb} or \ref{thm:bliprstarsemiring}, for example.) Each of the rings $\big\{\Z[\sqrt c] \ | \ c \in \Z \text{ not a square} \big\}$ induces a multiplication $\circledast_c$ on $\Z^2$ (under the usual identification of $x_1+x_2 \sqrt c$ with $(x_1,x_2)$) which makes $(\Z^2,+,\circledast_c)$ into a semiring. Corollary \ref{cor:ipstarimpliesmultpsstar} gives that $A \setminus \{0\}$ is multiplicatively $\pws^*$ with respect to each of the multiplications $\circledast_c$. The authors showed in \cite[Theorem A]{BGpapertwo} that the families $\pws^*(\Z^2 \setminus \{0\},\circledast_c)$ are, predominantly, in general position. Thus, the conclusion from Corollary \ref{cor:ipstarimpliesmultpsstar} that $A$ is large with respect to each of these multiplications consists of countably many genuinely different conclusions.
\end{example}

Though each of the multiplications described in the previous example are commutative, we do not assume commutativity in Corollary \ref{cor:ipstarimpliesmultpsstar} or in most of the other main results in Sections \ref{sec:posexamplea}, \ref{sec:posexampleb}, and \ref{sec:posexamplec}. 

\subsection{Additive \texorpdfstring{$\ipnaught^*$}{IP0*} sets arising in dynamics}\label{sec:ipzerostarindynamics}

Corollary \ref{cor:ipstarimpliesmultpsstar} is made more interesting by the plethora of examples of additive $\ipnaught^*$ sets arising in dynamics and combinatorics. Additive $\ipnaught^*$ sets first appeared in the work of Furstenberg and Katznelson \cite{furstenbergkatznelsonipszem} on the IP multidimensional Szemer\'{e}di theorem, and they appeared implicitly in \cite[Section 1]{bergelsonleibmanams} in connection with the multidimensional polynomial van der Waerden theorem. It was shown more recently in \cite{BLiprstarcharacterization} that $\ipnaught^*$ sets can be used to characterize nilsystems (algebraic generalizations of group rotations): roughly speaking, a system is a nilsystem if and only if the return times of every point to a neighborhood of itself is an $\ipnaught^*$ set.

While applications relying on the $\ipnaught^*$ formulations of recurrence theorems in dynamics have since appeared (e.g. \cite{blzpaper}) and are present in the current work, they have not been stated explicitly in an $\ipnaught^*$ form in the literature. We begin to remedy this here by deriving the $\ipnaught^*$ versions of these theorems in a succinct way from more general combinatorial theorems. Further improvements on the theorems and techniques in this section will appear in a forthcoming work.

We begin by deriving the the IP Szemer\'edi theorem of Furstenberg and Katznelson in an $\ipnaught^*$ form; the order of the quantifiers as it is stated here is implicit in their proof in \cite[Theorem 10.3]{furstenbergkatznelsonipszem}. The short proof here is possible by appealing to the density Hales-Jewett theorem, Theorem \ref{thm:densityhj}, and the results from Section \ref{sec:banach}.

\begin{theorem}\label{thm:ipszemeredicomb} Let $n \in \N$ and $\denlet > 0$. There exists $r \in \N$ and $\be >0$ for which the following holds. For all commutative semigroups $\semisplus$ and $\semirplus$, all homomorphisms $\varphi_1, \ldots, \varphi_n$ from $S$ to $R$, all left invariant means $\lambda$ on $R$, and all $A \subseteq R$ with $\lambda(A) \geq \denlet$, the set

\begin{align}\label{eqn:setofshiftstobelarge}\left\{ s \in S \ \middle| \ \lambda \big( \big(A - \varphi_1(s) \big) \cap \cdots \cap \big(A - \varphi_n(s) \big) \big) \geq \be \right\}\end{align}
is $\ip_r^*$ in $\semisplus$.
\end{theorem}

\begin{proof}
Let $n \in \N$ and $\denlet > 0$. Let $0 < \eta < \denlet$. Take $r = r(n,\eta)$ from Theorem \ref{thm:densityhj}, and put $\beta = (\denlet - \eta) \big / ((1-\eta)(n+1)^r)$. Let $\semisplus$ and $\semirplus$ be commutative semigroups, $\varphi_1, \ldots, \varphi_n$ be homomorphisms from $S$ to $R$, $\lambda$ be a left invariant mean on $R$, and $A \subseteq R$ with $\lambda(A) \geq \denlet$.

Let $(s_i)_{i=1}^r \subseteq S$. We must show that the set in (\ref{eqn:setofshiftstobelarge}) has non-empty intersection with the finite sums set generated by $(s_i)_{i=1}^r$. Let
\begin{align*}
\pi:[n]^r &\to R\\
\omega &\mapsto \varphi_{\omega_1}(s_1) + \cdots + \varphi_{\omega_r}(s_r).
\end{align*}
Let $F$ be the image of $\pi$, and define the multiset $\multif: F \to \N$ by
\[\multif(f) = \big| \{ \omega \in [n]^r \ \big| \ \pi(\omega) = f \}\big|.\]
Setting $R_\eta(A,\multif) = \{t \in R \ | \ | \multif \cap (A - t) | \geq \eta |\multif| \}$, Theorem \ref{thm:characterizationofdensity} gives that
\[\lambda \big(R_\eta(A,\multif) \big) \geq \frac{\denlet - \eta}{1-\eta}.\]

Abbreviate $R_\eta(A,\multif)$ by $R_\eta$, and let $t \in R_\eta$. Since $\big|\{ \omega \in [n]^r \ | \ t + \pi(\omega) \in A \} \big| \geq \eta n^r$, there exists a combinatorial line $L_t$ in $[n]^r$ such that
\[\{t + \pi(\omega) \ | \ \omega \in L_t \} \subseteq A.\]
Since $\lambda$ is additive and there are no more than $(n+1)^r$ combinatorial lines in $[n]^r$, there exists a combinatorial line $L$ in $[n]^r$ and a subset $R'_\eta \subseteq R_\eta$ such that $\lambda(R'_\eta) \geq \beta$ and for all $t \in R'_\eta$,
\[\{t + \pi(\omega) \ | \ \omega \in L \} \subseteq A.\]

Since $\semirplus$ is commutative and $L$ is a combinatorial line, there exists $u \in R$ and a non-empty $\al \subseteq \{1, \ldots, r\}$ such that for all $t \in R'_\eta$ and all $i \in \{1, \ldots, n\}$, $t + u + \varphi_i (s_\al) \in A$. This implies that
\begin{align}\label{eqn:translationcontainment}R'_\eta + u \subseteq \big(A - \varphi_1(s) \big) \cap \cdots \cap \big(A - \varphi_n(s) \big).\end{align}

Since $R'_\eta + u + (-u) \supseteq R'_\eta$ and $\lambda$ is translation invariant, we have that
\[\lambda \big(R'_\eta + u \big) = \lambda \big(R'_\eta + u + (-u) \big) \geq \lambda \big(R'_\eta\big) \geq \beta.\]
It follows from this and (\ref{eqn:translationcontainment}) that the finite sums set generated by $(s_i)_{i=1}^r$ indeed has non-empty intersection with the set in (\ref{eqn:setofshiftstobelarge}). 
\end{proof}

The following is an immediate corollary of Theorem \ref{thm:ipszemeredicomb} when applied to semirings $\semisring$ in the following way: put $\semirplus = \semisplus$, and take the homomorphisms $\varphi_i$ to be multiplication on the left (or right) by fixed elements of $S$.

\begin{corollary}\label{cor:ipszemfordensityinsemigroups}
Let $n \in \N$ and $\denlet > 0$. There exists $r \in \N$ and $\beta >0$ for which the following holds. Let $\semisring$ be a semiring and $F \subseteq S$ with $|F| \leq n$. For all $A \subseteq S$ with $\dstar_{\semisplus}(A) \geq \denlet$, the set
\[\left\{ s \in S \ \big| \ \dstar_{\semisplus} \left( \big\{ t \in S \ \big| \ t+F s \subseteq A \big\} \right) \geq \beta \right\}\]
is $\ip_r^*$ in $\semisplus$. The same conclusion holds with $Fs$ replaced by $sF$.
\end{corollary}

Besides providing a natural source of $\ipnaught^*$ sets, the uniformity implied by the order of the quantifiers in Theorem \ref{thm:ipszemeredicomb} finds two applications in this work. The $\ip_r^*$ conclusion in this theorem was used in the proof of Theorem \ref{thm:finitarymultsyndgivesaddthickinfields}, and we will use the lower bound on the mean of the sets in (\ref{eqn:setofshiftstobelarge}) to find ``geo-arithmetic'' patterns in sets of positive multiplicative Banach density in Theorem \ref{thm:geopatternsinsetsofdensity}.\\

We proceed now by proving a strengthening of the $\ipnaught^*$ formulation of the multidimensional polynomial van der Waerden theorem originally obtained in \cite[Theorem C and Corollary 1.12]{bergelsonleibmanams}. The uniformity implied by the order of the quantifiers in the theorem as it is stated here is implicit in its original proof and has already found applications, for example in \cite{blzpaper}. The short proof we provide here is possible by appealing to \cite[Theorem 4.4]{bltopmultiplerecurrence}, a polynomial van der Waerden-type theorem for nilpotent groups.

\begin{definition}
Let $\semigplus$ and $\semihplus$ be abelian groups. A map $\varphi: G \to H$ is a \emph{polynomial of degree at most $d$} if the application of any $d+1$ of the discrete difference operators $\Delta_g$, $g \in G$, defined by $(\Delta_g \varphi)(x) \defeq \varphi(x+g) - \varphi(x)$, reduces $\varphi$ to the constant $0_H$ function. The map $\varphi$ has \emph{zero constant term} if $\varphi(0_G) = 0_H$.
\end{definition}

\begin{theorem}\label{thm:bliprstarsemiring}
Let $n,d,k \in \N$. There exists $r \in \N$ for which the following holds. For all abelian groups $\semigplus$ and $\semihplus$, all polynomials $\varphi_1, \ldots, \varphi_n$ from $G \to H$ of degree at most $d$ with zero constant term, and all partitions $H = A_1 \cup \cdots \cup A_k$, there exists $i \in \{1, \ldots, k\}$ for which the set

\begin{align}\label{eqn:setofshiftstobelargevdw}\big\{ g \in G \ \big| \ \big(A_i - \varphi_1(g) \big) \cap \cdots \cap \big(A_i - \varphi_n(g) \big) \neq \emptyset \big\}\end{align}
is $\ip_r^*$ in $\semigplus$.
\end{theorem}

\begin{proof}
Let $n,d,k \in \N$. Let $r \in \N$ be $N$ from \cite[Theorem 4.4]{bltopmultiplerecurrence} where $w$ is our $d$, $r$ is our $k$, and $k$ is our $n$. Let $\semigplus$ and $\semihplus$ be abelian groups. Let $\varphi_1, \ldots, \varphi_n$ be polynomials of degree at most $d$ with zero constant term from $G$ to $H$, and suppose $H = A_1 \cup \cdots \cup A_k$.

Let $(g_i)_{i=1}^r \subseteq G$. We must show that the set in (\ref{eqn:setofshiftstobelargevdw}) has non-empty intersection with the finite sums set generated by $(g_i)_{i=1}^r$. For $i \in \{1, \ldots, n\}$, define $\overline{\varphi_i}: \finitesubsets {[r]} \to H$ by $\overline{\varphi_i}(\al) = \varphi_i(g_\al)$. It is straightforward to check that each $\overline{\varphi_i}$ is a polynomial of degree at most $d$ in the sense described in \cite{bltopmultiplerecurrence}. Now \cite[Theorem 4.4]{bltopmultiplerecurrence} gives $i \in \{1, \ldots, k\}$, a non-empty $\al \in \finitesubsets{[r]}$, and $h \in H$ such that
\[h + \{\overline{\varphi_1}(\al), \ldots, \overline{\varphi_n}(\al)\} \subseteq A_i.\]
This means $g_\al$ is a member of the set in (\ref{eqn:setofshiftstobelargevdw}); in particular, the finite sums set generated by $(g_i)_{i=1}^r$ indeed has non-empty intersection with that set.
\end{proof}

We conclude this section by presenting an application of Theorem \ref{thm:bliprstarsemiring} to polynomial Diophantine approximation. A one-variable variation on the following theorem is described in \cite[Theorem 7.7]{BergelsonSurveytwoten}. We denote by $\|(x_1, \ldots, x_k)\|$ the Euclidean distance from $(x_1, \ldots, x_k)$ to the nearest integer lattice point in $\R^k$.

\begin{theorem}\label{thm:diophantineapplication}
Let $A \subseteq \N^\ell$ be additively $\ipnaught$. For all $n, d \in \N$,
\[\min_{\substack{1 \leq m_1, \ldots, m_\ell \leq M \\ (m_1, \ldots, m_\ell) \in A}} \big\| \big(f_1(m_1,\ldots,m_\ell), \ldots, f_n(m_1,\ldots,m_\ell) \big) \big\| \longrightarrow 0  \ \text{ as } \ M \to \infty\]
uniformly in polynomials $f_1, \ldots, f_n \in \R[x_1,\ldots,x_\ell]$ of degree less than or equal to $d$ with no constant term.
\end{theorem}

\begin{proof}
Let $n, d \in \N$ and $\eps > 0$. Put $k = \lceil 2 / \eps \rceil$, $\semigplus = (\R^\ell,+)$, and $\semihplus = (\R,+)$. For $i \in \{1, \ldots, k\}$, put $A_i = [(i-1) \eps / 2, i \eps / 2) + \Z$, and note that $H = \cup_{i=1}^k A_i$. By Theorem \ref{thm:bliprstarsemiring} (with ``$n$'' equal to $n+1$), there exists an $r \in \N$ such that for $f_0 \defeq 0$ and for all $f_1, \ldots, f_n \in \R[x_1,\ldots,x_\ell]$ of degree less than or equal to $d$ with zero constant term (such maps are polynomials from $G$ to $H$), there exists $i \in \{1,\ldots,k\}$ such that the set
\begin{gather*}
\big\{ g \in \R^{\ell} \ \big| \ \big(A_i - f_0(g) \big) \cap \big(A_i - f_1(g) \big) \cap \cdots \cap \big(A_i - f_n(g) \big) \neq \emptyset \big\} = \\
\big\{ g \in \R^\ell \ \big| \ \exists t \in A_i, \ \forall j \in \{1, \ldots, n\}, \ t + f_j(g) \in A_i \big\}
\end{gather*}
is $\ip_r^*$ in $(\R^\ell,+)$. Note that if $t \in A_i$ and $t+f_1(g), \ldots, t+f_n(g) \in A_i$, then $\big\|\big(f_1(g),\allowbreak \ldots,\allowbreak f_n(g)\big)\big\| < \eps$. It follows that if we define the set
\begin{align*}
G_{(f_1,\ldots,f_n)} =& \ \big\{ g \in \R^\ell \ \big| \ \big\|\big(f_1(g), \ldots, f_n(g)\big)\big\| < \eps \big\},
\end{align*}
then for all $f_1, \ldots, f_n \in \R[x_1,\ldots,x_\ell]$ of degree less than or equal to $d$ with zero constant term, the set $G_{(f_1,\ldots,f_n)}$ is $\ip_r^*$ in $(\R^\ell,+)$. Since $(\N^\ell,+)$ is a sub-semigroup of $(\R^\ell,+)$, the set $G_{(f_1,\ldots,f_n)} \cap \N^\ell$ is $\ip_r^*$ in $(\N^\ell,+)$.

Since $A$ is $\ipnaught$ in $\N^\ell$, there exists an $M \in \N$ such that $A \cap \{1, \ldots, M\}^\ell$ is $\ip_r$.  It follows that for all $f_1, \ldots, f_n \in \R[x_1,\ldots,x_\ell]$ of degree less than or equal to $d$ with no constant term, the set $G_{(f_1,\ldots,f_n)} \cap A \cap \{1, \ldots, M\}^\ell$ is non-empty, implying
\[\min_{\substack{1 \leq m_1, \ldots, m_\ell \leq M \\ (m_1, \ldots, m_\ell) \in A}} \big\| \big(f_1(m_1,\ldots,m_\ell), \ldots, f_n(m_1,\ldots,m_\ell) \big) \big\| < \eps.\]
This proves the desired uniformity.
\end{proof}

This theorem can be contextualized as a multivariable, multidimensional, ineffective version of \cite[Theorem 1]{bakeruniform} along a restricted set of ``denominators.'' The following is a more concrete application demonstrating the interplay between addition inherent in Diophantine approximation on the torus and multiplication used to define the subsets along which the approximation takes place. The proof of the following corollary follows immediately from Theorem \ref{thm:diophantineapplication}, Corollary \ref{cor:omegaismultsyndetic}, and the remarks following Corollary \ref{cor:omegaismultsyndetic}; recall that $\Omega(p_1^{e_1} \cdots p_k^{e_k}) = \sum_{i=1}^k e_i$, where the $p_i$'s are distinct primes.

\begin{corollary}\label{cor:dioappomega}
Let $p \in \R[x]$ be a non-constant polynomial with irrational leading coefficient, and let $I \subseteq [0,1)$ be an interval. For all $d \in \N$,
\[\min_{\substack{1 \leq n \leq N \\ \{ p(\Omega(n)) \} \in I}} \|f(n) \| \longrightarrow 0 \ \text{ as } \  N \to \infty\]

uniformly in polynomials $f \in \R[x]$ of degree less than or equal to $d$ with no constant term.
\end{corollary}


\section{Multiplicative density implies additive combinatorial richness}\label{sec:posexamplec} 

It was proved in \cite[Theorem 3.2]{bmultlarge} that subsets of $\N$ of positive multiplicative upper Banach density are AP-rich. In addition to strengthening this fact and generalizing it to semirings, the following theorem yields a ``relativization'' of the result: \emph{multiplicatively large subsets of additively rich sets are additively rich}.

To make precise what we mean by ``rich,'' we need a slight variation on the class of \crich{} sets defined in Section \ref{sec:defs}. 

\begin{definition}\label{def:combrichuptoe}
Let $\semisplus$ be a commutative semigroup and $\Endset$ be a set of endomorphisms of $\semisplus$. A set $A \subseteq \semisplus$ is \emph{\crich{} up to $\Endset$} if for all $n \in \N$, there exists an $r \in \N$ such that for all $\matM \in S^{r\times n}$, there exists a non-empty $\al \subseteq \{1, \ldots, r\}$, $\Endelt \in \Endset$, and $s \in S$ such that for all $j \in \{1, \ldots, n\}$,
\[s+ \Endelt(\matm_{\al,j}) \in A.\]
We denote by $\crichclass_\Endset \semisplus$ the class of subsets of $\semisplus$ which are \crich{} up to $\Endset$.
\end{definition}

The analogue of Lemma \ref{lem:equivtocombrich} holds for the class $\crichclass_\Endset$, and the same argument used to prove Lemma \ref{crichispr} shows that $\crichclass_\Endset$ is partition regular. Note that the larger the set $\Endset$ of endomorphisms, the larger the class $\crichclass_\Endset$ is and the weaker the notion of \emph{combinatorially rich up to $\Endset$} becomes.

In a semiring $\semisring$, the right dilations $\{\rho_r: x \mapsto xr \ | \ r \in S \}$ form a natural class of endomorphisms of $\semisplus$. It is with respect to this set of endomorphisms that the following theorem says that \emph{sets of positive multiplicative density are additively combinatorially rich.} 

\begin{theorem}\label{thm:multdensitygivesaddpatterns}
Let $\semisring$ be a semiring. Suppose that $\semirtimes$ is a left amenable subsemigroup of $\semistimes$ which is additively large in the following way: $R$ is \crich{} in $\semisplus$. If $A \subseteq R$ is of positive upper Banach density in $\semirtimes$, then $A$ is \crich{} up to $\{\rho_r: x \mapsto xr \ | \ r \in R \}$ in $\semisplus$.
\end{theorem}

\begin{proof}
Let $0 < \eps < \dstar_{(R,\cdot)}(A)$ and $n \in \N$. Let $\ell = r(n,\eps)$ from Theorem \ref{thm:densityhj}, and using that $R$ is \crich{} in $\semisplus$, let $L = r(n,\ell)$ from Lemma \ref{lem:equivtocombrich} \ref{item:multidcombrich}.

Let $M \in S^{L \times n}$. By our choice of $L$ and Lemma \ref{lem:equivtocombrich} \ref{item:multidcombrich}, there exist disjoint, non-empty $\al_1, \ldots, \al_\ell \subseteq \{1, \ldots, L\}$ and $s \in S$ such that for all $w \in [n]^\ell$,
\[\pi(w) \defeq s+ \matm_{\al_1,w_1} + \matm_{\al_2,w_2} + \cdots + \matm_{\al_\ell,w_\ell} \in R.\]

Let $F \subseteq R$ be the image of $\pi: [n]^\ell \to R$, and define the multiset $\multif: F \to \N$ by
\[\multif(f) = \big| \{ w \in [n]^\ell \ | \ \pi(w) = f \}\big|.\]
Since $\dstar_{(R,\cdot)}(A) > \eps$, Theorem \ref{thm:characterizationofdensity} gives that there exists an $r \in R$ such that
\[\eps |\multif| \leq |\multif \cap A r^{-1}| = \big| \{ w \in [n]^\ell \ | \ \pi(w) \in Ar^{-1} \}\big|,\]
where $Ar^{-1}$ is computed in $\semirtimes$. Since $|\multif| = n^\ell$, Theorem \ref{thm:densityhj} gives that there exists a variable word $v: [n] \to [n]^\ell$ such that
\[\{ \pi(v(j)) \ | \ j \in \{1,\ldots,n\} \} \subseteq A r^{-1},\]
where now we can regard $A r^{-1}$ as computed in $\semistimes$. Therefore, there exists $\al \subseteq \{1, \ldots L\}$, $r \in R$, and $s' \in S$ such that for all $j \in \{1, \ldots, n\}$,
\[s'+ \rho_r(\matm_{\al,j}) \in A.\]
This shows that $A$ is \crich{} up to $\{\rho_r \ | \ r \in R \}$ in $\semisplus$.
\end{proof}

A set which is \crich{} up to $\{\rho_r: x \mapsto xr \ | \ r \in R \}$ contains an abundance of additive configurations provided the endomorphisms in $\Endset$ are not degenerate. (Note that, for example, any non-empty subset of $\Z$ is \crich{} up to $\{x \mapsto 0\}$.) In this case, Theorem \ref{thm:multdensitygivesaddpatterns} gives that sets of positive multiplicative density contain arbitrarily large additive patterns, some examples of which are given in Examples \ref{examples:combcubes}. Here are two example corollaries of the theorem.

\begin{corollary}\label{cor:posmultdenimpliesaddpatternsingaussianints}
Let $A$ be a subset of positive multiplicative density in the non-zero Gaussian integers $(\Z[i] \setminus \{0\},\cdot)$. For any $n \in \N$, there exists $z_0 \in \Z[i]$ and $z_1 \in \Z[i] \setminus \{0\}$ such that 
\begin{align}\label{eqn:gaussiancubeconfig}\big\{ z_0 + z_1(k + \ell i) \ \big| \ k, \ell \in \{0, \ldots, n\} \big\} \subseteq A.\end{align}
\end{corollary}

\begin{proof}
The sub-semigroup $\Z[i] \setminus \{0\}$ is combinatorially rich in $(\Z[i],+)$ because it is thick. It follows by Theorem \ref{thm:multdensitygivesaddpatterns} that the set $A$ is \crich{} up to dilations in $(\Z[i],+)$. Using the analogue of Lemma \ref{lem:equivtocombrich} \ref{item:homdef} for combinatorial richness up to dilations just as it is used in Examples \ref{examples:combcubes} \ref{item:thirdexampleofcrstructure}, we find $z_0 \in \Z[i]$ and a dilation $z_1 \in \Z[i] \setminus \{0\}$ for which (\ref{eqn:gaussiancubeconfig}) holds.
\end{proof}

The previous corollary can be extended to the ring of integers in a number field in the following way: any subset of positive multiplicative density of the multiplicative subsemigroup of non-zero integers contains translated, dilated copies of any finite set, in particular the cube $\{0,\ldots, n\}^d$ (under any association between the additive group of the ring of integers and $\Z^d$).

\begin{corollary}\label{cor:posmultdenimpliesaddpatternsinfpx}
Let $A$ be a subset of positive multiplicative density in the semigroup $(\F_p[x] \setminus \{0\},\cdot)$. For any $n \in \N$, there exist $f \in \F_p[x]$ and a vector subspace $V$ of $\F_p[x]$ of dimension $n$ for which $f + V \subseteq A$.
\end{corollary}

We forego the proof of Corollary \ref{cor:posmultdenimpliesaddpatternsinfpx} because much more will be shown in Section \ref{sec:geoarithmeticprogressions}. With a bit more work, we will show that the set $A$ in the previous two corollaries actually contains ``geo-arithmetic'' patterns involving both the additive and multiplicative structures.

It is worth noting that while the set $A$ in Theorem \ref{thm:multdensitygivesaddpatterns} is guaranteed to be additively \crich{} up to right dilations, it need not have additive density. Indeed, examples of sets which are multiplicatively thick but not additively combinatorially rich are given in Section \ref{sec:negexampleb}.

\subsection{Arithmetic progressions in sub-semigroups of \texorpdfstring{$\semintimes$}{(N,.)}}\label{sec:apinsubsemigroups}

The technique used in the proof of Theorem \ref{thm:multdensitygivesaddpatterns} can be used in conjunction with a theorem of Green and Tao to find arbitrarily long arithmetic progressions in sets of positive multiplicative density in certain sub-semigroups of $\semintimes$ of algebraic origin.

Let $L \big / \Q$ be a degree $d \geq 2$ field extension, and denote by $\normlq: L \to \Q$ the norm and $\intringl$ the ring of integers of $L$. Fix an integral basis $(\ell_1, \ldots, \ell_d)$ for $L \big / \Q$ (that is, $(\ell_i)_{i=1}^d \subseteq \intringl$ is a vector space basis for $L \big/ \Q$ and a $\Z$-module basis for $\intringl$). The map
\begin{align*}
\Psi: \Z^d &\to \Z \\
(x_1, \ldots, x_d) &\mapsto \normlq \big( x_1 \ell_1 + \ldots + x_d \ell_d \big)
\end{align*} 
is a degree $d$, homogeneous, integral-coefficient polynomial in $d$ variables. We will call $\Psi$ the \emph{norm form arising from the extension $L \big/ \Q$ and the basis $(\ell_1, \ldots, \ell_d)$}.

\begin{examples} The following are examples of norm forms of degree 2 and 3. That the bases are indeed integral bases of the associated number fields may be found in Chapter 2 of \cite{marcusbook} as Corollary 2 and Exercises 35 and 41.
\begin{enumerate}[label=(\Roman*)]
\item \label{item:examplenormformone} $\Psi(x_1,x_2) = x_1^2 - a x_2^2$ arises from $\Q(\sqrt{a}) \big / \Q$ and the basis $(1, \sqrt{a})$, where $a \in \Z$ is square-free and $a \equiv 2, 3 \pmod 4$.
\item $\Psi(x_1,x_2) = x_1^2 + x_1 x_2 - (a-1) x_2^2 \big / 4$ arises from $\Q(\sqrt{a}) \big / \Q$ and the basis $\big(1, (1+\sqrt{a})/2 \big)$, where $a \in \Z \setminus \{1\}$ is square-free and $a \equiv 1 \pmod 4$.
\item $\Psi(x_1,x_2,x_3) = x_1^3 + a x_2^3 + a^2 x_3^3 - 3a x_1x_2x_3$ arises from $\Q(\sqrt[3]{a}) \big / \Q$ and the basis $(1, \sqrt[3]{a}, (\sqrt[3]{a})^2)$, where $a \in \Z \setminus \{-1,0\}$ is square-free and $a \not\equiv \pm 1 \pmod 9$.
\item $\Psi(x_1,x_2,x_3) = x_1^3-x_2^3+x_3^3 -3 x_1 (x_2^2-x_3^2) +3 x_3 (2 x_1+x_2) (x_1+x_3)$ arises from $\Q(\omega) \big / \Q$ and the basis $(1, \omega, \omega^2)$, where $\omega = \zeta_{7}+\zeta_{7}^{-1}$.
\end{enumerate}
\end{examples}

For a norm form $\Psi$, let
\[\NPsi = \big\{ n \in \N \ \big| \ \exists z \in \Z^d, \ n = |\Psi(z)| \big\}\]
be the set of absolute values of non-zero integers represented by $\Psi$. It is interesting to ask about combinatorial patterns contained in $\NPsi$. Because $\dstar_{\seminplus}(\NPsi) = 0$ (see \cite[Theorem T]{odonireps}), we cannot immediately appeal to Szemer\'edi-type theorems to find combinatorial patterns.  Norm forms do, however, represent a positive upper relative density of the (positive, rational) prime numbers $\mathbb{P}$, as the following lemma shows.

\begin{lemma}\label{lem:normformsrepresentposdensityofp}
Every norm form $\Psi$ represents a positive relative upper density of the prime numbers:
\begin{align} \label{eqn:posreldensity} \limsup_{N \to \infty} \frac{\big| \NPsi \cap \mathbb{P} \cap \{1, \ldots, N \}\big|}{\big| \mathbb{P} \cap \{1, \ldots, N \}\big|} > 0.\end{align}
\end{lemma}

\begin{proof}
Suppose $\Psi$ arises from the extension $L \big/ \Q$. Consider the tower of finite field extensions $\Q \subseteq L \subseteq H \subseteq M$ where $H$ is the Hilbert class field of $L$ (\cite[Theorems 5.18 and 8.10]{coxbook}) and $M$ is a finite Galois extension of $\Q$ containing $H$ (take, for example, a Galois closure of the extension $H \big/ \Q$). It is a consequence of the \v{C}ebotarev Density Theorem that the Dirichlet density of the set of rational primes that split completely in $M$ exists and is positive (indeed, it is equal to $[M:\Q]^{-1}$); see \cite[Pg. 170]{coxbook}.

Let $p \in \mathbb{P}$ be a prime which splits completely in $M$. Since $p$ splits completely in $L$ (see \cite[Pg. 177]{coxbook}), the ideal $p \intringl$ factors into the product $\mathfrak{p}_1 \cdots \mathfrak{p}_d$ of $d = [L:\Q]$ distinct prime ideals of $\intringl$. Because $\mathfrak{p}_1$ lies above $p$, it splits completely in $M$, and hence it splits completely in $H$, too. Since $H$ is the Hilbert class field of $L$, this means that there exists $\ell \in \intringl$ for which $\mathfrak{p}_1 = \ell \intringl$ (see \cite[Corollary 5.25]{coxbook}). Since the norm of the ideal $p \intringl$ is $p^d$, the norm of $\mathfrak{p}_1$ is $p$. Since $\mathfrak{p}_1$ is principal, $|\normlq(\ell)| = p$. It follows by writing $\ell$ in the integral basis from which $\Psi$ arose that $p \in \NPsi$.

Thus the Dirichlet density of the set $\NPsi \cap \mathbb{P}$ exists and is positive. It follows that the the limit supremum in (\ref{eqn:posreldensity}) is positive (see \cite[Chapter III.1]{tenenbaumbook}).
\end{proof}

It is shown in \cite{greentaooriginal} that any subset of $\mathbb{P}$ of positive relative upper density (in the sense of (\ref{eqn:posreldensity})) is AP-rich. Therefore, the set $\NPsi$ is AP-rich.

Since the norm $\normlq$ is multiplicative, $\big( \NPsi, \cdot \big)$ is a sub-semigroup of $\semintimes$. Our aim now is to prove that subsets of the semigroup $\big( \NPsi, \cdot \big)$ which have positive multiplicative upper Banach density are AP-rich. The argument used here is essentially the same one used in the proof of Theorem \ref{thm:multdensitygivesaddpatterns}.

\begin{theorem}\label{thm:relativeszeminsubsemigroupsofn}
Let $\Psi$ be a norm form. If $A \subseteq \NPsi$ satisfies $\dstar_{(\NPsi,\cdot)}(A) > 0$, then $A$ is AP-rich.
\end{theorem}

\begin{proof}
Let $\ell \in \N$ and $0 < \eps < \dstar_{(\NPsi,\cdot)}(A)$. By Szemer\'edi's theorem \cite{szemeredi}, there exists an $L \in \N$ so that any subset of cardinality at least $\eps L$ of an arithmetic progression of length $L$ contains an arithmetic progression of length $\ell$. Let $P \subseteq \NPsi$ be an arithmetic progression of length $L$. Since $\dstar_{(\NPsi,\cdot)}(A) > \eps$, Theorem \ref{thm:characterizationofdensity} gives that there exists an $n \in \NPsi$ for which $|Pn \cap A| > \eps L$. Since $Pn$ is an arithmetic progression of length $L$, it follows by our choice of $L$ that $Pn \cap A$ contains an arithmetic progression of length $\ell$.
\end{proof}

The technique used here can be used to transfer other additive combinatorial patterns in the semigroup $\big( \NPsi, \cdot \big)$ to multiplicatively dense subsets of it. Thus, as results concerning additive patterns in $\NPsi$ improve, so will results for multiplicatively large subsets of it.

\subsection{Geo-arithmetic patterns in semirings}\label{sec:geoarithmeticprogressions}

It was shown in \cite[Theorem 3.15]{bmultlarge} that sets of positive multiplicative density are not only AP-rich, they contain ``geo-arithmetic'' patterns. In what follows, we prove a strengthening and relativization of this result by showing that such geo-arithmetic patterns exist in sets with positive multiplicative density in sub-semigroups of positive additive upper Banach density.

As the approach in \cite[Theorem 3.15]{bmultlarge} uses a correspondence principle and an intersectivity lemma, it is limited to countable semigroups. Here we avoid use of either by appealing to the purely combinatorial version of the IP Szemer\'edi theorem, Theorem \ref{thm:ipszemeredicomb}.

\begin{theorem}\label{thm:geopatternsinsetsofdensity}
Let $\semisring$ be a commutative semiring. Suppose that $\semirtimes$ is a subsemigroup of $\semistimes$ which is additively large in the following ways:
\begin{itemize}
\item $\dstar_{\semisplus}(R) > 0$, and
\item for all $r \in R$, $\dstar_{\semisplus}(rS) > 0$.
\end{itemize}
If $A \subseteq R$ satisfies $\dstar_{\semirtimes}(A) > 0$, then for all $n \in \N$ and all endomorphisms $\varphi_1, \ldots, \varphi_n$ of $R$, there exists an $\eps > 0$ such that for all $F \in \finitesubsets {S}$, there exists $F' \subseteq F$ with $|F'| \geq \eps |F|$, $x \in S$, and $z \in R$ such that $x+F' \subseteq R$ and
\begin{align}\label{eqn:geoarithmeticsets} \big\{ z \varphi_i(x+f) \ \big| \ f \in F', \ i \in \{1,\ldots,n\} \big\} \subseteq A.\end{align}
\end{theorem}

\begin{proof}
Let $\mean$ be a left invariant mean on $\semirtimes$ for which $\mean(A) = \denlet > 0$. Let $n \in \N$, and let $\beta = \beta(\denlet,n) > 0$ be from Theorem \ref{thm:ipszemeredicomb}. For $r \in R$, set
\[A_r = A\varphi_1(r)^{-1}  \cap \cdots \cap A\varphi_n(r)^{-1} .\]

By Theorem \ref{thm:ipszemeredicomb},
\begin{align}\label{eqn:shiftsfromipszemeredi}R' = \big\{ r \in R \ \big | \ \mean \big( A_r \big) \geq \beta \big\}\end{align}
is $\ipnaught^*$ in $\semirtimes$. In particular, $R'$ is syndetic in $\semirtimes$. Since $\dstar_{\semisplus}(R) > 0$ and $R'$ is syndetic in $\semirtimes$, there exists an $r \in R$ for which $r^{-1}R'$ (computed in $\semistimes$) satisfies $\dstar_{\semisplus}(r^{-1}R') > 0$. By assumption, $\dstar_{\semisplus}(rS) > 0$, so Corollary \ref{cor:bigandbigimpliesbig} \ref{item:density} gives that $\dstar_{\semisplus}(R') = \gamma > 0$.

Put $\eps = \be \ga / 2$, and let $F \in \finitesubsets S$. Using Theorem \ref{thm:charofdensityifsfcplus}, there exists an $x \in S$ for which $|F \cap (R'-x)| \geq \gamma |F|$. If $f \in F \cap (R'-x)$, then $x+f \in R'$, so
\[g \defeq \frac 1{|F \cap (R'-x)|} \sum_{f \in F \cap (R'-x)} \one_{A_{x+f}}: R \to [0,1]\]
satisfies $\lambda(g) \geq \beta$. By the positivity of $\lambda$, there exists $z \in R$ such that $g(z) \geq \be / 2$. It follows that the set $F' = \{ f \in F \cap (R'-x) \ | \ z \in A_{x+f} \}$ satisfies $|F'| \geq \eps |F|$ and is such that for all $f \in F'$ and all $i \in \{1,\ldots,n\}$, $z\varphi_i(x+f) \in A$.
\end{proof}

The following corollaries give example applications of this theorem. The finite set $F$ is chosen in both so that any $\eps$-dense subset of $F$ contains the desired combinatorial configuration, and the endomorphisms $\varphi$ are chosen to be exponentiation by a fixed element.

Denote by $\R_+$ the semiring of positive real numbers.

\begin{corollary}\label{cor:multconfigsinpositivereals}
Suppose $A \subseteq \R_+$  has positive multiplicative upper Banach density in $(\R_+,\cdot)$. For all finite $E \subseteq \R_+$ and $\ell \in \N$, there exist $d, x, z \in \R_+$ such that
\begin{align}\label{eqn:multlargeconfig}\big\{ z (x+i d)^e \ \big| \ e \in E, \ i \in  \{1, \ldots, \ell\} \big\} \subseteq A.\end{align}
\end{corollary}

\begin{proof}
Let $E = \{e_1,\ldots,e_n\} \subseteq \R_+$ be finite and $\ell \in \N$. Let $\eps > 0$ be as given in Theorem \ref{thm:geopatternsinsetsofdensity}. By Szemer\'edi's theorem \cite{szemeredi}, there exists $L \in \N$ so that any subset of density $\eps$ of a length-$L$ arithmetic progression contains an arithmetic progression of length $\ell$. Let $F \subseteq \R_+$ be an arithmetic progression of length $L$, and for $i \in \{1,\ldots,n\}$, define $\varphi_i(x) = x^{e_i}$. By Theorem \ref{thm:geopatternsinsetsofdensity}, there exists $F' \subseteq F$ with $|F'| \geq \eps |F|$ and $x, z \in \R_+$ so that (\ref{eqn:geoarithmeticsets}) holds. By Szemer\'edi's theorem, the set $F'$ contains an arithmetic progression of length $\ell$, yielding the configuration in (\ref{eqn:multlargeconfig}).
\end{proof}

Configurations similar to the one in (\ref{eqn:multlargeconfig}) were shown to exist in multiplicatively large subsets of $\seminring$ in \cite[Theorem 3.15]{bmultlarge}. That result can be recovered with the obvious modifications to the proof of Corollary \ref{cor:multconfigsinpositivereals}.

\begin{corollary}\label{cor:geoarithmeticinfpx}
Let $p \in\N$ be prime, and suppose $A \subseteq \F_p[x] \setminus \{0\}$ has positive multiplicative upper Banach density in $(\F_p[x] \setminus \{0\},\cdot)$. For all $d,n \in \N$, there exist $y,z \in \F_p[x]$, $z \neq 0$, and a $d$-dimensional vector subspace $V$ of $\F_p[x]$ such that
\[\bigcup_{v \in V} z \{ y + v, (y + v)^2, \cdots, (y + v)^k \big\} \subseteq A.\]
\end{corollary}

\begin{proof}
The proof is entirely analogous to the proof of Corollary \ref{cor:multconfigsinpositivereals} with Szemer\'edi's theorem replaced by \cite[Theorem 9.10]{furstenbergkatznelsonipszem}, which states that any $\eps$-dense subset of a high-enough dimensional vector space over $\F_p$ contains a $d$-dimensional affine subspace.
\end{proof}

\section{Extremal examples}\label{sec:counterexamples} 

The examples in this section demonstrate the optimality of the main results in Sections \ref{sec:posexamplea}, \ref{sec:posexampleb}, and \ref{sec:posexamplec} by showing that the naive attempt to improve those results within this framework will fail even for the semiring $\N$. The reader is encouraged to refer to Figure \ref{fig:implicationdiagram} in which the dotted arrows indicate the examples below.

\subsection{Multiplicatively \texorpdfstring{$\ipnaughtclass^*$}{IP0*} and \texorpdfstring{$\crichclass^*$}{AR*} but not additively \texorpdfstring{$\syndetic$}{S} or \texorpdfstring{$\thick$}{T}} \label{sec:negexamplea}

The example in Lemma \ref{lem:badexampleone} below demonstrates that we cannot conclude more from multiplicative syndeticity than additive centrality.

\begin{lemma}[{cf. \cite[Theorem 3.4]{berghind-onipsets}}]\label{lem:thicknomultpatterns}
There exists an additively thick subset of $\N$ which does not contain any set of the form $\{kx, ky, kxy\}$ for $k \geq 1$ and $x,y \geq 2$.
\end{lemma}

\begin{proof}
Put $A = \cup_{n \geq 1} \{x_n, x_n+1, \ldots, y_n\}$, where $\max \left(y_n / 2,y_{n-1}^2 \right) < x_n < y_n$ and $y_n - x_n \to \infty$. (For example, put $x_n = 4^{4^n}$ and $y_n = 4^{4^n} + n$.) The set $A$ is additively thick. On the other hand, if $kx, ky \in A$ with $kx \leq ky$ and $ky \in \{x_n \ldots, y_n\}$, then
\[y_n < 2x_n \leq x x_n \leq k x y \leq (ky)^2 \leq y_n^2 < x_{n+1},\]

meaning that $kxy$ is not an element of $A$.
\end{proof}

\begin{lemma}\label{lem:badexampleone}
There exists a subset of $\N$ which is multiplicatively $\ipnaughtclass^*$ and $\crichclass^*$ but not additively syndetic or thick.
\end{lemma}

\begin{proof}
Let $A \subseteq \N$ be the complement of an additively thick subset of $\N$ satisfying the conclusions of Lemma \ref{lem:thicknomultpatterns}.  The set $A$ is not additively syndetic, but it is both multiplicatively $\ipnaughtclass^*$ and $\crichclass^*$.

Let $B = 3\N \cup (3\N + 1)$. The set $B$ is clearly not additively thick. Since $B$ intersects every multiplicative $\ip_2$ set and every multiplicatively \crich{} set, it is multiplicatively $\ipnaughtclass^*$ and $\crichclass^*$.

Consider the set $C = A \cap B$. The set $C$ is neither additively thick nor additively syndetic. Moreover, it follows from the fact that the classes $\ipnaughtclass^*$ and $\crichclass^*$ are filters that $C$ is both multiplicatively $\ipnaughtclass^*$ and $\crichclass^*$.
\end{proof}

In fields, one form of multiplicative largeness does imply additive thickness: by Theorem \ref{thm:multsyndgivesaddthickinfields}, finite index subgroups of a field are additively thick. The same is not true for ``finite index'' sub-semigroups of $\semintimes$. For example, the set $3\N$ is a sub-semigroup of $\semintimes$ of finite index which is not additively thick.

This leads us to ask: what property of a sub-semigroup of $\semintimes$ (stronger than multiplicative syndeticity) would imply additive thickness? While $3\N$ is multiplicatively syndetic, it avoids many infinite sub-semigroups of $\semintimes$ (indeed, $\N \setminus 3\N$ is itself a sub-semigroup of $\semintimes$.) This is not a property shared in the case of fields: note that any finite index subgroup $\Gamma$ of the multiplicative group of a field is such that for all subgroups $H$, the subgroup $\Gamma \cap H$ is of finite index in $H$. A similar property fails to suffice in $\N$: the set $A = \{ n \in \N \ | \ \nu_2(A) \text{ even} \}$ is not additively thick yet it is multiplicatively syndetic in every sub-semigroup of $\semintimes$. 

These ideas are related to another open problem concerning geometric progressions. Lemma \ref{lem:badexampleone} gives that sets which are multiplicatively $\crichclass^*$ are not necessarily additively thick. It is still an open problem to determine whether or not sets which are $\mathcal{GP}^*$, i.e. have non-empty intersection with all sets containing arbitrarily long geometric progressions, are additively thick. The dual form of this problem was studied in \cite{bbhsaddimpliesmult}, and even the following ostensibly easier problem remains open.

\begin{question}\label{question:two}
If $B \subseteq \N$ is additively syndetic, does there exist $x,y \in \N$ such that $\{x, xy^2 \} \subseteq B$.
\end{question}

\subsection{Multiplicatively \texorpdfstring{$\thick$}{T} but not additively \texorpdfstring{$\crichclass$}{CR} or \texorpdfstring{$\ipclass$}{IP}}\label{sec:negexampleb}

A subset of $\N$ is multiplicatively thick if and only if it contains arbitrarily long arithmetic progressions starting from zero.\footnote{To see this, in Definition \ref{def:syndthick}, put $F = \{1,2,\ldots,N\}$.} If those progressions are chosen to be sufficiently separated and divisible (in the sense of the following lemmas), then the resulting multiplicatively thick set need be neither additively combinatorially rich nor additively $\ip$.

In what follows, given $A, B \subseteq \N$ and $n \in \N$, we define $A-B = \{a-b \ | \ a \in A, \ b \in B, \ a > b \}$ and $nA = \{n a \ | \ a \in A\}$.

\begin{lemma}\label{lem:nestedfinitedivisiblesetsnotcr}
Let $(A_i)_{i=1}^\infty \subseteq \finitesubsets \N$, and let $(d_i)_{i=1}^\infty \subseteq \N$ be such that

\begin{align}\label{eqn:growthrequirements} \lim_{i \to \infty} \big(d_{i+1} - \max \big(\cup_{j=1}^{i} d_j A_j \big) - d_i \max A_i\big) = \infty.\end{align}
The set $A \defeq \bigcup_{i=1}^\infty d_i A_i$ is not additively combinatorially rich.
\end{lemma}

\begin{proof}
It was explained in Examples \ref{examples:combcubes} \ref{item:additiveexampleofcrsets} that the difference set of a combinatorially rich set in $\seminplus$ is $\ip_0^*$.  Because $\ip_0^*$ sets are syndetic, to show that $A$ is not additively combinatorially rich, it suffices to show that the set $B \defeq A-A$ is not additively syndetic.

We will show that $\N \setminus B$ contains arbitrarily long intervals. By (\ref{eqn:growthrequirements}), eventually $d_{i+1}$ is greater than $d_i \max A_i$, so the set $B$ differs from the set
\[\bigcup_{i=1}^\infty B_i \defeq \bigcup_{i = 1}^\infty \big( d_i A_i - \bigcup_{j=1}^{i} d_j A_j \big)\]
by at most finitely many elements. By fiat, $B_i \subseteq \N$, so $\min B_i \geq d_i - \max \big(\cup_{j=1}^{i-1} d_j A_j \big)$. Also, $\max B_i \leq d_i \max A_i$. It follows by (\ref{eqn:growthrequirements}) that $\min B_{i+1} - \max B_i  \to \infty$, and this implies that the set $\N \setminus B$ eventually contains the arbitrarily long intervals $\big\{\max B_i  +1, \ldots, \min B_{i+1} - 1\big\}$.
\end{proof}

\begin{lemma}\label{lem:nestedfinitedivisiblesetszerodensity}
Let $(A_i)_{i=1}^\infty \subseteq \finitesubsets \N$, and let $(d_i)_{i=1}^\infty \subseteq \N$ be such that $d_i \to \infty$ as $i \to \infty$ and for all $i \leq j$, $d_i | d_j$. The set $A \defeq \bigcup_{i=1}^\infty d_i A_i$ has $\dstar_{\seminplus}(A) = 0$ and is not an additive $\ip$ set.
\end{lemma}

\begin{proof}
Let $\eps > 0$ and choose $k \in \N$ so that $d_k^{-1} < \eps$. By the divisibility condition, the set $\cup_{i=k}^\infty d_i A_i \subseteq d_k \N$, which has upper Banach density less than $\eps$. Since $\cup_{i=1}^{k-1} d_i A_i$ is finite, it has zero additive upper Banach density. The sub-additivity of $\dstar_{\seminplus}$ gives that $\dstar_{\seminplus}(A) < \eps$. Because $\eps > 0$ was arbitrary, $\dstar_{\seminplus}(A) = 0$.

To show that $A$ is not an additive $\ip$ set, it suffices to show that for all $x \in \N$, $|A \cap (A-x)|< \infty$; that is, that any positive difference of elements of $A$ appears only finitely often. Fix $x \in \N$, and set
\[X = x \left(1+ \max \bigcup \big\{ A_n \ | \ n \in \N, \ d_n \leq x \big\}\right).\]
It suffices to show that for all $n, m \in \N$, if $d_n A_n \cap (d_m A_m - x)$ is non-empty, then $d_n, d_m \leq X$.

Suppose $d_n A_n \cap (d_m A_m - x) \neq \emptyset$. There exists $a_m \in A_m$ and $a_n \in A_n$ such that
\[x = d_m a_m - d_n a_n.\]
If $m \leq n$, then $x = d_m (a_m - a_n d_n / d_m)$. Since $x, d_m \in \N$, $d_m \leq x$ and $a_n d_n / d_m \leq a_m$. It follows that $d_n \leq d_m a_m \leq X$. If $n \leq m$, then $x = d_n (a_m d_m / d_n - a_n)$, whereby $d_n \leq x$, and $d_m \leq a_m d_m = x+ d_na_n \leq X.$
\end{proof}

It follows from Lemmas \ref{lem:nestedfinitedivisiblesetsnotcr} and \ref{lem:nestedfinitedivisiblesetszerodensity} that
\[A = \bigcup_{i=1}^\infty 2^{2^i} \big\{ 1, \ldots, i \big\}\]
is a set which is multiplicatively thick but which is neither additively combinatorially rich nor additively $\ip$. (This is not in contradiction with Theorem \ref{thm:multdensitygivesaddpatterns}, which says that since $A$ is multiplicatively thick, it is additively combinatorially rich \emph{up to dilations by natural numbers}.) Also, note that the set $A$ contains solutions to all homogeneous systems of linear equations with solutions in $\N$. This shows that a set which contains solutions to all homogeneous, partition regular systems of linear equations is not necessarily combinatorially rich. 

We show in the next lemma another extreme example: no multiplicative \folner{} sequence is such that full density along it guarantees positive additive upper Banach density or infinite additive $\ip$ structure. Recall the definition of a \folner{} sequence in Definition \ref{def:sfcandfolner} \ref{item:folnersequence}.

\begin{lemma}\label{lem:multdensitywithoutadddensity}
For any \folner{} sequence $(F_n)_{n=1}^\infty \subseteq \finitesubsets \N$ for $\semintimes$, there exists a set $A \subseteq \N$ satisfying
\[\lim_{n \to \infty} \frac{|A \cap F_n|}{|F_n|} = 1, \ \dstar_{\seminplus}(A) = 0, \text{ and $A$ is not additively } \ip.\]
\end{lemma}

\begin{proof}
Put $N_0 = 1$, and note that for all $i \geq N_0$, $|2^0 F_i \cap F_i| \geq (1-1/2^{0})|F_i|$. Having defined $N_0< \cdots < N_k$, choose $N_{k+1} > N_k$ such that for all $i \geq N_{k+1}$, $|2^{k+1}F_i \cap F_i| \geq (1-1/2^{k+1})|F_i|$.

For $i \in \N$, let $\varphi(i) \in \N$ be such that $N_{\varphi(i)} \leq i < N_{\varphi(i+1)}$. Note that $\varphi(i) \to \infty$ as $i \to \infty$ and that $|2^{\varphi(i)}F_i \cap F_i| \geq (1-1/2^{\varphi(i)})|F_i|$. Put $A = \cup_i 2^{\varphi(i)}F_i$. By Lemma \ref{lem:nestedfinitedivisiblesetszerodensity}, the set $A$ has zero additive upper Banach density and is not additively $\ip$.

To prove that the limit in the conclusion of this lemma is equal to one, let $\eps > 0$ and choose $k$ so that $2^{-k} < \eps$. For all $i \geq N_k$,
\[|A \cap F_i| \geq (1-1/2^{\varphi(i)})|F_i| \geq (1-1/2^{k})|F_i| \geq (1-\eps)|F_i|.\]
The result follows since $\eps$ was arbitrary.
\end{proof}

\subsection{Multiplicatively \texorpdfstring{$\ipclass$}{IP} but not additively \texorpdfstring{$\ipnaughtclass$}{IP0} or \texorpdfstring{$\crichclass$}{AR}}\label{sec:negexamplec}

For any prime $p$, the set $A = \N \setminus p\N$ is multiplicatively $\ip$. Indeed, it is an (infinitely generated) sub-semigroup of $\semintimes$. The set $A$ is not additively $\ipnaught$: given any $p$ elements of $A$, some subsum of them is divisible by $p$. The set $A$, however, is of positive additive upper Banach density, and so it is \crich{} in $\seminplus$.

The following lemma shows that the \emph{finitely generated} sub-semigroups of $\semintimes$ provide the type of example we seek.

\begin{lemma}
Any finitely generated sub-semigroup of $\semintimes$ is multiplicatively $\ip$ but neither additively \crich{} nor $\ipnaught$.
\end{lemma}

\begin{proof}
Let $A$ be a finitely generated sub-semigroup of $\semintimes$. The set $A$ contains $\finiteproducts (a, a, \ldots)$, whereby $A$ is $\ip$ in $\semintimes$. Since $A$ is finitely generated, there exists a prime $p$ for which $A \cap p\N = \emptyset$. This shows that $A$ is not additively $\ip_p$, hence not additively $\ipnaught$.

We will prove that finitely generated sub-semigroups of $\semintimes$ are not \crich{} by showing that they do not contain arbitrarily long arithmetic progressions. This is accomplished by induction on the number of generators. The base case is simple: for all $n \in \N$, the set $\{n^e \ | \ e \in \N\}$ does not contain arithmetic progressions of length greater than two.

Suppose that for all $n_1, \ldots, n_k \in \N$, the set $\{n_1^{e_1} \cdots n_k^{e_k} \ | \ e_1, \ldots, e_k \in \N \}$ does not contain arbitrarily long arithmetic progressions. Let $n_1, \ldots, n_{k+1} \in \N$ and put
\[A = \big\{n_1^{e_1} \cdots n_{k+1}^{e_{k+1}} \ \big| \ e_1, \ldots, e_{k+1} \in \N \big \}.\]
Suppose $A$ has an arithmetic progression of length $L \geq 3 n_1$, and let $P \subseteq A$ be an arithmetic progression of length $L$ of minimal step size $d$. It cannot be that $P \subseteq n_1 \N \cap A$ since otherwise $A \big / n_1$ would contain a length $L$ arithmetic progression of shorter step. Let $a \in A \setminus n_1 \N$, and note that
\[(a + n_1 d \Z) \cap A \subseteq A \setminus n_1\N \subseteq \big\{n_2^{e_2} \cdots n_{k+1}^{e_{k+1}} \ \big| \ e_2, \ldots, e_{k+1} \in \N \big \}.\]
This means an arithmetic progression of length at least $L \big / n_1$ is contained in a finitely generated sub-semigroup of $\semintimes$ with $k$ generators. By the induction hypothesis, $L \big / n_1$ is bounded from above by a function of $k$; in particular, $A$ does not contain arbitrarily long arithmetic progressions, completing the inductive step.
\end{proof}

\subsection{Multiplicatively \texorpdfstring{$\density$}{D} but not additively \texorpdfstring{$\ipnaughtclass$}{IP0}}\label{sec:negexampled}

Consider the action $(T_s)_{s \in \N}$ of the semigroup $\semintimes$ on the torus $\R \big/ \Z$ given by $T_s x = sx \pmod 1$. This action is ergodic and can be used in conjunction with von Neumann's mean ergodic theorem to generate examples of sets with high multiplicative density which are not additively $\ipnaught$.

\begin{lemma}
Let $\eps > 0$. There exists a set $A \subseteq \N$ with $\dstar_{\semintimes}(A) \geq 1-\eps$ which is not additively $\ipnaught$.
\end{lemma}

\begin{proof}
Fix a \folner{} sequence $(F_n)_{n=1}^\infty \subseteq \finitesubsets \N$ for $\semintimes$. Since the action $(T_s)_{s \in \N}$ of the semigroup $\semintimes$ on the torus $\R \big/ \Z$ described above is ergodic, the mean ergodic theorem gives the $L^2$-convergence
\[ \lim_{n \to \infty}\frac 1{|F_n|} \sum_{s \in F_n} \one_{(\eps/2,1-\eps/2)}\circ T_s = \int_{\R / \Z} \one_{(\eps/2,1-\eps/2)} = 1-\eps.\]
Any $L^2$-convergent sequence has a pointwise a.e. convergent subsequence. By passing to such a subsequence, for Lebesgue a.e. $x \in \R \big/ \Z$,
\[ \lim_{k \to \infty}\frac 1{|F_{n_k}|} \sum_{s \in F_{n_k}} \one_{(\eps/2,1-\eps/2)}(T_s x) = \lim_{k \to \infty}\frac {\big| \{s \in \N \ | \ \| sx\| > \eps/2 \} \cap F_{n_k} \big|}{|F_{n_k}|} = 1-\eps.\]
Since $(F_{n_k})_{k=1}^\infty \subseteq \finitesubsets \N$ is a \folner{} sequence for $\semintimes$, for any point $x \in \R \big/ \Z$ for which there is convergence, the set
\[A = \big\{n \in \N \ | \ \| nx\| > \eps/2 \big\}\]
is such that $\dstar_{\semintimes}(A) \geq 1-\eps$. On the other hand, Dirichlet's pigeonhole principle argument shows that the set $\N \setminus A$ is $\ip_r^*$ (actually, $\Delta_r^*$) in $\seminplus$ for any $r > 2 \eps^{-1}$. Therefore, $A$ is not $\ipnaught$ in $\seminplus$.
\end{proof}

This lemma is optimal in the sense that if $\dstar_{\semintimes}(A) = 1$, then $A$ is multiplicatively thick and hence, by Theorem \ref{thm:multpwsimpliesaddipnaught}, additively $\ipnaught$.  In fact, this lemma produces sets $A$ which have multiplicative upper Banach density arbitrarily close to 1 but are not multiplicatively piecewise syndetic (by Corollary \ref{cor:ipstarimpliesmultpsstar}, since $\N \setminus A$ is additively $\ipnaught^*$). Such sets can be seen as weak multiplicative analogues of sets $A$ described by Ernst Straus \cite[Theorem 2.20]{bbhsaddimpliesmult} with $\lim_{n \to \infty} |A \cap \{1, \ldots, n\}| \big/ n$ existing and being arbitrarily close to 1 yet having no shifts being additively IP. Such sets were shown to exist in countably infinite, amenable groups -- and hence, in particular, in $\semintimes$ -- in \cite[Theorem 6.4]{bcrzpaper}.

\subsection{Multiplicatively \texorpdfstring{$\crichclass$}{AR} but not additively \texorpdfstring{$\crichclass$}{AR}}\label{sec:negexamplee}

Theorem \ref{thm:multdensitygivesaddpatterns} gives that sets of positive multiplicative upper Banach density contain a variety of additive combinatorial patterns. We show in the following lemma that the same is not true for sets that are only assumed to be multiplicatively \crich{}.

We begin with two auxiliary lemmas. In what follows, given $x_1, \ldots, x_r \in \N$ and non-empty $\al \subseteq \{1, \ldots, r\}$, we denote by $x_\al$ the product $\prod_{i \in \al} x_i$. For a matrix $\matM = (\matm_{i,j}) \in \N^{r\times n}$, $\matm_{\al,j}$ denotes the product $\prod_{i\in \al} \matm_{i,j}$.

\begin{lemma}\label{lem:prodofapsisnotap}
Let $r \in \N$, and for each $i \in \{1, \ldots, r\}$, let $x_i, y_i, z_i \in \N$ satisfy $x_i +z_i = 2y_i$ and $x_i < y_i < z_i$. For all $\al \subseteq \{1, \ldots, r\}$ with $|\al| \geq 2$, $x_\al + z_\al \neq 2 y_\al$.
\end{lemma}

\begin{proof}
It is simple to verify by induction the following identity: for all $k \in \N$ and $a_1, \ldots, a_k, b_1, \ldots, b_k \in \R$,

\[b_1 \cdots b_k - a_1 \cdots a_k = \sum_{i=1}^k a_1 \cdots a_{i-1} \cdot (b_i - a_i) \cdot b_{i+1} \cdots b_k.\]
Let $\al = \{\al_1, \ldots, \al_k \} \subseteq \{1, \ldots, r\}$ with $|\al| = k \geq 2$. By the previous identity
\begin{align*}
y_\al - x_\al &= \sum_{i=1}^k x_{\al_1} \cdots x_{\al_{i-1}} \cdot d_{\al_i} \cdot y_{\al_{i+1}} \cdots y_{\al_{k}} \\
&< \sum_{i=1}^k y_{\al_1} \cdots y_{\al_{i-1}} \cdot d_{\al_i} \cdot z_{\al_{i+1}} \cdots z_{\al_{k}} = z_\al - y_\al,
\end{align*}
where $d_i = y_i-x_i = z_i - y_i > 0$. This shows that $x_\al + z_\al \neq 2 y_\al$.
\end{proof}

\begin{lemma}\label{lem:mulmatrixwithoutaps}
For all $n \in \N$, there exists $r \in \N$ such that for all $\matM \in \N^{r\times n}$, there exists a non-empty $\al \subseteq \{1,\ldots,r\}$ such that the set $\{\matm_{\al,j}\}_{j=1}^n$ does not contain a non-constant three term arithmetic progression.
\end{lemma}

\begin{proof}
Let $n \in \N$. For $i,j,k \in \{1, \ldots, n\}$ and a row vector $(x_1, \ldots, x_n)$ of indeterminates, consider the system
\[E_{i,j,k}(x_1, \ldots, x_n): \qquad \begin{cases} x_i + x_k = 2x_j & \\ x_i < x_j < x_k & \end{cases}.\]

By Lemma \ref{lem:prodofapsisnotap}, if $M_1,\ldots,M_r$ are row vectors of real numbers such that for all $\ell \in \{1, \ldots, r\}$, $E_{i,j,k}(M_\ell)$ holds, then for all $\al \subseteq \{1, \ldots, r\}$ with $|\al| \geq 2$, $E_{i,j,k}(M_{\al,1},\allowbreak \ldots,\allowbreak M_{\al,n})$ does not hold.

Let $r \in \N$ satisfy $r > 2^{n^3} n^3!$, and let $\matM \in \N^{r\times n}$. Given a non-empty $\al \subseteq \{1, \ldots, r\}$, the set $\{\matm_{\al,j}\}_{j=1}^n$ contains no non-constant three term arithmetic progressions if and only if for all $i,j,k \in \{1, \ldots, n\}$, $E_{i,j,k}(M_{\al,1},\ldots,M_{\al,n})$ does not hold.

Such a set $\al \subseteq \{1, \ldots, r\}$ can be found by induction in the following way. To each $\ell \in \{1, \ldots, r\}$, associate a triple $(i,j,k) \in \{1, \ldots, n\}^3$ such that $E_{i,j,k}(M_{\ell,1},\ldots, \allowbreak M_{\ell,n})$ holds; if no such association is possible for $\ell$, we are done by putting $\al = \{\ell\}$. By the pigeon-hole principle, there exist indices $\ell_1 < \ldots < \ell_{2r'}$, $r' = r \big/(2n^3)$, each associated to the same triple $(i_1,j_1,k_1)$. Consider the matrix $\matM' = (M_{i,j}') \in \N^{r' \times n}$ given by $M'_{i,j} = M_{\ell_i,j} M_{\ell_{i+r'},j}$. It follows by the previous remarks that for all non-empty $\al \subseteq \{1, \ldots, r'\}$, $E_{i_1,j_1,k_1}(M'_{\al,1},\allowbreak \ldots,\allowbreak M'_{\al,n})$ does not hold.

We repeat this procedure with the matrix $\matM'$: to each $\ell \in \{1, \ldots, r'\}$, associate a triple $(i,j,k) \in \{1, \ldots, n\}^3$ such that $E_{i,j,k}(M_{\ell,1}, \ldots, \allowbreak M_{\ell,n})$ holds. By the first step, no index $\ell$ is associated to $(i_1,j_1,k_1)$. If no such association is possible for $\ell$, we are done by putting $\al = \{\ell,\ell+r'\}$. Using the pigeonhole principle again to find a new triple $(i_2,j_2,k_2)$, we may pass in the same way as before to a matrix $\matM''$ with the property that for all non-empty $\al \subseteq \{1, \ldots, r''\}$, neither $E_{i_1,j_1,k_1}(M''_{\al,1},\allowbreak \ldots,\allowbreak M''_{\al,n})$ nor $E_{i_2,j_2,k_2}(M''_{\al,1},\allowbreak \ldots,\allowbreak M''_{\al,n})$ holds.

Because there are at most $n^3$ triples to be excluded, this process will terminate in at most $n^3$ steps, resulting in the desired subset $\al \subseteq \{1, \ldots, r\}$.
\end{proof}

\begin{lemma}
There exists a subset of $\N$ which is multiplicatively \crich{} but which has no non-constant three term arithmetic progressions; in particular, it is not additively \crich{}.
\end{lemma}

\begin{proof}
Let $r: \N \to \N$, $n \mapsto r(n)$, be the function from Lemma \ref{lem:mulmatrixwithoutaps}. To each tuple $(n,\matM)$ with $n \in \N$ and $\matM \in \N^{r \times n}$, associate $x_{(n,\matM)} \in \N$ sufficiently large (to be described momentarily) and a non-empty $\al_{(n,\matM)} \subseteq \{1, \ldots, r\}$ using Lemma \ref{lem:mulmatrixwithoutaps} such that $A_{(n,\matM)} = \{\matm_{\al_{(n,\matM)},j}\}_{j=1}^n$ has no non-constant three term arithmetic progressions. The set
\begin{align}\label{eqn:defofmultcrichbutnoaps}A = \bigcup_{(n,\matM)} x_{(n,\matM)} A_{(n,\matM)}\end{align}
is multiplicatively \crich{} by construction. None of the constituent subsets $A_{(n,\matM)}$ contain non-constant three term arithmetic progressions. Since the union in (\ref{eqn:defofmultcrichbutnoaps}) is countable, it is easy to see that the $x_{(n,\matM)}$'s may be chosen in such a way that there are no three term arithmetic progressions with elements in different $A_{(n,\matM)}$'s. Therefore, the set $A$ contains no non-constant three term arithmetic progressions.
\end{proof}

\clearpage

\section{Index}\label{sec:index}

\vskip.5cm

\begin{tabularx}{\textwidth}{lll}\label{indextable}
Symbol & Location & Description \\
\hline \\

$\{ \cdot \}$, $\| \cdot \|$ & Ex. \ref{examples:homomorphisms} & Fractional part, distance to $\Z$ or $\Z^d$\\
$\one_A$ & Def. \ref{def:defofamenable} & Indicator function of the set $A$\\
$\star$ & Sec. \ref{sec:crsetsandhjtheorems} & Variable in a variable word\\

$A - n$, $A/n$, $Ax^{-1}$ & Sec. \ref{sec:defs} & Set operations\\
$A-A$ & Sec. \ref{sec:negexampleb} & Difference set\\
AP-, GP-rich & Sec. \ref{sec:intro} & Contains arbitrarily long arithmetic, geometric \\
& & progressions \\ 
$\beta S$ & Lem. \ref{lem:bigandbigimpliesbig} & Stone-\v{C}ech compactification of $S$\\
$\crichclass_\Endset$ & Def. \ref{def:combrichuptoe} & Class of sets combinatorially rich up to $\Endset$\\
$\overline{d}_{(F_N)}$, $\underline{d}_{(F_N)}$ & Lem. \ref{lem:multsyndimplieslowerdensity} & Upper, lower asymptotic density with respect to\\
& & the \folner{} sequence $(F_N)_N$\\
$d_{\semisplus}^*$ & Def. \ref{def:defofamenable} & Upper Banach density in $\semisplus$\\
$E(\beta S)$ & Lem. \ref{lem:bigandbigimpliesbig} & Set of idempotents in $(\beta S,\cdot)$ \\
$\multif$ & Def. \ref{def:multiset} & Multiset\\
$\mathbb{F}$, $\mathbb{F}_p$ & Sec. \ref{sec:syndeticinfields} & Finite field, and finite field with $p$ elements\\
$(F_N)_{N \in \N}$ & Def. \ref{def:sfcandfolner} & \folner{} sequence\\
$\text{FS}$, $\text{FP}$ & Def. \ref{def:ipstructure} & Finite sums and finite products\\

$\ip_r^*$, $\ip_0^*$ set & Def. \ref{def:ipstructure} & Set non-trivially intersecting all $\ip_r$, $\ip_0$ sets \\
$K(\beta S)$ & Lem. \ref{lem:bigandbigimpliesbig} & Smallest two-sided ideal in $(\beta S,\cdot)$\\
$\lambda$ & Def. \ref{def:defofamenable} & Translation invariant mean\\
$L \big / \Q$ & Sec. \ref{sec:apinsubsemigroups} & Finite field extension over $\Q$\\
$\matM$, $M_{i, j}$, $M_{\alpha, j}$ & Def. \ref{def:combrich} & Matrix, entry, and entry sum (or product)\\
$M_d(S)$ & Ex. \ref{ex:beginningexamplessectionfour} & $d$-by-$d$ matrices with entries from $S$\\
$\normlq$ &  Sec. \ref{sec:apinsubsemigroups} & Norm of the extension $L \big / \Q$\\
$\N_\Psi$ & Sec. \ref{sec:apinsubsemigroups} & Positive integers represented by norm form $\Psi$\\
$\nu_p$ & Ex. \ref{examples:homomorphisms} & $p$-adic valuation, $\nu_p(p^{e} p_2^{e_2} \cdots p_k^{e_k}) \defeq e$ \\
$\intringl$ & Sec. \ref{sec:apinsubsemigroups} & Ring of integers of $L$\\
$\subsets{S}$, $\finitesubsets{ S}$ & Sec. \ref{sec:defs} & Subsets and finite subsets of a set $S$\\
$\semisplus$, $\semirplus$ & Sec. \ref{sec:defs}  & Commutative semigroups \\
$\semistimes$,  $\semirtimes$ & Sec. \ref{sec:defs} & Semigroups (no commutativity assumed) \\
$\semisring$ & Def. \ref{def:semiring} & Semiring\\
$\syndetic$, $\ipclass$, ... & Sec. \ref{sec:defs} & Classes of largeness\\
$\syndetic\semisplus$ & Sec. \ref{sec:defs} & Class of syndetic subsets of $\semisplus$\\
$\syndetic^*$, $\ipclass^*$, ... & Def. \ref{def:dualclasses} & Dual classes \\
\sfc{}, \sfcplus{} & Sec. \ref{sec:banach} & Strong \folner{} conditions \\
$S^{r \times n}$ & Def. \ref{def:combrich} & $r$-by-$n$ matrices with entries from $S$\\
$\Omega$ & Ex. \ref{examples:homomorphisms} & $\Omega(p_1^{e_1}\cdots p_k^{e_k}) \defeq \sum_i e_i$ \\
$\arbclass$ & Sec. \ref{sec:defs} & Variable class of largeness\\
$x_\al$ & Def. \ref{def:ipstructure} & $x_\al \defeq \sum_{i \in \al} x_i$
\end{tabularx}

\clearpage

\bibliographystyle{alphanum}
\bibliography{pwsbib}

\end{document}